\documentclass[reqno,11pt]{amsart}
\usepackage[top=2.0cm,bottom=2.0cm,left=3cm,right=3cm]{geometry}
\usepackage[titletoc]{appendix}

\usepackage{amsthm,amsmath,amssymb,dsfont}
\usepackage{mathrsfs,amsfonts,functan,extarrows,mathtools}

\usepackage{marginnote}
\usepackage{xcolor}
\usepackage{stmaryrd}
\usepackage{esint}
\usepackage{graphicx}
\usepackage{bm}

\allowdisplaybreaks[4]
\theoremstyle{plain}
\newtheorem{thm}{Theorem}[section]

\newtheorem{lem}{Lemma}[section]
\newtheorem{prop}{Proposition}[section]
\newtheorem{rem}{Remark}[section]

\numberwithin{equation}{section}

\newcommand{\dv}{{\rm div}}

\newcommand{\p}{{\partial}}
\newcommand{\pa}{{\partial^m}}

\newcommand{\e}{\epsilon}
\renewcommand{\u}{{\rm u}^{\epsilon}}
\renewcommand{\d}{{\rm d}^{\epsilon}}
\newcommand{\cd}{\dot{\rm d}^{\epsilon}}
\newcommand{\pe}{\phi^{\epsilon}}
\newcommand{\ro}{\rho^{\epsilon}}
\newcommand{\R}{\mathbb{R}}
\newcommand{\eps}{\epsilon}

\newcommand{\na}{\nabla}
\newcommand{\De}{\Delta}
\newcommand{\ct}{\cdot}
\newcommand{\la}{\lambda}
\newcommand{\ka}{\kappa}
\newcommand{\lt}{\langle}
\newcommand{\rt}{\rangle}
\newcommand{\A}{{\rm A}}
\newcommand {\B}{{\rm B}}

\newcommand{\loc}{\rm loc}

\newcommand{\I}{{\mathrm I}}

\newcommand {\dd}{{\rm d}}

\begin{document}

\title[Incompressible limit for parabolic-hyperbolic liquid crystal models]
{Incompressible limit of the Ericksen-Leslie parabolic-hyperbolic liquid crystal model}

\author[Liang Guo]{Liang Guo}
\address[Liang Guo]{School of Mathematics and Statistics, Henan University, Kaifeng 475004, P.R.
China}
\email{guoliang0152@henu.edu.cn}

\author[Ning Jiang]{Ning Jiang}
\address[Ning Jiang]{School of Mathematics and Statistics, Wuhan University, Wuhan
 430072, P. R. China}
\email{njiang@whu.edu.cn}

\author[Fucai Li]{Fucai Li$ ^*$}
\address[Fucai Li]{Department of Mathematics, Nanjing University, Nanjing
 210093, P. R. China}
\email{fli@nju.edu.cn}
\thanks{$^*$ \!\! Corresponding author}

\author[Yi-Long Luo]{Yi-Long Luo}
\address[Yi-Long Luo]{School of Mathematics, South China University of Technology, Guangzhou, 510641, P. R. China}
\email{luoylmath@scut.edu.cn}

\author[Shaojun Tang]{Shaojun Tang}
\address[Shaojun Tang]
{\newline Department of Mathematics, Wuhan University of Technology, Wuhan, 430063, P. R. China}
\email{shaojun.tang@whut.edu.cn}

\begin{abstract}

Ericksen  and Leslie proposed a hydrodynamic model for liquid crystals in the format of conservation laws in the 1960s. Their original model includes inertial and compressibility effects, which makes the model a coupled parabolic-hyperbolic system.
In this paper we build up the connection between  the  compressible  and incompressible parabolic-hyperbolic liquid crystal model in the framework of classical solutions.

We first derive the scaled Ericksen-Leslie system with dimensionless numbers, including Mach, Reynolds, and Ericksen numbers. In particular, we introduce the so-called inertial constant $\chi$ which characterizes the inertial effect of the liquid crystal molecular.
Next, we establish the energy estimates uniform in the Mach number $\eps$ for both the compressible system and its time-derivative system with small data. Then, we pass to the limit $\eps \rightarrow 0$ in the compressible system, so that we establish the global classical solution of the incompressible system by the compactness arguments. Moreover, we also obtain the convergence rates associated with $L^2$-norm in the case of well-prepared initial data. This is the first result on the incompressible limit of   the   compressible parabolic-hyperbolic liquid crystal model and confirms the relations of different parabolic-hyperbolic liquid crystal model rigorously.
\end{abstract}

\keywords{Compressible Ericksen-Leslie parabolic-hyperbolic liquid crystal model;  Incompressible limit;  Uniform estimate; Convergence rate}

\subjclass[2010]{76A15, 35A01, 35B40}

\maketitle

\section{Introduction}\label{s1}

\subsection{Ericksen-Leslie parabolic-hyperbolic liquid crystal model in compressible flow}

Liquid crystals are an orientational well-ordered fluid and an intermediate state which exist in the conversion process between solid and liquid. They thereby have both optical properties of solid crystals and fluidity of liquids. The hydrodynamic theory of liquid crystals was established by Ericksen \cite{Eri1,Eri2} and Leslie \cite{Les} in the 1960s (see also Section 5.1 of \cite{Lin-Liu-2001}) through conservation laws for liquid crystals. In this paper, we choose some appropriate parameters of the model (from \cite{Les}) for simplicity. We consider the following Ericksen-Leslie parabolic-hyperbolic liquid crystal model in compressible flow over $(t,x) \in \R^+ \times \R^n (n = 2, 3)$:
\begin{align}\label{sys1}
\left\{
 \begin{array}{ll}
  \partial_{t} {\rho} + {\rm div} ( {\rho} {\rm u} ) = 0 \,,\\[2mm]
  \partial_{t} ( {\rho} {\rm u} ) + { \rm div} ( {\rho} {\rm u} \otimes {\rm u} ) + \nabla {p} = {\rm div} ( \Sigma_{1} + \Sigma_{2} + \Sigma_{3} ) \,,\\[2mm]
  {\rho} {\ddot{{\rm d}}} = {\kappa} \Delta {{\rm d}} + {\Gamma} {\rm d} + {\lambda}_{1} ( {\dot{{\rm d}}} + {\B} {\rm d} ) + {\lambda}_{2} {\A} {\rm d} \,,
 \end{array}\right.
\end{align}
where ${\rho} (t,x) \geq 0$ denotes the fluid density, $ {\rm u} (t,x) \in \mathbb{R}^n $  the bulk velocity, and $ {\rm d} (t,x) \in \mathbb{R}^n $   the direction field of the liquid molecules with the geometric constraint $|{\rm d}|=l$, respectively. Here $l$ is the averaged length of the rod-like nematic liquid crystal molecular. The pressure ${p} (t,x) \in \R$ satisfies ${p} ({\rho})=\tilde{a} {\rho}^\gamma$ with the constants $\tilde{a}>0$ and $\gamma>1$. Here, ${\dot{{\rm d}}}$ and ${\ddot{{\rm d}}}$ are the first-order and secondary material derivative of ${\rm d}$, respectively, i.e.,
\begin{align*}
    {\dot{{\rm d}}} = \p_t {\rm d} + {\rm u} \ct \na {\rm d} \,, \quad {\ddot{{\rm d}}} = \p_t {\dot{{\rm d}}} + {\rm u} \ct \na {\dot{{\rm d}}} \,.
\end{align*}
The notations
\begin{align*}
    {\A} = \tfrac{1}{2} ( \nabla {\rm u} + \nabla {\rm u}^\top ) \,, \quad {\B} = \tfrac{1}{2} ( \nabla {\rm u} - \nabla {\rm u}^\top )
\end{align*}
represent the rate of strain tensor and skew-symmetric part of the strain rate, respectively.
The notations $\Sigma_i\,(i=1,2,3)$ are defined as follows:
\begin{equation}\label{Tensors-Sigma}
  \begin{aligned}
    & \Sigma_1 := \tfrac{1}{2} {\mu}_4 ( \nabla {\rm u} + \nabla {\rm u}^\top ) + {\xi} \dv \, {\rm u} \, \I, \\
    & \Sigma_2 := \tfrac{1}{2} {\kappa} | \nabla {\rm d} |^2 \, \I - {\kappa} \nabla {\rm d} \odot \nabla {\rm d} \,, \\
    & \Sigma_3 := \tilde{\sigma}_{\bm{{\mu}}} ({\rm u}, {\rm d}, \dot{{\rm d}}) \,,
  \end{aligned}
\end{equation}
where the matrix $\na {\rm d} \odot  \na {\rm d}$ is of the form $(\na {\rm d} \odot \na {\rm d})_{ij} = \p_i {\rm d} \cdot \p_j {\rm d}$ and $\tilde{\sigma}_{\bm{{\mu}}} ({\rm u}, {\rm d}, \dot{{\rm d}})$ is the stress tensor with the forms of entries:
\begin{equation}\label{Tensor-Stress}
  \begin{aligned}
    (\tilde{\sigma}_{\bm{{\mu}}} ({\rm u}, {\rm d}, {\dot{{\rm d}}}))_{ij} = & {\mu}_{1} {\rm d}_{k} {\rm d}_{l} {\A}_{kl} {\rm d}_{i} {\rm d}_{j} + {\mu}_{2} {\rm d}_{i} ( {\dot{{\rm d}}}_{j} + {\B}_{jk} {\rm d}_{k} ) + {\mu}_{3} {\rm d}_{j} ( {\dot{{\rm d}}}_{i} + {\B}_{ik} {\rm d}_{k} )  \\
    & + {\mu}_{5} {\rm d}_{i} {\A}_{jk} {\rm d}_{k} + {\mu}_{6} {\rm d}_{j} {\A}_{ik} {\rm d}_{k} \,.
  \end{aligned}
\end{equation}
Here $\bm{\mu} = (\mu_1, \mu_2, \mu_3, \mu_5, \mu_6)$. These parameters $\bm{\mu}$ are called Leslie coefficients, $\mu_4$ and $\xi$ are the usual viscosities, and $\ka$ is the Frank constant, which measure the elasticity of liquid crystal molecular. Usually, the following relations are frequently introduced in the literature:
\begin{equation}\label{Parodi-1}
  \begin{aligned}
    \la_1 = \mu_2 - \mu_3 \,, \quad \la_2 = \mu_5 - \mu_6 \,, \quad \mu_2 + \mu_3 = \mu_6 - \mu_5 \,.
  \end{aligned}
\end{equation}
The first two relations are necessary conditions in order to satisfy the equation of motion identically, while the third relation is called {\em Parodi's relation}, which is derived from Onsager reciprocal relations expressing the equality of certain relations between flows and forces in thermodynamic systems out of equilibrium. For the analytical effects of Parodi's relation, see \cite{Wu-Xu-Liu-ARMA2013}. Moreover, the Lagrangian multiplier $\Gamma$ has the following form:
\begin{align}\label{Ga}
    \Gamma\equiv \Gamma ( \rho , {\rm u} , {\rm d} , {\dot{\rm d}} ) = \tfrac{1}{{l}^2} \left(- \rho | {\dot{\rm d}} |^2 + \kappa | \nabla {\rm d} |^2 - \lambda_2 {\rm d}^\top \A {\rm d}\right) \,.
\end{align}
For the detailed rewriting from Leslie's paper \cite{Les} to the system \eqref{sys1}, see \cite{JLT}. Note that in \cite{JLT}, the coefficient ${\kappa}$ appeared in $\Sigma_2$ of \eqref{Tensors-Sigma} was taken as $1$. But in this paper, we keep it as $\kappa$ since the scaling of ${\kappa}$ is important. Another important difference with previous literature needed to be pointed out is that in the Lagrangian multiplier $\Gamma$ \eqref{Ga}, we add a factor $\frac{1}{l^2}$, which comes from the geometric constraint $|{\rm d}|=l$. The reason is that the system \eqref{sys1} is dimensional, $|{\rm d}|$ has to have the unit length to make the units of the system \eqref{sys1} match, while in the previous literature, $|{\rm d}|$ was set as $1$, which express the unit  in an implicit way.

\subsection{Nondimensionization}
Now, since the main concern of this paper is about the asymptotic behavior of system \eqref{sys1}, we have to rewrite system \eqref{sys1} into its dimensionless form. First we set the units for the different physical quantities. Let $L_*, T_*$ and $U_*$ be the units for (macroscopic) length, time and bulk velocity, respectively, where $U_*= L_*/T_*$. Let $\rho_*$, $c_*$ and $\mu_*$ be the units of density, sound speed and viscosity, and let $\kappa_* = K/l^2$ be the unit of Frank constant, where $K>0$ is the so-called elastic constant.

To nondimensionalize system \eqref{sys1}, we set
\begin{equation}\label{units}
  \begin{aligned}
    &x=L_* \hat{x}\,, \quad t=T_* \hat{t}\,, \quad {\rm u}=U_* \hat{\rm u}\,,\quad \rho= \rho_* \hat{\rho}\,,\quad p=\rho_* c^2_* \hat{p}\,,
    \quad {\rm d}=l \hat{{\rm d}}\,,\\
    & \kappa=\kappa_* \hat{\kappa}\,,\quad \mu_4=\mu_* \hat{\mu}_4\,,\quad \xi=\mu_* \hat{\xi} \,,\quad \mu_1 = \tfrac{\mu_*}{l^4} \hat{\mu}_1\,,\quad \mu_i = \tfrac{\mu_*}{l^2} \hat{\mu}_i\quad\mbox{for}\quad\! i=2,3,5,6\,.
  \end{aligned}
\end{equation}
We also define the following dimensionless constants:
\begin{equation}\label{dimentionless}
  \begin{aligned}
  {\rm Ma}=\tfrac{U_*}{c_*}\,,\quad {\rm Re}=\tfrac{\rho_*L_* U_*}{\mu_*}\,, \quad {\rm Er}=\tfrac{\mu_*L_* U_*}{\kappa_* l^2}\,,\quad \chi= \tfrac{\rho_*U^2_*}{\kappa_*}\,.
  \end{aligned}
\end{equation}
The above constants are called Mach number, Reynolds Number, Ericksen number and the so-called inertial constant, respectively. We remark here that the inertial constant $\chi$ measures the inertial effect of the liquid crystal molecular. Experiments \cite{Gen} show that it is usually quite small.

By \eqref{units} and \eqref{dimentionless}, we can rewrite the system \eqref{sys1} into the dimensionless form (delete all the hats):
\begin{align}\label{sys-dimensionless}
\left\{
 \begin{array}{ll}
  \partial_{t} \rho + {\rm div} ( \rho {\rm u} ) = 0 \,,\\[2mm]
  \partial_{t} ( \rho {\rm u} ) + { \rm div} ( \rho {\rm u} \otimes {\rm u} ) + \tfrac{1}{{\rm Ma}^2}\nabla {p} = \tfrac{1}{\rm Re}{\rm div}\Sigma_{1} + \tfrac{1}{\rm Re} \tfrac{1}{\rm Er}{\rm div}\Sigma_{2} + \tfrac{1}{\rm Re}{\rm div}\Sigma_{3}  \,,\\[2mm]
  \tfrac{\chi}{\rm Er} \rho {\ddot{\rm d}} = \tfrac{\kappa}{\rm Er} \Delta {\rm d} + \Gamma \rm d + \lambda_1 ( {\dot{\rm d}} + \B {\rm d} ) + \lambda_2 \A \rm d\,,
 \end{array}\right.
\end{align}
where
\begin{equation*}
  \Gamma =  -\tfrac{\chi}{\rm Er} \rho| \dot{\rm d} |^2 + \tfrac{\kappa}{\rm Er} | \nabla {\rm d} |^2 - \lambda_2 {\rm d}^\top \A {\rm d} \,.
\end{equation*}

Based on the above dimensionless form, several asymptotic behavior as the dimensionless numbers tend to zero or infinity could be investigated. For example, as the inertial number $\chi$ goes to zero, the system \eqref{sys-dimensionless} will converge to the much more studied parabolic Ericksen-Leslie system, see \cite{JL-2019, JLT-NA, JLTZ-CMS2019}. In the $Q\mbox{-}$tensor version, say, the Beris-Edwards system, Wu, Xu and Zarnescu \cite{Wu-Xu-Zarnescu-ARMA2019} justified the limit letting the Reynolds and Ericksen numbers go to infinity (in the setting $\chi$=0).

Since we concern with the {\em low Mach number limit} in this paper, we set the coefficients ${\rm Re, Er}$ and $\chi$ as $1$, and  set Mach number  ${\rm Ma} = \eps$. Then the  system \eqref{sys-dimensionless} reads as
\begin{align}\label{sys2}
\left\{
 \begin{array}{ll}
   \p_{t} \ro + {\rm div} ( \ro \u ) = 0 \,, \\[2mm]
   \p_{t} ( \ro \u ) + {\rm div } ( \ro \u \otimes \u ) + \tfrac{1}{\e^{2}} \nabla p ( \ro ) = { \rm div } ( \Sigma_{1}^\e + \Sigma_{2}^\e + \Sigma_{3}^\e ) \,, \\[2mm]
   \ro \ddot{\rm d}^\e = \kappa \Delta \d + \Gamma^\e \d + \lambda_{1} ( \dot{\rm d}^\e + \B^\e \d ) + \lambda_{2} \A^\e \d \,,
  \end{array}\right.
\end{align}
where
\begin{align}\label{Ga-eps}
    \Gamma^\e = - \rho^\e | {\dot{\rm d}}^\e |^2 + \ka | \nabla {\rm d}^\e |^2 - \lambda_2 {{\rm d}^\e}^\top \A^\e {\rm d}^\e \,.
\end{align}

From the singular pressure term $\tfrac{1}{\e^{2}} \nabla p ( \ro )$ second equations in scaled system \eqref{sys2}, it is natural to see that in the limit $\e\rightarrow 0$, the density $\rho^\e$ will tend to a constant. So, without loss of generality, we consider the density with a small perturbation around the equilibrium state $\bar\rho = 1$ as
\begin{align*}
    \ro = 1 + \e \pe \,,
\end{align*}
then the system \eqref{sys2} becomes
\begin{align}\label{csys}
\left\{
 \begin{array}{ll}
    \p_t \pe + \u \cdot \na \pe + \pe \dv \u + \tfrac{1}{\e} \dv \u=0 \,, \\[2mm]
    \p_t \u + \u \cdot \na \u + \tfrac{1}{\e} \frac{1}{\ro} p' ( \ro ) \nabla \pe = \tfrac{1}{\ro} \dv ( \Sigma_{1}^\e + \Sigma_{2}^\e + \Sigma_{3}^\e ) \,, \\[2mm]
    \ddot{\rm d}^\e = \ka \tfrac{1}{\ro} \De \d + \tfrac{1}{\ro} \Gamma^\e \d + \lambda_{1} \tfrac{1}{\ro} ( \dot{\rm d}^\e + \B^\e \d ) + \lambda_{2} \tfrac{1}{\ro} \A^\e \d \,,
 \end{array}\right.
\end{align}
For the system \eqref{csys}, the initial data are given as
\begin{align}\label{civ}
    ( \pe , \u , \d, \cd ) |_{t=0} = ( \pe_{0}, \u_{0}, \d_{0}, \widetilde{\rm d}^\eps_0) \in \R \times \R^n \times \mathbb{S}^{n-1} \times \R^n \,,
\end{align}
and the boundary conditions at infinity are
\begin{align}\label{cbc}
    ( \phi^\eps , \u , \d ) \rightarrow ( 0 , { 0} , \bar{\rm d} ) , \quad {\rm{as}} \; \; |x| \rightarrow \infty \,,
\end{align}
where $\bar{\rm d}$ is a constant vector with $|\bar{\rm d}|=1$.

Formally, supposing that the limit $(\ro, \u, \d)\rightarrow (1, {\rm u}, {\rm d})$ exists and initially $\rho^\eps_0 = 1 + \eps \phi^\eps_0 \rightarrow 1$, $\u_0 \rightarrow {\rm u}_0$, $\d_0 \rightarrow {\rm d}_0$, $\widetilde{\rm d}^\eps_0 \rightarrow \widetilde{\rm d}_0$ as $\epsilon\rightarrow0$, then the limiting system will be
\begin{align}\label{isys}
\left\{
 \begin{array}{ll}
    {\rm div} \, {\rm u} = 0 \,, \\[2mm]
    \p_{t} {\rm u} + ( {\rm u} \ct \na ) {\rm u} - \tfrac{1}{2} \mu_4 \De {\rm u} + \na \pi = - \ka {\rm div} ( \na {\rm d} \odot \na {\rm d} ) + \dv \tilde{\sigma}_{\bm{\mu}} ({\rm u}, {\rm d}, \dot{\rm d}) \,, \\[2mm]
    \ddot{\rm d} = \kappa \Delta {\rm d} + {\Gamma} {\rm d} + \lambda_{1} ( \dot{\rm d} + {\rm B} {\rm d} ) + \lambda_{2} {\rm A} {\rm d} \,,
 \end{array}\right.
\end{align}
where $\na\pi$ is the ``limit'' of $\frac{1}{\e^2} \frac{\tilde a}{\ro} \na [  ( \ro )^\gamma - 1 ] - \ka \frac{1}{\ro} \na ( \frac{|\na\d|^2}{2} ) $, and ${\Gamma}$ is the formal limit of $\Gamma^\e$:
\begin{align}\label{Gal}
    {\Gamma} = - | \dot{\rm d} |^2 + \kappa | \nabla {\rm d} |^2 - \lambda_2 {\rm d}^\top \A {\rm d} \,.
\end{align}
Moreover, the system \eqref{isys} is endowed with the initial data:
\begin{align}\label{iiv}
    ( {\rm u} , {\rm d}, \dot{\rm d} ) |_{t=0} = ( {\rm u}_{0} , {\rm d}_{0}, \widetilde{\rm d}_0) \in \R^n \times \mathbb{S}^{n-1} \times \R^n \,,
\end{align}
and the boundary conditions at infinity
\begin{align}\label{ibc}
    ( {\rm u} , {\rm d} ) \rightarrow ( { 0} , {\bar{\rm d}} ) , \quad {\rm{as}} \;\; |x| \rightarrow \infty \,.
\end{align}
The main goal of our paper is to justify the above limit {\em rigorously} in the framework of classical solutions.

\subsection{Historical remarks}
In fluid (and related) models, Mach number measures the compressibility of fluids. As Mach number goes to zero, the compressible models behave asymptotically as the incompressible models. The rigorous mathematical justifications were initialized from Klainerman and Majda \cite{KM-81, KM-82} in the framework of classical solutions, which addressed the limit from compressible Euler to incompressible Euler equations. After then, there has been vast of literature on this topic in different models (Navier-Stokes, MHD, etc.), different contexts of solutions (classical solutions, weak solutions, strong solutions, etc.) and different domains (whole space, periodic domain, bounded domain, etc.). Among them, we mention Ukai \cite{Ukai} (initial layer in $\mathbb{R}^n$), Schochet \cite{Schochat-1994} and Grenier \cite{Grenier} (fast acoustic waves on torus), Lions and Masmoudi \cite{Lions-Masmoudi-1998} (incompressible limit for global in time weak solutions of isentropic compressible Navier-Stokes equations),  M\'etivier and Schochet \cite{Metivier-Schochet,sch86} (incompressible limit for non-isentropic Euler equations),
Alazard \cite{Al} (low Mach number limit of the full Navier-Stokes equations in $\mathbb{R}^n$), Jiang, Ju and Li \cite{JJL16} (incompressible limit for non-isentropic MHD equations in $\mathbb{R}^n$), and
Jiang et al. \cite{JJLX} (low Mach number limit of the full MHD equations in $\mathbb{R}^n$).
The interested readers can refer the survey papers \cite{d,ma} and the monograph \cite{fn} for more references.

We now review some related results on the incompressible limits for the parabolic-type Ericksen-Leslie's system without the inertia $\rho^\eps \ddot{\rm d}^\eps$, which is
\begin{align}\label{pcsys}
\left\{
 \begin{array}{ll}
    \p_t \ro + \dv ( \ro\u )=0, \\
    \p_t (\ro\u) + \dv (\ro\u\otimes\u) + \frac{1}{\e^2} \na p(\ro) \\
\qquad  = \mu \De \u - \la \dv \Big ( \na \d \odot \na \d - \big( \frac{1}{2} | \na \d |^2 + F(\d) \big) \I \Big), \\
    \p_t\d+\u\ct\na\d=\theta(\De\d-f(\d)),
 \end{array}\right.
\end{align}
where $f({\rm d})$ and $F({\rm d})$ denote the penalty function and the bulk part of the elastic energy, respectively, satisfying the relation $f({\rm d})=\na_{\rm d}F({\rm d})$. The constants $\mu,\la,\theta$ denote the viscosity, the competition between kinetic energy and potential energy, and the microscopic elastic relation time for the molecular orientation field, respectively. When taking the typical form
 $$
 F({\rm d})=\frac{1}{4\sigma_0^2}(|{\rm d}|^2-1)^2, \quad f({\rm d})=\frac{1}{2\sigma_0^2}(|{\rm d}|^2-1){\rm d}
 $$
  for a constant $\sigma_0>0$, Wang and Yu \cite{WY} studied weak solutions to weak solutions of the incompressible limit for the system \eqref{pcsys} with  general initial data in a bounded domain by employing the method of spectral analysis and Duhamel's principle,
  see also \cite{HL}.
  By adding the geometric constraint $ |{\rm d}|=1 $ and choosing $F({\rm d})=0$ and $f({\rm d})=|\na {\rm d}|^2{\rm d}$, the incompressible limit
  was studied by several authors, see Ding et al. \cite{DHWZ} (small classical solutions on torus), Yang \cite{yang}, Liu and Dou \cite{ld} and  Zeng, Ni, and Ai \cite{zeng}( small classical solutions on bounded domain),
  and Qi and Xu \cite{QX}, Bie et al. \cite{bie} ( small strong solutions in Besov space). For the studies of the well-posedness to \eqref{pcsys}, see \cite{DHWZ} and its references cited therein.  The interested readers can refer the nice summary  papers \cite{LW,HP} and the monographs \cite{Gen,Ste}
  for more results on the parabolic-type  Ericksen-Leslie's system.

  For the inertial case, the research on the Ericksen-Leslie system is much less.  When taking $\eps = 1$ and $\ka = 1$ in the compressible flow \eqref{sys2}, Jiang, Luo and Tang \cite{JLT} established both the local existence of the classical solution with finite initial energy under $\la_1 \leq 0$, and the global well-posedness with small size of the initial data under a damping effect $\la_1 < 0$. Recently, Huang et al. \cite{HJLZ-Cmp} proved the global well-posedness of the compressible system without damping effect ($\la_1 = 0$) under small size of initial data but without the kinematic transport.

  For the incompressible  Ericksen-Leslie hyperbolic liquid crystal model \eqref{isys} with $\ka = 1$ and the inertial term $\ddot{\rm d}$ replaced by $\rho_1 \ddot{\rm d}$ for an inertial constant $\rho_1 > 0$, Jiang and Luo \cite{JL} not only verified the local existence and uniqueness of the classical solution with finite initial energy and $\la_1 \leq 0$, but also obtained the global classical solution under the small initial energy with an additional damping effect (say, $\la_1 < 0$). Later on, Huang et al. \cite{HJLZ-Incmp} and Cai-Wang \cite{CW} proved the global existence of the incompressible hyperbolic Ericksen-Leslie's liquid crystal model without damping effect (i.e. $\la_1 = 0$) and kinematic transport under the assumption of small initial data. The main concern of this paper is to connect the solutions of \cite{JLT} and \cite{JL} of  the inertial Ericksen-Leslie systems in the context of classical solutions.

  The incompressible hyperbolic liquid crystal model in the Q-tensor framework is called Qian-Sheng model, which couples a hyperbolic type equation involving a second order material derivative with a forced incompressible Navier--Stokes equation. De Anna and Zarnescu \cite{DeAnna-Zarnescu-2018} derived the
  energy law and proved the local well-posedness for bounded initial data and global well-posedness under the assumptions that the initial data are small in some suitable norm and the coefficients satisfy some further damping property. To the best knowledge of the authors, \cite{DeAnna-Zarnescu-2018} might be the first work that treats the second-order material derivative for multidimensional case in the field of liquid crystal equations. In \cite{DeAnna-Zarnescu-2018}, they also provided an example of twist-wave solutions, which are solutions of the
  coupled system for which the flow vanishes for all times. Furthermore, for the inviscid
  version of the Qian--Sheng model, Feireisl et al. \cite{Feireisl-etal-2018} proved a global existence of the dissipative solution which is inspired from that of incompressible Euler equation
  defined by P.-L. Lions \cite{Lions}.

  \subsection{Notations and main results}

In order to state our main results and simplify the presentation, we introduce some notations. For $p\in [1,+\infty]$, we introduce the weighted $L^p_w$ $(:= L^p_w (\R^n))$ spaces endowed with the norms
\begin{equation*}
  \begin{aligned}
    \| \phi \|_{L^p_w} = \Big( \int_{\R^n} w (x) | \phi (x) |^p {\dd} x \Big)^\frac{1}{p}  \, (1 \leq p < +\infty) \,, \;\; \| \phi \|_{L^\infty_w} = \textrm{ess} \sup_{x \in \R^n}  | \phi (x) w (x)| \,,
  \end{aligned}
\end{equation*}
for some positive weight function $w : \R^n \rightarrow \R^+$. If the weight function $w \equiv 1$, $L^p_1$ stands for the usual $L^p$ space. For $p = 2$, we use the symbol $\langle \cdot , \cdot \rangle$ to represent the $L^2$-inner product in $L^2$. We denote $\A = (a_{ij})$ to mean that $a_{ij}$ is the $(ij)^{\mathrm{th}}$ element of $n \times n$ matrix $\A$. The notation ${\rm u} \otimes {\rm v}$ is the tensor product of vectors ${\rm u}$ and {\rm v}, which reads as ${\rm u} \otimes {\rm v} = (u_i v_j)$. The symbol $\A:\B$ stands for the scalar product of the matrixes $\A$ and $\B$, specifically, $\A:\B = \sum_{i,j} a_{ij} b_{ij}$. And we use the notation $\lt \A : \B \rt$ to substitute for $\int_{\R^n} \A : \B \dd x $.

For the multi-index $m = (m_1, m_2 , \cdots , m_n) \in \mathbb{N}^n$, we denote the $m$-{th} derivative operator $\p^m f$ as
\begin{align*}
    &\p^m f : =\frac{\p^{|m|}f}{\p_{x_1}^{m_1}\p_{x_2}^{m_2}\cdots\p_{x_n}^{m_n}},
\end{align*}
where $|m| = m_1 + m_2 + \cdots + m_n$. The notation $m \leq \tilde{m}$ means that each component of $m$ is not greater than that of $\tilde{m}$. In addition, $m < \tilde{m}$ represents $m_i \leq \tilde{m}_i(i=1,\cdots,n)$ and $|m| < |\tilde{m}|$. Let $s \geq 0$ be an integer. The symbols $H^{s}_w$ and $\dot{H}^{s}_w$ stand for the weighted Sobolev space $H^{s}_w(\mathbb{R}^{n})$ and the weighted homogeneous Sobolev space $\dot{H}^{s}_w(\mathbb{R}^{n})$, respectively, with the norms
\begin{align*}
    \| f \|_{H^s_w} = \bigg( \sum_{ |m| \leq s} \| \pa f \|^2_{L^2_w} \bigg)^{\frac{1}{2}} \,, \quad \| f \|_{\dot{H}^s_w} = \bigg( \sum_{1 \leq |m| \leq s} \| \pa f \|^2_{L^2_w} \bigg)^{\frac{1}{2}}
\end{align*}
for some positive weighted function $w : \R^n \rightarrow \R^+$. If $w \equiv 1$, $H^s = H^s_1$ and $\dot{H}^s = \dot{H}^s_1$ are the usual Sobolev space and homogeneous Sobolev space, respectively. Moreover, if $s = 0$, we have $L^2_w = H^0_w$. Furthermore, we define the local Sobolev space $H^s_{\loc}$ in the mean of $ \int_\Omega |\p^m f|^2 {\dd} x < + \infty$ for all multi-index $m \in \mathbb{N}^n$ with $|m| \leq s$ and any compact domain $\Omega \subseteq \R^n$.

For convenience, the symbol $[ \p^m , f ] g$ denotes $\p^m ( f g ) - f \p^m g$. The same letter $C$ denotes various positive constants independent of the Mach number $\epsilon$. The notation $a\lesssim b$ means that there exists two constant $C>0$ such that $a\leq Cb$. Moreover, if there are some positive constants $C_1$ and $C_2$, independent of $\eps$, such that $C_1 a \leq b \leq C_2 a$, we denote it by $a \approx b$.

Now we state our main results in the following.
\begin{thm}\label{Thm-global}
	Let $0 < \eps \leq 1$ and the integer $s > \tfrac{n}{2} + 1$ with $n=2,3$. The initial data $\big( \phi^\eps_0 , \u_0, \d_0 , \widetilde{\rm d}^\eps_0  \big) \in \R \times \R^n \times \mathbb{S}^{n-1} \times \R^n$ are assumed to satisfy
\begin{equation}\label{d-initial}
   \widetilde{\rm d}^\eps_0 \cdot \d_0 = 0\,,
\end{equation}
and
\begin{equation}\label{phi-initial}		
\| \phi^\eps_0 \|_{L^\infty} \leq \tfrac{1}{2}\,,\quad\!\phi^\eps_0 \in H^s_{p' (\rho^\eps_0)}\,,\quad\!\u_0, \,\widetilde{\rm d}^\eps_0 \in H^s_{\rho^\eps_0}\,,\quad\!\na \d_0 \in H^s\,,
\end{equation}
 where $\rho^\eps_0 = 1 + \eps \phi^\eps_0$. We further assume that the Leslie coefficients satisfy \eqref{Parodi-1} and
	\begin{equation}\label{Coefficients}
	  \ka > 0 \,, \ \mu_1 \geq 0 \,, \ \mu_4 > 0 \,, \ \tfrac{1}{2} \mu_4 + \xi \geq 0 \,, \ \la_1 < 0 \,, \ \mu_5 + \mu_6 + \tfrac{\la_2^2}{\la_1} \geq 0 \,.
	\end{equation}
Then following two statements hold:
 \begin{itemize}
	\item There exists a small $\delta > 0$, independent of $\eps \in (0,1]$, such that if
	\begin{equation}\label{IC-1}
	  \begin{aligned}
	    \| \phi^\eps_0 \|^2_{H^s_{p'(\rho^\eps_0)}} + \| \u_0 \|^2_{H^s_{\rho^\eps_0}} + \| \widetilde{\rm d}^\eps_0 \|^2_{H^s_{\rho^\eps_0}} + \ka \| \na \d_0 \|^2_{H^s} \leq \delta \,,
	  \end{aligned}
	\end{equation}
	then the Cauchy problem \eqref{csys}-\eqref{civ} admits a unique global solution $(\phi^\eps, \u , \d)$ with
	\begin{equation*}
	  \begin{aligned}
	 \ \ \quad  \quad  \phi^\eps \in L^\infty (\R^+; H^s_{p' (\rho^\eps)}) , \ \u, \cd \in L^\infty (\R^+; H^s_{\rho^\eps}) , \ \na \u \in L^2 (\R^+; H^s) , \ \na \d \in L^\infty (\R^+; H^s) ,
	  \end{aligned}
	\end{equation*}
	where $\rho^\eps = 1 + \eps \phi^\eps$. Moreover, the following uniform $\mathrm{(}$in $\eps \in (0,1]$$\mathrm{)}$ energy bounds hold:
	\begin{equation}\label{Unif-Bnd-1}
	  \begin{aligned}
	    \| \phi^\eps \|^2_{L^\infty(\R^+; H^s_{p'(\rho^\eps)} )} & + \| \u \|^2_{L^\infty (\R^+; H^s_{\rho^\eps})} + \| \cd \|^2_{L^\infty (\R^+; H^s_{\rho^\eps})} \\
	    & + \ka \| \na \d \|^2_{L^\infty (\R^+; H^s)} + \tfrac{1}{2} \mu_4 \| \na \u \|^2_{L^2 (\R^+; H^s)} \lesssim \delta \,,
	  \end{aligned}
	\end{equation}
	\begin{equation}\label{Unif-Bnd-rho}
	  \begin{aligned}
	   \| \rho^\eps \|_{L^\infty (0, T_1; L^\infty)} = \| 1 + \eps \phi^\eps \|_{L^\infty (0, T_1; L^\infty)} \approx 1 \,,
	  \end{aligned}
	\end{equation}
for any finite $ 0 < T_1 < +\infty $, and
	\begin{equation}\label{Unif-Bnd-2}
	  \begin{aligned}
	    \| \p_t \cd \|^2_{L^\infty (\R^+; H^{s-1})} + \| \p_t \d \|^2_{L^\infty(\R^+; H^s)} \leq C_{\rm d} \,,
	  \end{aligned}
	\end{equation}
	where the constant $C_{\rm d} > 0$ is independent of $\eps$.

    \item If besides the smallness condition \eqref{IC-1}, the initial data are ``well-prepared", i.e. $\phi^\eps_0$ and $\u_0$ further satisfy
	\begin{equation}\label{IC-2}
	  \begin{aligned}
	    \| \na \phi^\eps_0 \|_{H^{s-2}} \leq C_\phi \eps \,\,\,\,\, \textrm{and}\,\,\,\, \| \dv \u_0 \|_{H^{s-2}} \leq C_{\rm u} \eps
	  \end{aligned}
	\end{equation}
     for some $\eps$-independent constants $C_\phi, C_{\rm u} > 0$, then, $\phi^\eps$ and $\u$ have better regularity, i.e.  for any $T > 0$,
	\begin{equation}\label{Unif-Bnd-3}
	  \begin{aligned}
	    \| \p_t \phi^\eps \|^2_{L^\infty(0,T; H^{s-2}_{p'(\rho^\eps)})} + \| \p_t \u \|^2_{L^\infty(0,T; H^{s-2}_{\rho^\eps})} \leq C_{\phi {\rm u}} (T)
	  \end{aligned}
	\end{equation}
	and
	\begin{equation}\label{Unif-Bnd-4}
	  \begin{aligned}
	    \tfrac{1}{\eps} \| \dv \u \|_{L^\infty(0;T;H^{s-2})} + \tfrac{1}{\eps} \| p' (\rho^\eps) \na \phi^\eps \|_{L^\infty(0,T; H^{s-2})} \leq C_{\phi {\rm u}}' (T) \,,
	  \end{aligned}
	\end{equation}
	where the constants $ C_{\phi {\rm u}} (T), C_{\phi {\rm u}}' (T) > 0$ are all independent of $\eps \in (0,1]$.
\end{itemize}
\end{thm}
\begin{rem}
In the above Theorem, if setting $\eps=1$, we automatically recover the existence results in \cite{JLT}. However, Theorem \ref{Thm-global} here proved much more precise estimates. In fact, even in Part 1 of Theorem \ref{Thm-global} we justify the uniform in $\eps$ bounds on the perturbations with size $\eps$.
\end{rem}

The next theorem is about the limit from the Ericksen-Leslie hyperbolic liquid crystal model in compressible flow to the corresponding model in incompressible flow.

\begin{thm}\label{Thm-Limit}
	Under the same assumptions as those in Theorem \ref{Thm-global}, i.e. the assumptions \eqref{d-initial}, \eqref{phi-initial}, \eqref{Coefficients}, \eqref{IC-1} and \eqref{IC-2},  we further assume that there exist  ${\rm u}_0, \widetilde{\rm d}_0 \in \R^n$ and ${\rm d}_0 \in \mathbb{S}^{n-1}$ with ${\rm u}_0, \widetilde{\rm d}_0, \na {\rm d}_0 \in H^s$, such that ${\rm d}_0 \cdot \widetilde{\rm d}_0 = 0$, $\dv {\rm u}_0 = 0$, $\| \d_0 - {\rm d}_0 \|_{L^2} \rightarrow 0$ and
	\begin{equation}
	  \begin{aligned}
	    \big( \u_0, \widetilde{\rm d}^\eps_0, \na \d_0 \big) \rightarrow \big( {\rm u}_0, \widetilde{\rm d}_0, \na {\rm d}_0 \big) \textrm{ strongly in } H^s
	  \end{aligned}
	\end{equation}
	as $\eps \rightarrow 0$. Let $\big( \rho^\eps= 1 + \eps \phi^\eps, \u, \d \big)$ be the family of solutions to the system \eqref{csys} constructed in Theorem \ref{Thm-global}. Then there exist ${\rm u} \in \R^n$, $\pi \in \R$ and ${\rm d}\in \mathbb{S}^{n-1}$ with ${\rm u}, \dot{\rm d}, \na {\rm d} \in L^\infty(\R^+; H^s) \cap C(\R^+; H^{s-1}_{\loc})$, $\na {\rm u} \in L^2(\R^+; H^s)$ and $ \pi \in L^\infty (\R^+; H^{s-1}_{\rm loc}) $, such that $\mathrm{(}$in the sense of subsequences$\mathrm{)}$
	\begin{equation}
	  \begin{aligned}
	    & \rho^\eps \rightarrow 1 \textrm{ strongly in } L^\infty (\R^+; H^s) \,, \\
	    & \tfrac{1}{\eps^2} \nabla p (\rho^\eps) - \tfrac{1}{2} \ka \nabla |\na \d|^2 \rightarrow \nabla \pi \textrm{ weakly-} \star \textrm{ for } t > 0, \textrm{ weakly in } H^{s-2}_{\rm loc}
	  \end{aligned}
	\end{equation}
	as $\eps \rightarrow 0$ and
	\begin{equation}
	  \begin{aligned}
	    \big( \u, \na \d , \cd \big) \rightarrow \big( {\rm u}, \na {\rm d}, \dot{\rm d} \big)
	  \end{aligned}
	\end{equation}
	weakly-$\star$ for $t > 0$, weakly in $H^s$ and strongly in $C(\R^+; H^{s-1}_{\loc})$ as $\eps \rightarrow 0$. Here $({\rm u}, \pi, {\rm d})$ is the solution to the Ericksen-Leslie hyperbolic liquid crystal model in incompressible flow \eqref{isys} with initial data \eqref{iiv}. Moreover, the following global energy bound holds:
	\begin{equation}\label{Bnd-Limit}
	  \begin{aligned}
	    \| {\rm u} \|^2_{L^\infty (\R^+; H^s)} + \| \dot{\rm d} \|^2_{L^\infty (\R^+; H^s)} & + \| \na {\rm d} \|^2_{L^\infty (\R^+; H^s)} + \tfrac{1}{2} \mu_4 \| \na {\rm u} \|^2_{L^2 (\R^+; H^s)}  \lesssim \delta \,,
	  \end{aligned}
	\end{equation}
where $\delta$ is the small constant in the assumption \eqref{IC-1}.
\end{thm}

The last result is about the convergence rate of the limit in Theorem \ref{Thm-Limit}.

\begin{thm}\label{Thm-Convergence-Rate}
Under the assumptions of Theorem \ref{Thm-Limit}, we further assume that
 \begin{equation}\label{IC-CVRT}
   \begin{aligned}
     \| \sqrt{\rho^\eps_0} \u_0 - {\rm u}_0 \|^2_{L^2} + \| \sqrt{\rho^\eps_0} \widetilde{\rm d}^\eps_0 - \widetilde{\rm d}_0 \|^2_{L^2} & + \ka \| \na \d_0 - \na {\rm d} \|^2_{L^2} \\
     & + \| \d_0 - {\rm d} \|^2_{L^2} + \lt \Pi^\eps_0 , 1 \rt \lesssim \eps^{\alpha_0}
   \end{aligned}
 \end{equation}
for some constant $\alpha_0 > 0$, where $ \Pi^\eps_0 = \tfrac{1}{\eps^2} \tfrac{\tilde a}{\gamma - 1} \big[ (\rho^\eps_0)^\gamma - \gamma (\rho^\eps_0) - 1 \big ] $.
Then, for any fixed $T>0$, we have
\begin{equation*}
  \begin{aligned}
    \| \sqrt{\rho^\eps} \u - {\rm u} \|^2_{L^2} + \| \sqrt{\rho^\eps} \cd - \dot{\rm d} \|^2_{L^2} + \ka \| \na \d - \na {\rm d} \|^2_{L^2} + \| \d - {\rm d} \|^2_{L^2} + \lt \Pi^\eps , 1 \rt \leq C_T \eps^{\beta_0}
  \end{aligned}
\end{equation*}
for all $t\in[0,T]$, where $\Pi^\eps = \tfrac{1}{\eps^2} \tfrac{\tilde a}{\gamma - 1} \big[ (\rho^\eps)^\gamma - \gamma (\rho^\eps - 1) -1 \big]$, and the constants $\beta_0 = \min \{ 2 , \alpha_0, 1 + \tfrac{\alpha_0}{2} \} > 0$ and $C_T = C (1 + T) \exp (CT) > 0$ for some positive constant $C$, independent of $\eps$.
\end{thm}

\begin{rem}
In view of interpolation inequality, the convergence space $ L^2 $ can be raised to the Sobolev space $ H^{s'} (0<s' < s)$, correspondingly, the convergence rate will become slowly.
\end{rem}

\subsection{Sketch of proofs}

The {\em key ideas} of the current paper is that we obtain the uniform estimates for the unknown functions $(\phi^\eps, \u, \d , \cd)$, their time derivatives $(\p_t \phi^\eps, \p_t \u, \p_t \d, \p_t \cd)$ and the singular quantities $( \tfrac{1}{\eps} \dv \u , \tfrac{1}{\eps} \tfrac{1}{\rho^\eps} p' (\rho^\eps) \na \phi^\eps )$. While justifying Theorem \ref{Thm-global}, we divide the proof into four steps as follows:
\begin{enumerate}
	\item We derive the a priori estimate of the system \eqref{csys}. The {\em key point} is to deal with the singular terms $\tfrac{1}{\eps} \dv \u$ in the $\phi^\eps$-equation and $\tfrac{1}{\eps} \tfrac{1}{\rho^\eps} p' (\rho^\eps) \na \phi^\eps $ in the $\u$-equation of \eqref{csys}. More precisely, we utilize the following cancellation under the relation $\rho^\eps = 1 + \eps \phi^\eps$ to overcome the singular terms:
	\begin{equation*}
	  \begin{aligned}
	    & \tfrac{1}{\epsilon} \lt \p^m \dv \u , p' (\rho^\epsilon) \p^m \phi^\epsilon \rt + \tfrac{1}{\epsilon} \lt \p^m \big( \tfrac{1}{\rho^\epsilon} p' (\rho^\epsilon) \na \phi^\epsilon \big) , \rho^\epsilon \p^m \u \rt \\
	    = & - \lt p'' (\rho^\epsilon) \na \phi^\epsilon \p^m \phi^\epsilon , \p^m \u \rt + \tfrac{1}{\epsilon} \lt \big[ \p^m , \tfrac{1}{\rho^\epsilon} p' (\rho^\epsilon) \na \big] \phi^\epsilon , \rho^\epsilon \p^m \u \rt
	  \end{aligned}
	\end{equation*}
	for all multi-index $m \in \mathbb{N}^n$, of which derivations will be given later. We note that $ [ \p^m, \tfrac{1}{\rho^\eps} p' (\rho^\eps) \na ]$ vanishes when $m = 0$, and will generate a coefficient $\eps$ under the relation $\rho^\eps = 1 + \eps \phi^\eps$ when $m \neq 0$. In other word, the term $ \tfrac{1}{\epsilon} \lt \big[ \p^m , \tfrac{1}{\rho^\epsilon} p' (\rho^\epsilon) \na \big] \phi^\epsilon , \rho^\epsilon \p^m \u \rt $ is actually not a singular term.
	
	\item We prove the local well-posedness of the system \eqref{csys} with uniform in $\eps$ small initial data \eqref{civ} by employing the iterative scheme. Actually, we also could obtain a local existence time $T \in (0,1)$ which is independent of $\eps$. Based on the previous constructed local solutions, we seek some more dissipative structures about the density fluctuation of $ \phi^\eps$ and the direction field $ \d$, say, $\| \na \phi^\eps \|^2_{H^{s-1}_{w(\rho^\eps)}}$ with $w(\rho^\eps) = \tfrac{1}{\rho^\eps} p'(\rho^\eps)$ and $\| \na \d \|^2_{\dot{H}^s_{1/\rho^\eps}}$, so that we can globally extend the constructed local solution.
	
	\item We derive the uniform bounds \eqref{Unif-Bnd-2} and \eqref{Unif-Bnd-3}, which are concerned on the time derivatives $\p_t \u, \p_t \phi^\eps, \p_t \d$ and $\p_t \cd$. These uniform bounds will be employed to derive some strong convergences by utilizing the Aubin-Lions-Simon Theorem given in Lemma \ref{Lmm-ALS-Thm}. The key point is still to deal with the singularity. Noticing that the evolution of $\d$ and $\p_t \d = \cd - \u \cdot \na \d$ do not involve singular terms, we can directly deduce \eqref{Unif-Bnd-2} from the uniform bound \eqref{Unif-Bnd-1}, \eqref{Unif-Bnd-rho} and the equations of $\d$ in \eqref{csys}. However, the evolutions of $\phi^\eps$ and $\u$ contain the singular terms $\tfrac{1}{\eps} \dv \u$ and $\tfrac{1}{\eps} \tfrac{1}{\rho^\eps} p'(\rho^\eps) \na \phi^\eps$, respectively, which have not yet controlled uniformly in $\eps \in (0,1]$. Thus, when deriving the uniform bounds of $\p_t \phi^\eps$ and $\p_t \u$, we need to eliminate these singular effects under the following coupled cancellation:
	\begin{equation*}
	  \begin{aligned}
	    & \tfrac{1}{\eps} \lt \p^m \dv \p_t \u , p' (\rho^\eps) \p^m \p_t \phi^\eps \rt + \tfrac{1}{\eps} \lt \p^m ( \tfrac{1}{\rho^\eps} p' (\rho^\eps) \na \p_t \phi^\eps \phi^\eps ) , \rho^\eps \p^m \p_t \u \rt \\
	    = & \tfrac{1}{\eps} \lt \big[ \p^m , \tfrac{1}{\rho^\eps} p' (\rho^\eps) \na \big] \p_t \phi^\eps , \rho^\eps \p^m \p_t \u \rt - \lt p'' (\rho^\eps) \na \phi^\eps \p^m \p_t \phi^\eps , \p^m \p_t \u \rt
	  \end{aligned}
	\end{equation*}
	for all $m \in \mathbb{N}^n$, where the last term in the right-hand side is obviously not a singular term and the first term in the right-hand side is actually not singular as explained before. Therefore, we derive the energy inequality \eqref{dt-Es-1}, i.e.,
	\begin{equation*}
	  \begin{aligned}
	    E_s (\p_t \phi^\eps , \p_t \u) (t) \leq \big( 1 + E_s (\p_t \phi^\eps (0) , \p_t \u (0)) \big) \exp (C_8 T)
	  \end{aligned}
	\end{equation*}
	for all $t \in [0,T]$, where $T>0$ is an any fixed number, and the energy functional $E_s (\p_t \phi^\eps , \p_t \u)$ is defined in \eqref{ED-t-derivative}. In order to obtain the uniform bound of $E_s (\p_t \phi^\eps , \linebreak \p_t \u)$, we only need to ensure the initial energy $E_s (\p_t \phi^\eps (0) , \p_t \u (0))$ is uniformly bounded. So the additional initial conditions \eqref{IC-2} is required.
	
	\item Based on the uniform bounds \eqref{Unif-Bnd-1}, \eqref{Unif-Bnd-rho}, \eqref{Unif-Bnd-2} and \eqref{Unif-Bnd-3}, we easily derive the uniform bound \eqref{Unif-Bnd-4} of the singular terms $\tfrac{1}{\eps} \dv \u$ and $\tfrac{1}{\eps} \tfrac{1}{\rho^\eps} p' (\rho^\eps) \na \phi^\eps$ from the structures of evolutions \eqref{csys}.
\end{enumerate}

In order to prove Theorem \ref{Thm-Limit}, we employ the compactness arguments depended on Aubin-Lions-Simon Theorem stated in Lemma \ref{Lmm-ALS-Thm}, so that the global classical solution of the compressible system \eqref{csys} constructed in Theorem \ref{Thm-global} converges (in the sense of subsequences) to the solution of incompressible equations \eqref{isys}.

While proving Theorem \ref{Thm-Convergence-Rate}, we employ the modulated energy method or, say, the relative entropy approach, motivated by \cite{JJL}. There is an important observation \eqref{Cdisp-equ} in Lemma \ref{Lmm-Contrl-C-disp}, i.e.,
\begin{equation*}
  \begin{aligned}
    \mathscr{C}_{\rm disp} = & -\mu_1 \int_0^t \| {\d}^\top({\A}^\eps-\A)\d \|^2_{L^2} {\dd} \tau + \la_1 \int_0^t \big\| \cd - \dot{\rm d} + ({\B}^\eps - {\B}) \d + \tfrac{\la_2}{\la_1} ({\A}^\eps - {\A}) \d \big\|^2_{L^2} {\dd} \tau \\
    & - ( \mu_5 + \mu_6 + \tfrac{\la_2^2}{\la_1} ) \int_0^t \| (\A^\eps - \A) \d \|^2_{L^2} \mathrm{d} \tau + \mathscr{R}_{\rm{\Sigma}}^\eps \,,
  \end{aligned}
\end{equation*}
where the quantity $\mathscr{C}_{\rm disp}$ is given in \eqref{C-disp}, so that we get the convergence rates for the case of well-prepared initial data.

\subsection{Organizations of current paper}

In next section, we prove the global well-posedness of the compressible system \eqref{csys} under uniformly in $\eps$ small initial data and give the uniform (in $\eps$) bounds of the solutions, their time derivatives and the singular quantities $\tfrac{1}{\eps} \dv \u$, $ \tfrac{1}{\eps} \tfrac{1}{\rho^\eps} p' (\rho^\eps) \na \phi^\eps $. In Section \ref{Sec:Limit}, we rigorously show the limit process between \eqref{csys} and \eqref{isys} by employing compactness arguments. In Section \ref{Sec:Convergence}, based on Theorem \ref{Thm-Limit}, we prove the convergence rates in $L^2$ space with well-prepared initial data.

\section{Global uniform energy bounds: Proof of Theorem \ref{Thm-global}}\label{s3}

In this section, we aim at proving Theorem \ref{Thm-global}, i.e. the global well-posedness of the system \eqref{csys} with small initial data \eqref{civ} for any fixed $\eps \in (0,1]$, and deriving the uniform in $\eps$ global-in-time energy bounds. We first derive the a priori estimates of the system \eqref{csys}. Second, we prove the local well-posedness of the system \eqref{csys} with uniform in $\eps$ small initial data \eqref{civ} by employing the iterative scheme. Third, we derive the uniform bounds \eqref{Unif-Bnd-2} and \eqref{Unif-Bnd-3}, which are concerned on the time derivatives $\p_t \phi^\eps$, $\p_t \u$, $\p_t \cd$ and $\p_t \d$. Finally, based on the uniform bounds \eqref{Unif-Bnd-1}, \eqref{Unif-Bnd-rho}, \eqref{Unif-Bnd-2} and \eqref{Unif-Bnd-3}, we easily derive the uniform bound \eqref{Unif-Bnd-4} of the singular terms $\tfrac{1}{\eps} \dv \u$ and $\tfrac{1}{\eps} \tfrac{1}{\rho^\eps} p' (\rho^\eps) \na \phi^\eps$ from the structures of evolutions \eqref{csys}.

\subsection{Preliminaries} In this subsection, we first list some basic conclusions, which will be used frequently in the procedure of proving Theorem \ref{Thm-global}.

\begin{lem}[Moser-type calculus inequalities \cite{KM-81}]\label{lem1}
Assume $f,g\in H^{s}$. Then for any multi-index $m = ( m_{1} , \cdots , m_{n}) \in \mathbb{N}^n$ with $1 \leq |m|\leq s $, we have
\begin{align*}
    \| \p^m ( f g ) \|_{L^{2}} \lesssim\, & \| f \|_{L^{\infty}} \| g \|_{\dot{H}^s} + \| g \|_{L^{\infty}} \| f \|_{\dot{H}^s} \,, \\
    \| [\p^m , f] g \|_{L^2} \lesssim\, & \|\na f \|_{L^{\infty}} \| g \|_{H^{s-1}} + \| g \|_{L^{\infty}} \| f \|_{ \dot{H}^s } \,.
\end{align*}
In particular, if $s > \tfrac{n}{2}$, then $f g \in H^{s}$ and
\begin{align*}
    \| f g \|_{H^{s}}\lesssim & \| f \|_{H^{s}} \| g \|_{H^{s}} \,.
\end{align*}
\end{lem}

\begin{lem}[Sobolev's embedding inequalities \cite{N}]\label{lem2}
	The following conclusions holds:
	\begin{enumerate}
		\item Assume $f \in H^{s}$ for $s > \tfrac{n}{2}$. Then $ f \in L^{\infty} $ with the bound
		\begin{align*}
		  \| f \|_{L^{\infty}} \lesssim \| f \|_{H^{s}} \,.
		\end{align*}
		
		\item Assume $ f \in H^{1}$. Then $f \in L^p$ with the bound
		\begin{align*}
		 \| f \|_{L^p} \lesssim \| \nabla f \|_{L^{2}}^{\frac{n}{2} - \frac{n}{p}} \| f \|_{L^{2}}^{1-\frac{n}{2} + \frac{n}{p}} \,.
		\end{align*}
		Here $2 < p < + \infty$ for $n = 2$ and $2 < p \leq \frac{2n}{n-2}$ for $n \geq 3$.
	\end{enumerate}
\end{lem}

\begin{lem}\label{Lmm-Cummutators}
	Let the integer $s > \tfrac{n}{2} + 1$ with $n = 2, 3$. For any functions ${\rm u} , g \in H^s$ and multi-index $m \in \mathbb{N}^n$ with $1 \leq |m| \leq s$, we have
	\begin{equation*}
	  \begin{aligned}
	    \big\| [\p^m , {\rm u} \cdot \na] g \big\|_{L^2} \lesssim \| {\rm u} \|_{\dot{H}^s} \| g \|_{\dot{H}^s} \,.
	  \end{aligned}
	\end{equation*}
\end{lem}

\begin{proof}
	It follows from Lemma \ref{lem2} that
	\begin{equation*}
	  \begin{aligned}
	   \big\| [\p^m , {\rm u} \cdot \na] g \big\|_{L^2} \lesssim\, &\sum_{0 \neq m' \leq m} \| \p^{m'} {\rm u} \cdot \na \p^{m-m'} g \|_{L^2} \\
	    \lesssim\, & \| \p^m {\rm u} \|_{L^2} \| \na g \|_{L^\infty} + \sum_{|m'|=1, m' \leq m} \| \p^{m'} {\rm u} \|_{L^\infty} \| \na \p^{m-m'} g \|_{L^2} \\
	    & + \sum_{|m'| \geq 2, m' < m} \| \p^{m'} {\rm u} \|_{L^4} \| \na \p^{m-m'} g \|_{L^4} \\
	    \lesssim\, & \| \p^m {\rm u} \|_{L^2} \| \na g \|_{H^{s-1}} + \sum_{|m'|=1, m' \leq m} \| \p^{m'} {\rm u} \|_{H^{s-1}} \| \na \p^{m-m'} g \|_{L^2} \\
	    & + \sum_{|m'| \geq 2, m' < m} \| \p^{m'} {\rm u} \|_{H^1} \| \na \p^{m-m'} g \|_{H^1} \\
	    \lesssim\, & \| {\rm u} \|_{\dot{H}^s} \| g \|_{\dot{H}^s} \,.
	  \end{aligned}
	\end{equation*}
\end{proof}

\begin{lem}[\cite{JLT}]\label{Lm-rho}
	Assume that $\underline{\rho} \leq \rho_0 (x) \leq \overline{\rho}$ for some constants $\underline{\rho} \,, \overline{\rho} > 0 $ and $\rho (t,x) \geq 0$ satisfies the continuity equation
	\begin{equation*}
	  \left\{
	    \begin{array}{l}
	      \partial_t \rho + \dv (\rho {\rm u}) = 0 \,, \\[2mm]
	      \rho (0, x) = \rho_0 (x)
	    \end{array}
	  \right.
	\end{equation*}
	for some given velocity ${\rm u }(t,x)$. Then there holds
	\begin{equation*}
	  \underline{\rho} \exp \Big ( - \int_{0}^{t} \| \operatorname{div} {\rm u} \|_{L^{\infty}}(\tau) \mathrm{d} \tau \Big ) \leq \rho(t, x) \leq \overline{\rho} \exp \Big ( \int_{0}^{t} \| \operatorname{div} {\rm u} \|_{L^{\infty}}(\tau) \mathrm{d} \tau \Big )
	\end{equation*}
	for all $(t,x) \in \R^+ \times \R^n$.
\end{lem}

\begin{lem}[\cite{JLT}]\label{lem3}
	Let $f(\rho)$ be a smooth function, then for any multi-index $m \in  \mathbb{N}^n$ with $ |m| \geq 1$ and $\rho \in H^{|m|} \cap L^\infty$, we have
	\begin{equation*}
	  \begin{aligned}
	    \p^m f(\rho) = \sum_{i=1}^{|m|} f^{(i)} (\rho) \sum_{\substack{\sum_{\ell = 1}^i m_\ell = m \\ |m_\ell| \geq 1}} \prod_{1 \leq \ell \leq i} \p^{m_\ell} \rho \,.
	  \end{aligned}
	\end{equation*}
	Furthermore, if $s > \tfrac{n}{2}$, $f (\rho) = \tfrac{1}{\rho}$ and $\rho$ satisfies the assumption stated in Lemma \ref{Lm-rho}, we then have
	\begin{equation*}
	  \begin{aligned}
	    \| \p^m f (\rho) \|_{L^2} \leq\, & \sum_{i=1}^{|m|} \tfrac{i !}{\underline{\rho}^{i+1}} \exp \Big ( (i+1) \int_0^t \| \dv {\rm u} \|_{L^\infty} (\tau) {\dd} \tau \Big ) \sum_{\substack{\sum_{\ell = 1}^i m_\ell = m \\ |m_\ell| \geq 1}} \Big\| \prod_{1 \leq \ell \leq i} \p^{m_\ell} \rho \Big\|_{L^2} \\
	    \leq\, & C (\underline{\rho}, m) \exp \Big ( (i+1) \int_0^t \| \dv {\rm u} \|_{L^\infty} (\tau) {\dd} \tau \Big ) P_{|m|} ( \| \rho \|_{\dot{H}^s} ) \\
	    \leq\, & \ka_1 \mathcal{Q}_{\ka_2} ({\rm u}) P_{|m|} ( \| \rho \|_{\dot{H}^s} )
	  \end{aligned}
	\end{equation*}
	for all $ t \in \mathbb{R}^{+} $ with some generic constants $\ka_1, \ka_2 > 0$, where $P_k(y)=\sum_{i=1}^{k} y^i$ and
	\begin{equation}\label{Q(u)}
	  \begin{aligned}
	    \mathcal{Q}_{\ka_2} ({\rm u}) = \exp \Big( \ka_2 \int_0^t \| \dv {\rm u} \|_{L^\infty} (\tau) {\dd} \tau \Big) \,.
	  \end{aligned}
	\end{equation}
\end{lem}
We remark that, in the rest of the paper, the notation $\mathcal{Q}_{\ka_2} ({\rm u})$ will be used frequently with different constant lower index $\ka_2$.

\subsection{A priori estimates for the system \eqref{csys}}

In this subsection, we intend to get the uniform estimates independent of $\e$ for the system \eqref{csys}. Observing that the two singular terms $\frac{1}{\e}\dv\u$ and $\frac{1}{\e}\frac{1}{\ro}p'(\ro)\na\pe$ occur in the continuity equation $\eqref{csys}_1$ and the momentum equation $\eqref{csys}_2$, respectively, and the two singular terms own some kind of skew-symmetry, our motivation is canceling the singularity to give the uniform estimates. For fixed $\e$, Jiang, Luo and Tang \cite{JLT} obtained the uniform estimates about time for the system equivalent to \eqref{csys}. Inspired by the method developed in \cite{JLT}, it is also effective for us to get the uniform estimates on both time and the Mach number, which not only helps us cancel the singularity but also enhance the regularity.

We first introduce the following energy functional
\begin{equation}\label{Es-local}
  \begin{aligned}
    \mathcal{E}_{s} (\phi^\epsilon , \u, \d) = \| \phi^\epsilon \|^2_{H^s_{p' (\rho^\epsilon)}} + \| \u \|^2_{H^s_{\rho^\epsilon}} + \| \cd \|^2_{H^s_{\rho^\epsilon}} + \ka \| \nabla \d \|^2_{H^s} \,,
  \end{aligned}
\end{equation}
and the energy dissipative rate functional
\begin{equation}\label{Ds-local}
  \begin{aligned}
    \mathcal{D}_{s} ( \u, \d) =\, & \tfrac{1}{2} \mu_4 \| \nabla \u \|^2_{H^s} + \big( \tfrac{1}{2} \mu_4 + \xi \big) \| \dv \u \|^2_{H^s} + \mu_1  \sum_{|m| \leq s} \| \d{}^\top ( \partial^m \A^\epsilon ) \d \|^2_{L^2} \\
    & - \la_1 \sum_{|m| \leq s} \left\| \partial^m \cd + ( \partial^m {\B}^\epsilon ) \d + \tfrac{\la_2}{\la_1} ( \partial^m \A^\epsilon ) \d \right\|^2_{L^2} \\
    & + \Big( \mu_5 + \mu_6 + \tfrac{\la_2^2}{\la_1} \Big) \sum_{|m| \leq s} \| ( \partial^m \A^\epsilon ) \d \|^2_{L^2} \,.
  \end{aligned}
\end{equation}

\begin{lem}\label{Lmm-Apriori-Est}
	Let the integer $s > \tfrac{n}{2} + 1$ $(n=2,3)$ and $0 < \eps \leq 1$. Assume that $(\phi^\eps, \u, \d)$ is a sufficiently smooth solution to \eqref{csys}. Then there is a constant $c > 0$, independent of $\eps$, such that
	\begin{equation}\label{Local-Apriori-Est}
	  \begin{aligned}
	    & \tfrac{1}{2} \tfrac{\dd}{\dd t} \mathcal{E}_s (\phi^\eps , \u , \d) + \mathcal{D}_s (\u, \d) \\
	    & \qquad \lesssim \mathcal{Q}_c (\u) \Big( 1 + \mathcal{E}_s^{s+1} (\phi^\eps, \u, \d) \Big) \mathcal{E}_s^\frac{1}{2} ( \phi^\eps , \u, \d ) \mathcal{A}_s (\phi^\eps, \u , \d) \,,
	  \end{aligned}
	\end{equation}
	where the energy functional $ \mathcal{A}_s (\phi^\eps, \u , \d) $ is defined as
	\begin{equation}\label{As-local}
	  \begin{aligned}
	    \mathcal{A}_s (\phi^\eps, \u , \d) = \,& \big( \| \phi^\eps \|_{\dot{H}^s} + \| \u \|_{\dot{H}^s} + \| \na \d \|_{\dot{H}^s} + \| \cd \|_{H^s} \big) \\
	    & \times \big( \| \na \u \|_{H^s} + \| \phi^\eps \|_{\dot{H}^s} + \| \na \d \|_{\dot{H}^s} + \| \cd \|_{H^s} \big)
	  \end{aligned}
	\end{equation}
	and $\mathcal{Q}_c (\u)$ is given in \eqref{Q(u)}.
\end{lem}

\begin{proof}
We divide it into two steps.

\vspace{0.2cm}
\emph{Step 1. Basic energy estimates.} We first multiply the $\phi^\epsilon$-equation of \eqref{csys} by $p' (\rho^\epsilon) \phi^\epsilon$ and integrate by parts over $x \in \R^n$. We then have
\begin{equation}\label{L2-rho}
  \begin{aligned}
    & \tfrac{1}{2} \tfrac{\dd}{\dd t} \| \phi^\epsilon \|^2_{L^2_{p' (\rho^\epsilon)}} + \tfrac{1}{\epsilon} \langle \dv \u ,  p' ( \ro ) \phi^\epsilon \rangle \\
    &\qquad = \tfrac{1}{2} \langle \partial_{t} p' (\rho^\epsilon) ,  |\phi^\epsilon|^2 \rangle - \langle \u \cdot \nabla \phi^\epsilon , p' (\ro) \phi^\epsilon \rangle - \langle p' (\rho^\epsilon) |\phi^\epsilon|^2 , \dv \u \rangle \,.
  \end{aligned}
\end{equation}
Following the same calculations in Section 2 of \cite{JLT}, taking the $L^2$-inner products of $\u$-equation and $\d$-equation in \eqref{csys} with $\rho^\epsilon \u$ and $\rho^\epsilon \cd$, respectively, it says that
\begin{equation}\label{L2-u}
  \begin{aligned}
    & \tfrac{1}{2} \tfrac{\dd}{\dd t} \| \u \|^2_{L^2_{\rho^\epsilon}} + \tfrac{1}{2} \mu_4 \| \nabla \u \|^2_{L^2} + \left( \tfrac{1}{2} \mu_4 + \xi \right) \| \dv \u \|^2_{L^2} \\
    & \qquad + \tfrac{1}{\epsilon} \lt \tfrac{1}{\rho^\epsilon} p' (\rho^\epsilon) \nabla \phi^\epsilon , \rho^\epsilon \u \rt = \lt \dv (\Sigma_2^\epsilon + \Sigma_3^\epsilon) , \u \rt \,,
  \end{aligned}
\end{equation}
and
\begin{equation}\label{L2-d}
  \begin{aligned}
    & \tfrac{1}{2} \tfrac{\dd}{\dd t} \left( \| \cd \|^2_{L^2_{\rho^\epsilon}} + \kappa \| \nabla \d \|^2_{L^2} \right) - \kappa \lt \Delta \d , \u \cdot \nabla \d \rt \\
    & \qquad = \la_1 \| \cd \|^2_{L^2} + \la_1 \lt \B^\epsilon \d , \cd \rt + \la_2 \lt \A^\epsilon \d , \cd \rt \,.
  \end{aligned}
\end{equation}
For the singular terms (with coefficient $\frac{1}{\epsilon}$) in \eqref{L2-rho} and \eqref{L2-u}, we employ the following cancellation to eliminate the singularity:
\begin{equation}\label{cs0}
  \begin{aligned}
    & \tfrac{1}{\e} \lt \dv \u,  p'(\ro) \pe \rt + \tfrac{1}{\e} \lt \tfrac{1}{\ro} p'(\ro) \na \pe , \ro \u \rt \\
    =\, & \tfrac{1}{\e} \lt \dv ( p' (\ro) \pe \u ) , 1 \rt - \tfrac{1}{\e} \lt \na ( p'(\ro) \pe ) , \u \rt + \tfrac{1}{\e} \lt \tfrac{1}{\ro} p'(\ro) \na \pe , \ro \u \rt \\
    =\, & - \tfrac{1}{\e} \lt \na p'(\ro) \pe , \u \rt = - \lt p''(\ro) \pe \na \pe , \u \rt \,.
  \end{aligned}
\end{equation}
Furthermore, from the analogous calculations in Section 2 of \cite{JLT}, one easily derives the following cancellation relations:
\begin{equation}\label{cs1}
  \begin{aligned}
    \ka \lt \Delta \d , \u \cdot \nabla \d \rt + \lt \dv \Sigma_2^\epsilon , \u \rt = 0 \,,
  \end{aligned}
\end{equation}
and
\begin{equation}\label{cs2}
  \begin{aligned}
    & \lt \dv \Sigma_3^\epsilon , \u \rt + \la_1 \| \cd \|^2_{L^2} + \la_1 \lt \cd , \B^\epsilon \d \rt + \la_2 \lt \cd , \A^\epsilon \d \rt \\
    = & - \mu_1 \| \d{}^\top \A^\epsilon \d \|^2_{L^2} + \la_{1} \| \cd \|^2_{L^2} + \la_1 \| \B^\epsilon \d \|^2_{L^2} - (\mu_5 + \mu_6 ) \| \A^\epsilon \d \|^2_{L^2} \\
    & + 2 \la_1 \lt \cd , \B^\epsilon \d \rt + 2 \la_2 \lt \cd + \B^\epsilon \d , \A^\epsilon \d \rt \\
    =  & - \mu_1 \| \d{}^\top \A^\epsilon \d \|^2_{L^2} + \la_{1} \left\| \cd + \B^\epsilon \d + \tfrac{\la_2}{\la_1} \A^\epsilon \d \right\|^2_{L^2} - \left( \mu_5 + \mu_6 + \tfrac{\la_2^2}{\la_1} \right) \| \A^\epsilon \d \|^2_{L^2} \,.
  \end{aligned}
\end{equation}
It is thereby derived from adding the equalities \eqref{L2-rho}-\eqref{L2-u} to \eqref{L2-d} and combining the cancellations \eqref{cs0}-\eqref{cs1}-\eqref{cs2} that
\begin{equation}\label{be0}
  \begin{aligned}
    & \tfrac{1}{2} \tfrac{\dd}{\dd t} \big( \| \phi^\epsilon \|^2_{L^2_{p' (\rho^\epsilon)}} + \| \u \|^2_{L^2_{\rho^\epsilon}} + \| \cd \|^2_{L^2_{\rho^\epsilon}} + \ka \| \nabla \d \|^2_{L^2} \big) + \tfrac{1}{2} \mu_4 \| \nabla \u \|^2_{L^2} + \mu_1 \| \d{}^\top \A^\epsilon \d \|^2_{L^2} \\
    & + ( \tfrac{1}{2} \mu_4 + \xi ) \| \dv \u \|^2_{L^2} - \la_1 \| \cd + \B^\epsilon \d + \tfrac{\la_2}{\la_1} \A^\epsilon \d \|^2_{L^2} + \big( \mu_5 + \mu_6 + \tfrac{\la_2^2}{\la_1} \big) \| \A^\epsilon \d \|^2_{L^2} \\
    = \,& \underbrace{ \tfrac{1}{2} \lt \partial_t p' (\rho^\epsilon) , |\phi^\epsilon|^2 \rt }_{A_1} \ \underbrace{ - \lt \u \cdot \nabla \phi^\epsilon , p' (\rho^\epsilon) \phi^\epsilon \rt - \lt p' (\rho^\epsilon) |\phi^\epsilon|^2 , \dv \u \rt + \lt p'' (\rho^\epsilon) \phi^\epsilon \nabla \phi^\epsilon , \u \rt }_{A_2} \,.
  \end{aligned}
\end{equation}

Recalling that $\partial_t \rho^\epsilon + \dv (\rho^\epsilon \u) = 0$, we have
\begin{equation}\label{p-rho-evl}
  \begin{aligned}
    \p_t p' (\rho^\epsilon) + \epsilon p'' (\rho^\epsilon) \u \cdot \nabla \phi^\epsilon + p'' (\rho^\epsilon) \rho^\epsilon \dv \u = 0 \,.
  \end{aligned}
\end{equation}
Then the quantity $A_1$ can be calculated as
\begin{equation}\label{A1}
  \begin{aligned}
    A_1 =\, & - \tfrac{1}{2} \epsilon \lt p'' (\rho^\epsilon) \u \cdot \nabla \phi^\epsilon , |\phi^\epsilon|^2 \rt - \tfrac{1}{2} \lt p'' (\rho^\epsilon) \rho^\epsilon \dv \u , |\phi^\epsilon|^2 \rt \\
    \lesssim \,& \epsilon \| p'' (\rho^\epsilon) \|_{L^\infty} \| \nabla \phi^\epsilon \|_{L^2} \| \u \|_{L^6} \| \phi^\epsilon \|_{L^6} + \| p'' (\rho^\epsilon) \|_{L^\infty} \| \rho^\epsilon \|_{L^\infty} \| \dv \u \|_{L^2} \| \phi^\epsilon \|^2_{L^4} \\
    \lesssim \, & \exp \Big( c_0' \int_0^t \| \dv \u \|_{L^\infty} (\tau) {\dd} \tau \Big) \| \nabla \phi^\epsilon \|_{L^2} \big( \| \nabla \u \|_{L^2} + \| \nabla \phi^\epsilon \|_{L^2} \big) \big( \| \u \|_{H^1} + \| \phi^\epsilon \|_{H^1} \big)
  \end{aligned}
\end{equation}
for some constant $c_0' > 0$, independent of the Mach number $0 < \epsilon \leq 1$, where the last two inequalities are implied by the H\"older's inequality, Lemma \ref{lem2}, the expression $p'' (\rho^\epsilon) = \tilde{a} \gamma (\gamma - 1) \big( \rho^\epsilon \big)^{\gamma - 2}$ and Lemma \ref{Lm-rho}. Moreover, similar arguments in estimating $A_1$ reduce to
\begin{equation}\label{A2}
  \begin{aligned}
    A_2 \lesssim\, & \| p' (\rho^\epsilon) \|_{L^\infty} \| \phi^\epsilon \|_{L^4} \big( \| \nabla \phi^\epsilon \|_{L^2} \| \u \|_{L^4} + \| \dv \u \|_{L^2} \| \phi^\epsilon \|_{L^4} \big) \\
    & + \| p'' (\rho^\epsilon) \|_{L^\infty} \| \nabla \phi^\epsilon \|_{L^2} \| \u \|_{L^4} \| \phi^\epsilon \|_{L^4} \\
    \lesssim \,& \exp \Big( c_0'' \int_0^t \| \dv \u \|_{L^\infty} (\tau) {\dd} \tau \Big) \| \nabla \phi^\epsilon \|_{L^2} \big( \| \nabla \u \|_{L^2} + \| \nabla \phi^\epsilon \|_{L^2} \big) \big( \| \u \|_{H^1} + \| \phi^\epsilon \|_{H^1} \big)
  \end{aligned}
\end{equation}
for some constant $c_0'' > 0$, independent of $\epsilon > 0$. Let $c_0 = \max \{ c_0' , c_0'' \} > 0$. From substituting the bounds \eqref{A1} and \eqref{A2} into the equality \eqref{be0}, we thereby derive that
\begin{equation}\label{be}
  \begin{aligned}
    & \tfrac{1}{2} \tfrac{\dd}{\dd t} \big( \| \phi^\epsilon \|^2_{L^2_{p' (\rho^\epsilon)}} + \| \u \|^2_{L^2_{\rho^\epsilon}} + \| \cd \|^2_{L^2_{\rho^\epsilon}} + \ka \| \nabla \d \|^2_{L^2} \big) + \tfrac{1}{2} \mu_4 \| \nabla \u \|^2_{L^2} + \mu_1 \| \d{}^\top \A^\epsilon \d \|^2_{L^2} \\
    & + ( \tfrac{1}{2} \mu_4 + \xi ) \| \dv \u \|^2_{L^2} - \la_1 \| \cd + {\B}^\epsilon \d + \tfrac{\la_2}{\la_1} {\A}^\epsilon \d \|^2_{L^2} + \big( \mu_5 + \mu_6 + \tfrac{\la_2^2}{\la_1} \big) \| \A^\epsilon \d \|^2_{L^2} \\
    \lesssim\, & \mathcal{Q}_{c_0} (\u) \| \nabla \phi^\epsilon \|_{L^2} \big( \| \nabla \u \|_{L^2} + \| \nabla \phi^\epsilon \|_{L^2} \big) \big( \| \u \|_{H^1} + \| \phi^\epsilon \|_{H^1} \big)
  \end{aligned}
\end{equation}
for all $0 < \epsilon \leq 1$.

\vspace{0.2cm}
\emph{Step 2. Higher order energy estimates for $\pe$, $\u$ and $\cd$.} For all $1\leq |m| \leq s$, applying the differential operator $\pa$ to the first $\phi^\epsilon$-equation of \eqref{csys} and taking $L^2$-inner product via multiplying by $p' (\rho^\epsilon) \p^m \phi^\epsilon$, we have
\begin{equation}\label{HD-phi}
  \begin{aligned}
    & \tfrac{1}{2} \tfrac{\dd}{\dd t} \| \partial^m \phi^\epsilon \|^2_{L^2_{p' (\rho^\epsilon)}} + \tfrac{1}{\epsilon} \lt \p^m \dv \u , p' (\rho^\epsilon) \p^m \phi^\epsilon \rt \\
   &\quad = \tfrac{1}{2} \lt \p_t p' (\rho^\epsilon) , |\p^m \phi^\epsilon|^2 \rt - \lt \p^m (\u \cdot \nabla \phi^\epsilon) , p' (\rho^\epsilon) \p^m \phi^\epsilon \rt - \lt \p^m (\phi^\epsilon \dv \u) , p' (\rho^\epsilon) \p^m \phi^\epsilon \rt \,.
  \end{aligned}
\end{equation}
We then employ the operator $\p^m$ to the $\u$-equation of \eqref{csys} and multiply by $\rho^\epsilon \p^m \u$, and integrate the resulting equation over $\R^n$ with respect to $x$. We thereby derive
\begin{equation}\label{HD-u}
  \begin{aligned}
    & \tfrac{1}{2} \tfrac{\dd}{\dd t} \| \p^m \u \|^2_{L^2_{\rho^\epsilon}} + \tfrac{1}{2} \mu_4 \| \nabla \p^m \u \|^2_{L^2} + \big( \tfrac{1}{2} \mu_4 + \xi \big) \| \dv \p^m \u \|^2_{L^2} \\
    & + \tfrac{1}{\epsilon} \lt \p^m \big( \tfrac{1}{\rho^\epsilon} p' (\rho^\epsilon) \na \phi^\epsilon \big) , \rho^\epsilon \p^m \u \rt + \lt [\p^m , \u \cdot \na ] \u , \rho^\epsilon \p^m \u \rt \\
    = \,& \lt \dv \p^m (\Sigma_{2}^\epsilon + \Sigma_{3}^\epsilon) , \p^m \u \rt + \big\lt \big[ \p^m , \tfrac{1}{\rho^\epsilon} \dv \big] ( \Sigma_{1}^\epsilon + \Sigma_{2}^\epsilon + \Sigma_{3}^\epsilon ) , \rho^\epsilon \p^m \u \big\rt \,.
  \end{aligned}
\end{equation}
We finally derive from applying the operator $\p^m$ to the $\d$-equation of \eqref{csys} and taking $L^2$-inner product with $ \rho^\epsilon \p^m \cd$ that
\begin{equation}\label{HD-d}
  \begin{aligned}
    & \tfrac{1}{2} \tfrac{\dd}{\dd t} \big( \| \p^m \cd \|^2_{L^2_{\rho^\epsilon}} + \ka \| \nabla \p^m \d \|^2_{L^2} \big) - \ka \lt \Delta \p^m \d , \p^m (\u \cdot \na \d ) \rt \\
    =\, & \la_1 \lt \p^m (\cd + {\B}^\epsilon \d) , \p^m \cd \rt + \la_2 \lt \p^m ({\A}^\epsilon \d) , \p^m \cd \rt \\
    & + \ka \big\lt [ \p^m , \tfrac{1}{\rho^\epsilon} \Delta ] \d , \rho^\epsilon \p^m \cd \rt + \lt \p^m ( \Gamma^\epsilon \d ) , \p^m \cd \rt \\
    & - \lt [ \p^m , \u \cdot \na ] \cd , \rho^\epsilon \p^m \cd \rt + \lt [\p^m , \tfrac{1}{\rho^\epsilon}] ( \Gamma^\epsilon \d ) , \rho^\epsilon \p^m \cd \rt \\
    & + \la_1 \lt [\p^m , \tfrac{1}{\rho^\epsilon}] ( \cd + {\B}^\epsilon \d ) , \rho^\epsilon \p^m \cd \rt + \la_2 \lt [\p^m , \tfrac{1}{\rho^\epsilon}] ( {\A}^\epsilon \d ) , \rho^\epsilon \p^m \cd \rt \,.
  \end{aligned}
\end{equation}
In order to deal with the singular terms (with coefficient $\tfrac{1}{\epsilon}$) in \eqref{HD-phi} and \eqref{HD-u}, we calculate the following cancellation
\begin{equation}\label{csm}
  \begin{aligned}
    & \tfrac{1}{\epsilon} \lt \p^m \dv \u , p' (\rho^\epsilon) \p^m \phi^\epsilon \rt + \tfrac{1}{\epsilon} \lt \p^m \big( \tfrac{1}{\rho^\epsilon} p' (\rho^\epsilon) \na \phi^\epsilon \big) , \rho^\epsilon \p^m \u \rt \\
    =\, & \tfrac{1}{\epsilon} \lt \dv \big( p' (\rho^\epsilon) \p^m \phi^\epsilon \p^m \u \big) , 1 \rt - \tfrac{1}{\epsilon} \lt \na \big( p' (\rho^\epsilon) \p^m \phi^\epsilon \big) , \p^m \u \rt \\
    & + \tfrac{1}{\epsilon} \lt p' (\rho^\epsilon) \p^m \na \phi^\epsilon , \p^m \u \rt + \tfrac{1}{\epsilon} \lt \big[ \p^m , \tfrac{1}{\rho^\epsilon} p' (\rho^\epsilon) \na \big] \phi^\epsilon , \rho^\epsilon \p^m \u \rt \\
    =\, & - \tfrac{1}{\epsilon} \lt \na p' (\rho^\epsilon) \p^m \phi^\epsilon , \p^m \u \rt + \tfrac{1}{\epsilon} \lt \big[ \p^m , \tfrac{1}{\rho^\epsilon} p' (\rho^\epsilon) \na \big] \phi^\epsilon , \rho^\epsilon \p^m \u \rt \\
    =\, & - \lt p'' (\rho^\epsilon) \na \phi^\epsilon \p^m \phi^\epsilon , \p^m \u \rt + \tfrac{1}{\epsilon} \lt \big[ \p^m , \tfrac{1}{\rho^\epsilon} p' (\rho^\epsilon) \na \big] \phi^\epsilon , \rho^\epsilon \p^m \u \rt \,,
  \end{aligned}
\end{equation}
where the last formally singular term in \eqref{csm} is not a real singularity since the commutator operator $\big[ \p^m , \tfrac{1}{\rho^\epsilon} p' (\rho^\epsilon) \na \big]$ will generate a small coefficient $\epsilon$.

By adding the equalities \eqref{HD-phi}-\eqref{HD-u} to the relation \eqref{HD-d} and combining the cancellation \eqref{csm}, one has
\begin{equation}\label{he}
  \begin{aligned}
    & \tfrac{1}{2} \tfrac{\dd}{\dd t} \big( \| \p^m \phi^\epsilon \|^2_{L^2_{p' (\rho^\epsilon)}} + \| \p^m \u \|^2_{L^2_{\rho^\epsilon}} + \| \p^m \cd \|^2_{L^2_{\rho^\epsilon}} + \ka \| \nabla \p^m \d \|^2_{L^2} \big) \\
    &\qquad  + \tfrac{1}{2} \mu_4 \| \na \p^m \u \|^2_{L^2} + \big( \tfrac{1}{2} \mu_4 + \xi \big) \| \dv \p^m \u \|^2_{L^2} = \mathcal{I} + \mathcal{J} \,,
  \end{aligned}
\end{equation}
where
\begin{equation}\label{I-mc}
  \begin{aligned}
    \mathcal{I} =\, & \underbrace{ \tfrac{1}{2} \lt \p_t p' (\rho^\epsilon) , |\p^m \phi^\epsilon|^2 \rt }_{\mathcal{I}_1} + \underbrace{ \lt p'' (\rho^\epsilon) \na \phi^\epsilon \p^m \phi^\epsilon , \p^m \u \rt }_{\mathcal{I}_2} \\
    & + \underbrace{ \lt \dv \p^m \Sigma_{2}^\epsilon , \p^m \u \rt + \ka \lt \Delta \p^m \d , \p^m (\u \cdot \nabla \d) \rt }_{\mathcal{I}_3} + \underbrace{ \lt \p^m ( \Gamma^\epsilon \d ) , \p^m \cd \rt }_{\mathcal{I}_4} \\
    & + \underbrace{ \la_1 \lt \p^m ( \cd + {\B}^\epsilon \d ) , \p^m \cd \rt + \la_2 \lt \p^m ({\A}^\epsilon \d) , \p^m \cd \rt + \lt \dv \p^m \Sigma_3^\epsilon , \p^m \u \rt }_{\mathcal{I}_5} \\
    & \underbrace{ - \lt \p^m (\u \cdot \na \phi^\epsilon) , p' (\rho^\epsilon) \p^m \phi^\epsilon \rt - \lt \p^m ( \phi^\epsilon \dv \u ) , p' (\rho^\epsilon) \p^m \phi^\epsilon \rt }_{\mathcal{I}_6} \,,
  \end{aligned}
\end{equation}
and
  \begin{align}\label{J-mc}
    \nonumber \mathcal{J} =\, & \underbrace{ - \big\lt [\p^m , \u \cdot \na] \u , \rho^\epsilon \p^m \u \big\rt }_{\mathcal{J}_1} \ \underbrace{ - \tfrac{1}{\epsilon} \big\lt \big[ \p^m , \tfrac{1}{\rho^\epsilon} p' (\rho^\epsilon) \nabla \big] \phi^\epsilon , \rho^\epsilon \p^m \u \big\rt }_{\mathcal{J}_2} \\
    \nonumber & \underbrace{ - \big\lt [\p^m , \u \cdot \na] \cd , \rho^\epsilon \p^m \cd \big\rt }_{\mathcal{J}_3} + \underbrace{ \big\lt [ \p^m , \tfrac{1}{\rho^\epsilon} \dv ] \Sigma_{1}^\epsilon , \rho^\epsilon \p^m \u \big\rt }_{\mathcal{J}_4} \\
    \nonumber & + \underbrace{ \big\lt [ \p^m , \tfrac{1}{\rho^\epsilon} \dv ] \Sigma_2^\epsilon , \rho^\epsilon \p^m \u \big\rt }_{\mathcal{J}_5} + \underbrace{ \big\lt [ \p^m , \tfrac{1}{\rho^\epsilon} \dv ] \Sigma_3^\epsilon , \rho^\epsilon \p^m \u \big\rt }_{\mathcal{J}_6} \\
    \nonumber & + \underbrace{ \big\lt [ \p^m , \tfrac{1}{\rho^\epsilon} ] ( \Gamma^\epsilon \d ) , \rho^\epsilon \p^m \cd \big\rt }_{\mathcal{J}_7} + \underbrace{ \la_2 \big\lt [ \p^m , \tfrac{1}{\rho^\epsilon} ] ( \A^\epsilon \d ) , \rho^\epsilon \p^m \cd \big\rt }_{\mathcal{J}_8} \\
    & + \underbrace{ \la_1 \big\lt [ \p^m , \tfrac{1}{\rho^\epsilon} ] ( \cd + {\B}^\epsilon \d ) , \rho^\epsilon \p^m \cd \big\rt }_{\mathcal{J}_9} + \underbrace{ \ka \big\lt [ \p^m , \tfrac{1}{\rho^\epsilon} \Delta ] \d , \rho^\epsilon \p^m \cd \rt }_{\mathcal{J}_{10}} \,.
  \end{align}

It remains to estimate the terms $\mathcal{I}$ and $\mathcal{J}$, respectively. For the quantity $\mathcal{I}_1$, by employing the evolution \eqref{p-rho-evl}, we obtain
\begin{equation}\label{I1-mc}
  \begin{aligned}
    \mathcal{I}_1 =\, & - \tfrac{1}{2} \epsilon \lt p'' (\rho^\epsilon) \u \cdot \na \phi^\epsilon , |\p^m \phi^\epsilon|^2 \rt - \tfrac{1}{2} \lt p'' (\rho^\epsilon) \rho^\epsilon \dv \u , |\p^m \phi^\epsilon|^2 \rt \\
    \lesssim \,& \| p'' (\rho^\epsilon) \|_{L^\infty} ( 1 + \| \rho^\epsilon \|_{L^\infty} ) \big( \| \u \|_{L^\infty} \| \nabla \phi^\epsilon \|_{L^\infty} + \| \dv \u \|_{L^\infty} \big) \| \p^m \phi^\epsilon \|^2_{L^2} \\
    \lesssim \,& \mathcal{Q}_{c_1} (\u) \big( \| \u \|_{H^s} \| \phi^\epsilon \|_{\dot{H}^s} + \| \u \|_{\dot{H}^s}
     \big) \| \phi^\epsilon \|^2_{\dot{H}^s}
  \end{aligned}
\end{equation}
for all $0 < \epsilon \leq 1$, where the constant $c_1 > 0$ is independent of $\epsilon$ and Lemmas \ref{lem2} and \ref{Lm-rho} are also utilized. Here $s > \tfrac{n}{2} + 1$ is required. Similarly in \eqref{I1-mc}, one has
\begin{equation}\label{I2-mc}
  \begin{aligned}
    \mathcal{I}_2 \lesssim \mathcal{Q}_{c_2} (\u) \| \u \|_{\dot{H}^s} \| \phi^\epsilon \|^2_{\dot{H}}
  \end{aligned}
\end{equation}
for all $\epsilon \in (0,1]$ and some $c_2 > 0$ independent of $\epsilon$.

Moreover, it follows from the estimates $I_3$ (see (3.16) in Page 138 of \cite{JLT}) and $I_4$ (see (3.17) in Page 140 of \cite{JLT}) that
\begin{equation}\label{I3-mc}
  \begin{aligned}
    \mathcal{I}_3 \lesssim \| \na \d \|^2_{\dot{H}^s} \| \nabla \u \|_{H^s} \,,
  \end{aligned}
\end{equation}
and
\begin{equation}\label{I4-mc}
  \begin{aligned}
    \mathcal{I}_4 \lesssim\, & \big( \| \rho^\eps \|_{L^\infty} + \| \rho^\eps \|_{\dot{H}^s} \big) \| \na \d \|_{H^s} \| \cd \|^3_{H^s} + \| \na \d \|^2_{H^s} \| \na \d \|_{\dot{H}^s} \| \cd \|_{H^s} \\
    & + |\la_2| \sum_{1 \leq j \leq 3} \| \na \d \|^j_{H^s} \| \cd \|_{H^s} \| \na \u \|_{H^s} \\
    \lesssim \,& \big( 1 + \mathcal{Q}_{c_4} (\u) \big) \big( \| \cd \|^3_{H^s} + \| \phi^\eps \|^3_{\dot{H}^s} + \| \na \d \|^3_{H^s} + \| \na \d \|_{H^s} \big) \\
    & \times \big( \| \na \u \|_{H^s} + \| \na \d \|_{\dot{H}^s} + \| \cd \|_{H^s} \big) \| \cd \|_{H^s} \,,
  \end{aligned}
\end{equation}
respectively, for all $0 < \epsilon \leq 1$, where the constant $c_4 > 0$ is independent of $\eps$,  and Lemma \ref{Lm-rho} and the expansion $\rho^\eps = 1 + \eps \phi^\eps$ are also used. Furthermore, from the same arguments in Page 140-142 of \cite{JLT}, we deduce that
\begin{equation}\label{I5-mc}
  \begin{aligned}
    &\mathcal{I}_5 - \la_1 \big\| \p^m \cd + (\p^m {\B}^\eps) \d + \tfrac{\la_2}{\la_1} (\p^m {\A}^\eps) \d \big\|^2_{L^2} \\
    & + \mu_1 \big\| \d{}^\top (\p^m \A^\eps) \d \big\|^2_{L^2} + \big( \mu_5 + \mu_6 + \tfrac{\la_2^2}{\la_1} \big) \| (\p^m \A^\eps) \d \|^2_{L^2} \\
    \lesssim\,&\sum_{1 \leq j \leq 4} \| \na \d \|^j_{H^s} \big( \| \u \|_{\dot{H}^s} + \| \cd \|_{H^s} \big) \| \na \u \|_{H^s} + \| \u \|_{\dot{H}^s} \| \na \d \|_{H^s} \| \cd \|_{H^s} \\
 \lesssim\,    & \big( \| \na \d \|_{H^s} + \| \na \d \|^4_{H^s} \big) \big( \| \u \|_{\dot{H}^s} + \| \cd \|_{H^s} \big) \| \na \u \|_{H^s} \,.
  \end{aligned}
\end{equation}
For the term $\mathcal{I}_6$, we estimate it as
  \begin{align}\label{I6-mc}
    \nonumber \mathcal{I}_6 = \,& \tfrac{1}{2} \lt p' (\rho^\eps) \dv \u + \eps p '' (\rho^\eps) \u \cdot \na \phi^\eps , |\p^m \phi^\eps|^2 \rt - \lt [ \p^m , \u \cdot \na ] \phi^\eps , p' (\rho^\eps) \p^m \phi^\eps \rt \\
    \nonumber \,& - \lt \p^m (\phi^\eps \dv \u) , p' (\rho^\eps) \p^m \phi^\eps \rt \\
    \nonumber \lesssim\, & \big( \| p' (\rho^\eps) \|_{L^\infty} \| \dv \u \|_{L^\infty} + \eps \| p'' (\rho^\eps) \|_{L^\infty} \| \u \|_{L^\infty} \| \na \phi^\eps \|_{L^\infty} \big) \| \p^m \phi^\eps \|^2_{L^2} \\
    \nonumber & + \| p' (\rho^\eps) \|_{L^\infty} \| \p^m \phi^\eps \|_{L^2} \big( \| [\p^m , \u \cdot \na] \phi^\eps \|_{L^2} + \| \p^m (\phi^\eps \dv \u) \|_{L^2} \big) \\
    \nonumber \lesssim\, & \mathcal{Q}_{c_6} (\u) \Big\{ \big( \| \u \|_{\dot{H}^s} + \| \u \|_{H^s} \| \phi^\eps \|_{\dot{H}^s} \big) \| \phi^\eps \|^2_{\dot{H}^s} \\
    \nonumber & \qquad\qquad + \| \phi^\eps \|_{\dot{H}^s} \big( \| \u \|_{\dot{H}^s} \| \phi^\eps \|_{\dot{H}^s} + \| \na \u \|_{H^s} \| \phi^\eps \|_{H^s} \big) \Big\} \\
    \lesssim\, & \mathcal{Q}_{c_6} (\u) \big( \| \phi^\eps \|_{H^s} + \| \u \|_{H^s} + \| \u \|_{H^s} \| \phi^\eps \|_{H^s} \big) \big( \| \phi^\eps \|^2_{\dot{H}^s} + \| \na \u \|_{H^s} \| \phi^\eps \|_{\dot{H}^s} \big)
  \end{align}
for all $0 < \eps \leq 1$ and some positive constant $c_6$, independent of $\eps$, where we have made use of Lemmas \ref{lem1}, \ref{Lm-rho} and \ref{Lmm-Cummutators}.

By plugging the bounds \eqref{I1-mc}, \eqref{I2-mc}, \eqref{I3-mc}, \eqref{I4-mc}, \eqref{I5-mc} and \eqref{I6-mc} into \eqref{I-mc}, one therefore deduces that
\begin{equation}\label{I-mc-bnd}
  \begin{aligned}
    & \mathcal{I} - \la_1 \big\| \p^m \cd + (\p^m {\B}^\eps) \d + \tfrac{\la_2}{\la_1} (\p^m {\A}^\eps) \d \big\|^2_{L^2} \\
    & + \mu_1 \big\| \d{}^\top (\p^m \A^\eps) \d \big\|^2_{L^2} + \big( \mu_5 + \mu_6 + \tfrac{\la_2^2}{\la_1} \big) \| (\p^m \A^\eps) \d \|^2_{L^2} \\
    \lesssim\, &\mathcal{Q}_{c'_{\mathcal{I}}} (\u) \big( 1 + \| \phi^\eps \|^2_{H^s} + \| \u \|^2_{H^s} + \| \na \d \|^2_{H^s} + \| \cd \|^2_{H^s} \big) \\
    &   \times \big( \| \phi^\eps \|_{H^s} + \| \u \|_{H^s} + \| \na \d \|_{H^s} + \| \cd \|_{H^s} \big) \\
    &   \times \big( \| \phi^\eps \|_{\dot{H}^s} + \| \u \|_{\dot{H}^s} + \| \na \d \|_{\dot{H}^s} + \| \cd \|_{H^s} \big) \\
    & \qquad \times \big( \| \na \u \|_{H^s} + \| \phi^\eps \|_{\dot{H}^s} + \| \na \d \|_{\dot{H}^s} + \| \cd \|_{H^s} \big)
  \end{aligned}
\end{equation}
for all $m \in \mathbb{N}^n$ with $1 \leq |m| \leq s$ and for any $0 < \eps \leq 1$, where $c'_{\mathcal{I}} = \max \{ c_1, c_2, c_4, c_6 \} > 0$.

Next, we estimate the term $\mathcal{J}$ in \eqref{J-mc}. One notices that the term $\mathcal{J}_2$ has a singularity, namely, with the coefficient $\frac{1}{\eps}$ in the front. Fortunately, the derivative of $f (\rho^\eps) = \frac{p'(\ro)}{\ro}$ with $\ro=1+\eps \pe$ can product an $\eps$ to balance the singularity. Specifically, the term $\mathcal{J}_2$ can be controlled as
\begin{equation}\label{J2-mc}
  \begin{aligned}
    \mathcal{J}_2 = \,& - \tfrac{1}{\eps} \sum_{0 \neq m' \leq m} C_m^{m'} \big\lt \p^{m'} f (\rho^\eps) \na \p^{m-m'} \phi^\eps , \rho^\eps \p^m \u \big\rt \\
    \lesssim\, & \tfrac{1}{\eps} \| \p^m f (\rho^\eps) \|_{L^2} \| \na \phi^\eps \|_{L^\infty} \| \rho^\eps \|_{L^\infty} \| \p^m \u \|_{L^2} \\
    & + \tfrac{1}{\eps} \sum_{m' \leq m, |m'| = 1} \| f' (\rho^\eps) \|_{L^\infty} \| \p^{m'} \rho^\eps \|_{L^\infty} \| \na \p^{m-m'} \phi^\eps \|_{L^2} \| \rho^\eps \|_{L^\infty} \| \p^m \u \|_{L^2} \\
    & + \tfrac{1}{\eps} \sum_{m' < m, |m'| \geq 2} \| \p^{m'} f (\rho^\eps) \|_{L^4} \| \na \p^{m-m'} \phi^\eps \|_{L^4} \| \rho^\eps \|_{L^\infty} \| \p^m \u \|_{L^2} \\
    \lesssim \,& \tfrac{1}{\eps} \sum_{i=1}^{|m|} \| f^{(i)} (\rho^\eps) \|_{L^\infty} \| \rho^\eps \|_{L^\infty} P_{|m|} ( \| \rho^\eps \|_{\dot{H}^s} ) \| \phi^\eps \|_{\dot{H}^s} \| \u \|_{\dot{H}^s} \\
    & + \tfrac{1}{\eps} \sum_{m' \leq m , |m'| = 1} \| f' (\rho^\eps) \|_{L^\infty} \| \rho^\eps \|_{L^\infty} \| \p^{m'} \rho^\eps \|_{H^{s-1}} \| \na \p^{m-m'} \phi^\eps \|_{L^2} \| \u \|_{\dot{H}^s} \\
    & + \tfrac{1}{\eps} \sum_{m' < m, |m'| \geq 2} \| \p^{m'} f (\rho^\eps) \|_{H^1} \| \na \p^{m-m'} \phi^\eps \|_{H^1} \| \rho^\eps \|_{L^\infty} \| \u \|_{\dot{H}^s} \\
    \lesssim\, & \tfrac{1}{\eps} \mathcal{Q}_{c_2^\star} (\u) P_s ( \| \rho^\eps \|_{\dot{H}^s} ) \| \phi^\eps \|_{\dot{H}^s} \| \u \|_{\dot{H}^s} \\
    \lesssim \,& \mathcal{Q}_{c_2^\star} (\u) \big( \| \phi^\eps \|_{\dot{H}^s} + \| \phi^\eps \|^s_{\dot{H}^s} \big) \| \phi^\eps \|_{\dot{H}^s} \| \u \|_{\dot{H}^s}
  \end{aligned}
\end{equation}
for all $0 < \eps \leq 1$ and some constant $c_2^\star > 0$, independent of $\eps$, where we have made use of Lemma \ref{lem3}, Lemma \ref{lem1} and the expansion $\rho^\eps = 1 + \eps \phi^\eps$.

As for the other remained terms, $ ( \mathcal{J}_1, \mathcal{J}_3, \mathcal{J}_4, \mathcal{J}_5, \mathcal{J}_6, \mathcal{J}_7, \mathcal{J}_8 + \mathcal{J}_9, \mathcal{J}_{10} )$ are the same as the terms $(J_1, J_3, J_8, J_9, J_{10}, J_{11}, J_{12} + J_{13}, J_4)$, respectively, defined in Page 137 of \cite{JLT}. It easily follows from the same technical arguments in Page 143-146 of \cite{JLT} that
  \begin{align}\label{J-J2-mc}
    \nonumber \mathcal{J} - \mathcal{J}_2 \lesssim\, & \mathcal{Q}_{c_3^\star} (\u) \big( 1 + \| \rho^\eps \|^{2s+2}_{\dot{H}^s} + \| \u \|^{2s+2}_{\dot{H}^s} + \| \na \d \|^{2s + 2}_{H^s} + \| \cd \|^{2s+2}_{H^s} \big) \\
    \nonumber & \times \big( \| \rho^\eps \|_{\dot{H}^s} + \| \u \|_{\dot{H}^s} + \| \na \d \|_{H^s} + \| \cd \|_{H^s} \big) \\
    \nonumber & \times \big( \| \na \u \|_{H^s} \| \rho^\eps \|_{\dot{H}^s} + \| \cd \|^2_{H^s} + \| \na \d \|^2_{\dot{H}^s} \big) \\
    \nonumber \lesssim\, & \mathcal{Q}_{c_3^\star} (\u) \big( 1 + \| \phi^\eps \|^{2s+2}_{H^s} + \| \u \|^{2s+2}_{H^s} + \| \na \d \|^{2s + 2}_{H^s} + \| \cd \|^{2s+2}_{H^s} \big) \\
    \nonumber & \times \big( \| \phi^\eps \|_{H^s} + \| \u \|_{H^s} + \| \na \d \|_{H^s} + \| \cd \|_{H^s} \big) \\
    & \times \big( \| \na \u \|_{H^s} \| \phi^\eps \|_{\dot{H}^s} + \| \cd \|^2_{H^s} + \| \na \d \|^2_{\dot{H}^s} \big)
  \end{align}
for all $0 < \eps \leq 1$ and some positive constant $c_3^\star > 0$, which is independent of $\eps$. Here we have made use of the bounds $\| \rho^\eps \|_{\dot{H}^s} = \eps \| \phi^\eps \|_{\dot{H}^s} \leq \| \phi^\eps \|_{\dot{H}^s} \leq \| \phi^\eps \|_{H^s}$ for $\eps \in (0,1]$, which are derived from the expansion $\rho^\eps = 1 + \eps \phi^\eps$. Consequently, by substituting the inequalities \eqref{J2-mc} and \eqref{J-J2-mc} into \eqref{J-mc}, we deduce that
\begin{equation}\label{J-mc-bnd}
  \begin{aligned}
    \mathcal{J} \lesssim\, & \mathcal{Q}_{c_{\mathcal{J}}'} (\u) \big( 1 + \| \phi^\eps \|^{2s+2}_{H^s} + \| \u \|^{2s+2}_{H^s} + \| \na \d \|^{2s + 2}_{H^s} + \| \cd \|^{2s+2}_{H^s} \big) \\
    & \times \big( \| \phi^\eps \|_{H^s} + \| \u \|_{H^s} + \| \na \d \|_{H^s} + \| \cd \|_{H^s} \big) \\
    & \times \big( \| \na \u \|_{H^s} \| \phi^\eps \|_{\dot{H}^s} + \| \cd \|^2_{H^s} + \| \na \d \|^2_{\dot{H}^s} \big)
  \end{aligned}
\end{equation}
for all $\eps \in (0,1]$ and any multi-index $m \in \mathbb{N}^n$ with $1 \leq |m| \leq s$, where $c_{\mathcal{J}}' = \max \{ c_2^\star , c_3^\star \} > 0$.

From plugging the bounds \eqref{I-mc-bnd} and \eqref{J-mc-bnd} into \eqref{he}, summing up for all $1 \leq |m| \leq s$ and adding them into the inequality \eqref{be}, one immediately derives that
\begin{equation}\label{H-bnd}
  \begin{aligned}
    & \tfrac{1}{2} \tfrac{\dd}{\dd t} \big( \| \phi^\eps \|^2_{H^s_{p' (\rho^\eps)}} + \| \u \|^2_{H^s_{\rho^\eps}} + \| \cd \|^2_{H^s_{\rho^\eps}} + \ka \| \na \d \|^2_{H^s} \big) + \mu_1 \sum_{|m| \leq s} \| \d{}^\top (\p^m \A^\eps) \d \|^2_{L^2} \\
    & + \tfrac{1}{2} \mu_4 \| \na \u \|^2_{H^s} - \la_1 \sum_{|m| \leq s} \| \p^m \cd + (\p^m {\B}^\eps) \d + \tfrac{\la_2}{\la_1} (\p^m {\A}^\eps) \d \|^2_{L^2} \\
    & + (\tfrac{1}{2} \mu_4 + \xi ) \| \dv \u \|^2_{H^s} + ( \mu_5 + \mu_6 + \tfrac{\la_2^2}{\la_1} ) \sum_{|m| \leq s} \| (\p^m \A^\eps) \d \|^2_{L^2} \\
    \lesssim\, & \mathcal{Q}_{c'} (\u) \big( 1 + \| \phi^\eps \|^{2s+2}_{H^s} + \| \u \|^{2s+2}_{H^s} + \| \na \d \|^{2s + 2}_{H^s} + \| \cd \|^{2s+2}_{H^s} \big) \\
    & \times \big( \| \phi^\eps \|_{H^s} + \| \u \|_{H^s} + \| \na \d \|_{H^s} + \| \cd \|_{H^s} \big) \\
    & \times \big( \| \phi^\eps \|_{\dot{H}^s} + \| \u \|_{\dot{H}^s} + \| \na \d \|_{\dot{H}^s} + \| \cd \|_{H^s} \big) \\
    & \times \big( \| \na \u \|_{H^s} + \| \phi^\eps \|_{\dot{H}^s} + \| \na \d \|_{\dot{H}^s} + \| \cd \|_{H^s} \big)
  \end{aligned}
\end{equation}
for all $\eps \in (0,1]$, where $c' = \max \{ c_0, c'_{\mathcal{I}} , c'_{\mathcal{J}} \} > 0$. By employing Lemma \ref{lem3}, we easily obtain that
\begin{equation}\label{H-norm-relt}
  \begin{aligned}
    & \| \phi^\eps \|_{H^s} + \| \u \|_{H^s} + \| \cd \|_{H^s} \\
    \lesssim\, & \big( \| \tfrac{1}{p' (\rho^\eps)} \|^\frac{1}{2}_{L^\infty} + \| \tfrac{1}{ \rho^\eps } \|^\frac{1}{2}_{L^\infty} \big) \big( \| \phi^\eps \|_{H^s_{p'(\rho^\eps)}} + \| \u \|_{H^s_{\rho^\eps}} + \| \cd \|_{H^s_{\rho^\eps}} \big) \\
    \lesssim\, & \mathcal{Q}_{c^\star} (\u) \big( \| \phi^\eps \|_{H^s_{p'(\rho^\eps)}} + \| \u \|_{H^s_{\rho^\eps}} + \| \cd \|_{H^s_{\rho^\eps}} \big)
  \end{aligned}
\end{equation}
for some positive constant $c^\star$, independent of $\eps$. Let $c = c' + (2s+3) c^\star > 0$. Then the bounds \eqref{H-bnd} and \eqref{H-norm-relt} finish the proof.
\end{proof}

\subsection{Local well-posedness for the system \eqref{csys}} In this subsection, we will construct the local existence of the system \eqref{csys} with uniformly in $\eps$ small initial data by employing the nonlinear iteration method. More precisely, we will give the following result of local existence.

\begin{prop}[Local well-posedness]\label{Prop-Local}
	Let the integer $s > \tfrac{n}{2} + 1$ $(n=2,3)$, $0 < \eps \leq 1$, $0 < r_1 < \tfrac{1}{2}$, $r_2 > \tfrac{3}{2}$ and $\rho^\eps_0 = 1 + \eps \phi^\eps_0$. Then there exist constants $\delta_0$ and $ T \in (0,1)$, independent of $\eps$, such if
	\begin{equation}\label{IC-small-Engy-Loc}
	  \begin{aligned}
	   & |\d_0| = 1 \,, \quad \widetilde{\rm d}^\eps_0 \cdot \d_0 = 0 \,,\quad \| \phi^\eps_0 \|_{L^\infty} \leq \tfrac{1}{2} \,, \\
& \| \phi^\eps_0 \|^2_{H^s_{p' (\rho^\eps_0)}} + \| {\rm u}^\eps_0 \|^2_{H^s_{\rho^\eps_0}} + \| \widetilde{\rm d}^\eps_0 \|^2_{H^s_{\rho^\eps_0}} + \ka \| \na {\rm d}^\eps_0 \|^2_{H^s} \leq \tfrac{\delta_0}{2} \,,
	  \end{aligned}
	\end{equation}
	then the Cauchy problem \eqref{csys}-\eqref{civ} admits a unique classical solution $(\phi^\eps, \u , \d)$ satisfying
	\begin{equation*}
	  \begin{aligned}
	    & \phi^\eps \in L^\infty (0,T; H^s_{p'(\rho^\eps)}) \,, \ \u \in L^\infty (0,T; H^s_{\rho^\eps}) \cap L^2 (0,T; H^{s+1}) \,, \\
	    & \cd \in L^\infty (0,T; H^s_{\rho^\eps}) \,, \ \na \d \in L^\infty (0,T; H^s) \,, \ \rho^\eps \in L^\infty (0,T; L^\infty)
	  \end{aligned}
	\end{equation*}
	with uniform \emph{(}in $\eps \in (0,1]$\emph{)} energy bounds
	$$0 < r_1 \leq \rho^\eps (t,x) = 1 + \eps \phi^\eps (t,x) \leq r_2$$
	for all $ t \in [0,T] $ and
	\begin{equation*}
	  \begin{aligned}
	 &   \| \phi^\eps \|^2_{ L^\infty (0,T; H^s_{p'(\rho^\eps)}) } + \| \u \|^2_{ L^\infty (0,T; H^s_{\rho^\eps}) } + \| \cd \|^2_{ L^\infty (0,T; H^s_{\rho^\eps}) } \\
	 & \qquad \qquad \quad   + \ka \| \na \d \|^2_{ L^\infty (0,T; H^s) } + \tfrac{1}{2} \mu_4 \int_0^T \| \na \u \|^2_{H^s} \dd t \leq \delta_0 \,.
	  \end{aligned}
	\end{equation*}
\end{prop}

\begin{proof}
	First, we construct the approximate system of \eqref{csys} by iteration. More precisely, the iterative approximate equations are constructed as follows:
	\begin{equation}\label{Iter-Appro-Sys}
	  \left\{
	    \begin{array}{l}
	      \p_t \phi^{\eps, k+1} + {\rm u}^{\eps, k} \cdot \na \phi^{\eps, k+1} + \phi^{\eps, k+1} \dv {\rm u}^{\eps, k} + \tfrac{1}{\eps} \dv {\rm u}^{\eps, k+1} = 0 \,, \\[2mm]
	      \p_t {\rm u}^{\eps, k+1} + {\rm u}^{\eps, k} \cdot \na {\rm u}^{\eps, k+1} + \tfrac{1}{\eps} \tfrac{p' (\rho^{\eps, k})}{\rho^{\eps, k}} \na \phi^{\eps, k+1} = \tfrac{1}{\rho^{\eps, k}} \dv \big( \Sigma_1^{\eps, k+1} + \Sigma_2^{\eps, k} + \Sigma_3^{\eps, k+1} \big) \,, \\[2mm]
	      \p_t \dot{{\rm d}}^{\eps, k+1} + {\rm u}^{\eps, k} \cdot \na \dot{{\rm d}}^{\eps, k+1} = \tfrac{\ka}{\rho^{\eps, k}} \Delta {\rm d}^{\eps, k+1} + \tfrac{1}{\rho^{\eps, k}} \Gamma^{\eps, k+1} {\rm d}^{\eps, k+1} \\[2mm]
	      \qquad \qquad \,\;\quad \qquad \qquad \qquad + \tfrac{\la_1}{\rho^{\eps, k}} \big( \dot{{\rm d}}^{\eps, k+1} + {\B}^{\eps, k} {\rm d}^{\eps, k+1} \big) + \tfrac{\la_2}{\rho^{\eps, k}} {\A}^{\eps, k} {\rm d}^{\eps, k+1} \,, \\[2mm]
	      \big( \phi^{\eps, k+1} , {\rm u}^{\eps, k+1} , {\rm d}^{\eps, k+1} , \dot{{\rm d}}^{\eps, k+1} \big) |_{t=0} = \big( \phi^\eps_0 , \u_0 , \d_0 , \widetilde{\rm d}^\eps_0 \big) (x),
	    \end{array}
	  \right.
	\end{equation}
where
	\begin{equation*}
	  \begin{aligned}
	    \Sigma_1^{\eps, k+1} & : = \tfrac{1}{2} \mu_4 \big( \na {\rm u}^{\eps, k+1} + \na^\top {\rm u}^{\eps, k+1} \big) + \xi \dv {\rm u}^{\eps, k+1} \I \,, \\
	    \Sigma_2^{\eps, k} & : = \tfrac{1}{2} \ka |\na {\rm d}^{\eps, k}|^2 \I - \ka \na {\rm d}^{\eps, k} \odot \na {\rm d}^{\eps, k} \,, \\
	    \Sigma_3^{\eps, k+1} & : = \tilde{\sigma}_{\bm{\mu}} ({\rm u}^{\eps, k+1} , {\rm d}^{\eps, k}, \dot{{\rm d}}^{\eps, k}) \,,
	  \end{aligned}
	\end{equation*}
	and
	\begin{equation*}
	  \begin{aligned}
	   \, & \A^{\eps, k} = \tfrac{1}{2} \big( \na {\rm u}^{\eps, k} + \na^\top {\rm u}^{\eps, k} \big) \,,\;\; \;\; {\B}^{\eps, k} = \tfrac{1}{2} \big( \na {\rm u}^{\eps, k} - \na^\top {\rm u}^{\eps, k} \big) \,, \\
	    \, &\rho^{\eps, k} = 1 + \eps \phi^{\eps, k} \,,\;\; \;\; \dot{{\rm d}}^{\eps, k+1} = \p_t {\rm d}^{\eps, k+1} + {\rm u}^{\eps, k} \cdot \na {\rm d}^{\eps, k+1} \,, \\
 \,&\Gamma^{\eps, k+1} = \Gamma (\rho^{\eps, k}, {\rm u}^{\eps, k} , {\rm d}^{\eps, k+1}, \dot{{\rm d}}^{\eps, k+1}) = - \rho^{\eps, k} |\dot{{\rm d}}^{\eps, k+1}|^2 + \ka |\na {\rm d}^{\eps, k+1}|^2 - \la_2 { {\rm d}^{\eps,k+1} }^\top \A^{\eps,k} {\rm d}^{\eps,k+1} \,.
	  \end{aligned}
	\end{equation*}
 We start the iteration from $k = 0$ with
	\begin{equation*}
	  \big( \phi^{\eps, 0} , {\rm u}^{\eps, 0} , {\rm d}^{\eps, 0} , \dot{{\rm d}}^{\eps, 0} \big) (t,x) = \big( \phi^\eps_0 , \u_0 , \d_0 , \widetilde{\rm d}^\eps_0 \big) (x) \in \R \times \R^n \times \mathbb{S}^{n-1} \times \R^n
	\end{equation*}
	for all $(t,x) \in \R^+ \times \R^n$.
	
	It easily follows from Lemma 5.1 of \cite{JLT} that the following conclusions hold: {\em Suppose that the integer $s > \tfrac{n}{2} + 1$ and the initial data $\big( \phi^\eps_0 , \u_0 , \d_0 , \widetilde{\rm d}^\eps_0 \big) \in \R \times \R^n \times \mathbb{S}^{n-1} \times \R^n $  satisfy $ \widetilde{\rm d}^\eps_0 \cdot {\rm d}^\eps_0 = 0 $ and $\phi^\eps_0 , \u_0 , \na \d_0 , \widetilde{\rm d}^\eps_0 \in H^s$. Then, for all $k \geq 0$, there is a maximal time $T^\star_{\eps, k+1} > 0$ such that the iterative approximate system \eqref{Iter-Appro-Sys} admits a unique solution $( \phi^{\eps, k+1} , {\rm u}^{\eps, k+1} , {\rm d}^{\eps, k+1} , \dot{{\rm d}}^{\eps, k+1})$ satisfying $ | {\rm d}^{\eps, k+1} | = 1 $ and
		 $${\rm u}^{\eps, k+1} \in C(0, T^\star_{\eps, k+1}; H^s) \cap L^2 (0, T^\star_{\eps, k+1}; H^{s+1}) \textrm{ and } \phi^{\eps, k+1} , \na {\rm d}^{\eps, k+1} , \dot{{\rm d}}^{\eps, k+1} \in C(0, T^\star_{\eps, k+1} ; H^s) \,.$$}
	 We remark that $T^\star_{\eps, k+1} \leq T^\star_{\eps, k}$.
	
	 Second, we shall obtain the uniform (in $k \geq 0$) energy bound of the iterative approximate system \eqref{Iter-Appro-Sys} and the uniform (in $k \geq 0$ and $\eps \in (0,1]$) lower bound $T > 0$ of the time sequence $\{ T^\star_{\eps, k+1} \}$ given in the previous. We now introduce the following iterative approximate energy $\mathcal{E}_{s,k+1} (t)$ and energy dissipative rate $\mathcal{D}_{s,k+1} (t)$:
	 \begin{equation*}
	   \begin{aligned}
	     \mathcal{E}_{s,k+1} (t) = \| \phi^{\eps, k+1} \|^2_{H^s_{p' (\rho^{\eps, k})}} + \| {\rm u}^{\eps, k+1} \|^2_{H^s_{\rho^{\eps, k}}} + \| \dot{{\rm d}}^{\eps, k+1} \|^2_{H^s_{\rho^{\eps, k}}} + \ka \| \na {\rm d}^{\eps, k+1} \|^2_{H^s}
	   \end{aligned}
	 \end{equation*}
	 and
	   \begin{align*}
	     \mathcal{D}_{s, k+1} (t) =\, & \tfrac{1}{2} \mu_4 \| \na {\rm u}^{\eps, k+1} \|^2_{H^s} + ( \tfrac{1}{2} \mu_4 + \xi ) \| \dv {\rm u}^{\eps, k+1} \|^2_{H^s} \\
	     & - \la_1 \| \dot{{\rm d}}^{\eps, k+1} \|^2_{H^s} + \mu_1 \sum_{|m| \leq s} \| {\rm d}^{\eps, k}{}^\top ( \p^m \A^{\eps, k+1} ) {\rm d}^{\eps, k} \|^2_{L^2} \\
	     & - \la_1 \sum_{|m| \leq s} \big\| (\p^m {\B}^{\eps, k+1}) {\rm d}^{\eps, k} + \tfrac{\la_2}{\la_1} (\p^m {\A}^{\eps, k+1}) {\rm d}^{\eps, k} \big\|^2_{L^2} \\
	     & + ( \mu_5 + \mu_6 + \tfrac{\la_2^2}{\la_1} ) \sum_{|m| \leq s} \| (\p^m \A^{\eps, k+1}) {\rm d}^{\eps, k} \|^2_{L^2}.
	   \end{align*}
	 Then, following the similar arguments to that in Lemma \ref{Lmm-Apriori-Est}, we deduce that
\begin{align}\label{indc-energy-1}
    & \tfrac{1}{2} \tfrac{d}{dt} \mathcal{E}_{s, k+1} + \mathcal{D}_{s, k+1} \nonumber\\
    \lesssim\, & F( |\rho^{\eps, k}|, |\rho^{\eps, k}|^{-1}, |\rho^{\eps, k-1}|^{-1}, |\rho^{\eps, k-2}|^{-1} ) \big ( 1 + \mathcal{E}_{s,k-1}^{\frac12} \big ) \big ( 1 + \mathcal{D}_{s,k}^{\frac12} \big ) \big ( 1 + \mathcal{E}_{s,k}^{\frac{s+5}{2}} \big ) \big ( 1 + \mathcal{E}_{s,k+1}^{\frac{5}{2}} \big ) \nonumber\\
    & + F ( |\rho^{\eps, k}|, |\rho^{\eps, k}|^{-1}, |\rho^{\eps, k-1}|^{-1}, |\rho^{\eps, k-2}|^{-1} ) \big ( \mathcal{E}_{s,k}^{\frac12} + \mathcal{E}_{s,k}^{\frac{s+4}{2}} \big ) \big ( 1 + \mathcal{E}_{s,k+1}^{\frac12} \big ) \mathcal{D}_{s,k+1}^{\frac12}  
\end{align}
for all $\eps \in (0,1]$ and $t \in [0, T^\star_{\eps, k+1}]$. Here, $ F (\cdot,\cdot,\cdot,\cdot) $ is the polynomial function of $ |\rho^{\eps, k}| $, $ |\rho^{\eps, k}|^{-1} $, $ |\rho^{\eps, k-1}|^{-1} $ and $ |\rho^{\eps, k-2}|^{-1} $, which takes the one form as
\begin{equation}\label{F(t)}
\begin{aligned}
    & F( |\rho^{\eps, k}|, |\rho^{\eps, k}|^{-1}, |\rho^{\eps, k-1}|^{-1}, |\rho^{\eps, k-2}|^{-1} ) \\
     &\qquad = 1 + |\rho^{\eps, k}|^{n_1} + |\rho^{\eps, k}|^{-n_2} + |\rho^{\eps, k-1}|^{-n_3} + |\rho^{\eps, k-2}|^{-n_4} = : F (t) \, 
\end{aligned}
\end{equation}
for some positive integer $n_1$, $n_2$, $n_3$ and $n_4$.

In view of Young's inequality, the term $ \mathcal{D}_{s, k+1}^{\frac12} $ on the right-hand side of \eqref{indc-energy-1} can be controlled by its left dissipation term $\mathcal{D}_{s, k+1}$. Thus, we have
\begin{equation}\label{indc-energy-2}
	   \begin{aligned}
    \tfrac{d}{dt} \mathcal{E}_{s, k+1} (t) + \mathcal{D}_{s, k+1} (t) \leq C_\star \mathcal{K}_{s,k} (t) \big ( 1 + \mathcal{E}_{s, k+1} (t) \big )^{3}
\end{aligned}
\end{equation}
for some positive constant $C_\star$, where
\begin{equation}\label{K(t)}
\begin{aligned}
    \mathcal{K}_{s,k} (t) = ( F (t) + | F(t) |^2 ) \big ( 1 + \mathcal{E}_{s, k-1}^{\frac{1}{2}} (t) \big ) \big ( 1 + \mathcal{E}_{s, k}^{s+4} (t) \big ) \big ( 1 + \mathcal{D}_{s, k}^{\frac{1}{2}} (t) \big ) \,.
\end{aligned}
\end{equation}
Solving the ordinary differential inequality \eqref{indc-energy-2} gives
\begin{align*}
    \mathcal{E}_{s, k+1} (t) \leq \Big [ ( \mathcal{E}_{s, k+1} (0) + 1 )^{-2} - 2 C_\star \int_0^t \mathcal{K}_{s,k} (\tau) \dd \tau \Big ]^{-\frac12} - 1.
\end{align*}
Putting the above inequality into \eqref{indc-energy-2}, and integrating on $[0,t]$, we thereby derive that
\begin{equation}\label{Indc-energy}
	   \begin{aligned}
     & \mathcal{E}_{s, k+1} (t) + \int_0^t \mathcal{D}_{s, k+1} (\tau) \dd \tau \\
     \leq \,& \mathcal{E}_{s, k+1} (0) + C_\star \big ( 1 + \sup_{0 \leq \tau \leq t} \mathcal{E}_{s, k+1} (\tau) \big )^{3} \int_0^t \mathcal{K}_{s,k} (\tau) \dd \tau \\
     \leq\, & \mathcal{E}_{s, k+1} (0) + C_\star \Big [ ( \mathcal{E}_{s, k+1} (0) + 1 )^{-2} - 2 C_\star \int_0^t \mathcal{K}_{s,k} (\tau) \dd \tau \Big ]^{-\frac{3}{2}} \int_0^t \mathcal{K}_{s,k} (\tau) \dd \tau \,.
\end{aligned}
\end{equation}

Noticing $ \rho^{\eps, k+1} = 1 + \eps \phi^{\eps,k+1} $, the continuity equation in the approximate system \eqref{Iter-Appro-Sys} reads as
\begin{align*}
    \p_t \rho^{\eps, k+1} + {\rm u}^{\eps, k} \cdot \nabla \rho^{\eps, k+1} + \rho^{\eps, k+1} \dv {\rm u}^{\eps, k} = \dv ( {\rm u}^{\eps, k} - {\rm u}^{\eps, k+1} ) \,.
\end{align*}
Using the characteristic method to solve the above equation, we have
\begin{align*}
    \rho^{\eps, k+1} = \exp \Big ( - \int_0^t \dv {\rm u}^{\eps, k} \dd \tau \Big ) \bigg [ \rho^{\eps, k+1}_0 + \int_0^t \exp \Big ( - \int_0^\tau \dv {\rm u}^{\eps, k} \dd \tau \Big ) \dv ( {\rm u}^{\eps, k} - {\rm u}^{\eps, k+1} ) \dd \tau \bigg ] \,.
\end{align*}
Recalling the notation $ \mathcal{Q}_c( {\rm u} )$ defined by \eqref{Q(u)}, it infers that
\begin{align*}
    & \mathcal{Q}_{-1}( {\rm u}^{\eps, k} ) \Big [ \tfrac{1}{2} - \mathcal{Q}_{1}( {\rm u}^{\eps, k} ) \Big ( \int_0^t \| \dv ( {\rm u}^{\eps, k} - {\rm u}^{\eps, k+1} ) \|_{L^\infty} \Big) \Big ] \\
    & \qquad \leq \rho^{\eps, k+1} \leq \mathcal{Q}_{1}( {\rm u}^{\eps, k} ) \Big [ \tfrac{3}{2} + \mathcal{Q}_{1}( {\rm u}^{\eps, k} ) \Big ( \int_0^t \| \dv ( {\rm u}^{\eps, k} - {\rm u}^{\eps, k+1} ) \|_{L^\infty} \Big) \Big ] \,,
\end{align*}
where we have utilized $\rho^\eps_0(x) = 1 + \eps \phi^\eps_0 (x) \in [ \tfrac{1}{2} , \tfrac{3}{2} ]$ for all $\eps \in (0,1]$ and $x \in \R^n$, which is derived from the initial condition $\| \phi^\eps_0 \|_{L^\infty} \leq \tfrac{1}{2}$ given in \eqref{IC-small-Engy-Loc}. Then, by the Sobolev embedding $ H^{s-1} \hookrightarrow L^\infty $ $ ( s > \frac{n}{2} + 1 ) $ with the generic constant $ \widehat{C} $ and H\"{o}lder's inequality, we get
	\begin{equation}\label{Indc-rho}
	  \begin{aligned}
	    & \bigg \{ \tfrac{1}{2} - \widehat{C} t^{\frac{1}{2}} \Big [ ( \int_0^t \mathcal{D}_{s,k} \dd \tau )^{\frac{1}{2}} + ( \int_0^t \mathcal{D}_{s,k+1} \dd \tau )^{\frac{1}{2}} \Big ] \\
	    &  \times \exp \Big ( \widehat{C} t^{\frac{1}{2}} ( \int_0^t \mathcal{D}_{s,k} \dd \tau )^{\frac{1}{2}} \Big ) \bigg \} \exp \big( - \widehat{C} t^{\frac{1}{2}} ( \int_0^t \mathcal{D}_{s,k} \dd \tau )^\frac{1}{2} \big) \\
	 \leq  \, & \rho^{\eps, k+1} (t,x) \leq \bigg \{ \tfrac{3}{2} + \widehat{C} t^{\frac{1}{2}} \Big [ ( \int_0^t \mathcal{D}_{s,k} \dd \tau )^{\frac{1}{2}} + ( \int_0^t \mathcal{D}_{s,k+1} \dd \tau )^{\frac{1}{2}} \Big ] \\
	    &  \times \exp \Big ( \widehat{C} t^{\frac{1}{2}} ( \int_0^t \mathcal{D}_{s,k} \dd \tau )^{\frac{1}{2}} \Big ) \bigg \} \exp \Big( \widehat{C} t^{\frac{1}{2}} ( \int_0^t \mathcal{D}_{s,k} \dd \tau )^\frac{1}{2} \Big)
	  \end{aligned}
	\end{equation}
for all $(t,x) \in [0,T^\star_{\eps, k+1}] \times \R^n$ and $\eps \in (0,1]$.

With a similar process to \eqref{Indc-rho}, we have
\begin{equation}\label{rho-k}
\begin{aligned}
    | \rho^{\eps, k} | \leq & \exp \Big( \widehat{C} t^{\frac{1}{2}} \big ( \int_0^t \mathcal{D}_{s,k} \dd \tau \big)^{\frac{1}{2}} \Big)
    \Big[ \tfrac32 + \exp \Big( \widehat{C} t^{\frac{1}{2}} \big ( \int_0^t \mathcal{D}_{s,k} \dd \tau \big)^{\frac{1}{2}} \Big) \\
     & \quad\quad\quad\quad \times \widehat{C} t^{\frac12} \Big( ( \int_0^t \mathcal{D}_{k-1} \dd \tau )^{\frac12} + ( \int_0^t \mathcal{D}_{k} \dd \tau )^{\frac12} \Big ) \Big ]
\end{aligned}
\end{equation}
and
\begin{equation}\label{rho-i}
\begin{aligned}
    | \rho^{\eps, i} |^{-1} \leq & \exp \Big( \widehat{C} t^{\frac{1}{2}} \big ( \int_0^t \mathcal{D}_{s,k} \dd \tau \big)^{\frac{1}{2}} \Big)
    \Big[ \tfrac12 - \exp \Big( \widehat{C} t^{\frac{1}{2}} \big ( \int_0^t \mathcal{D}_{s,k} \dd \tau \big)^{\frac{1}{2}} \Big) \\
     & \quad\quad\quad\quad \times \widehat{C} t^{\frac12} \Big( ( \int_0^t \mathcal{D}_{s, k-1} \dd \tau )^{\frac12} + ( \int_0^t \mathcal{D}_{s, k} \dd \tau )^{\frac12} \Big ) \Big ]^{-1}  \,,
\end{aligned}
\end{equation}
for $ i = k, k-1, k-2 $.
	
	Now we claim that {\em for any fixed $r_1 \in (0, \tfrac{1}{2})$ and $r_2 \in (\tfrac{3}{2} , + \infty)$, there exist $\delta_0, T \in (0,1)$, independent of $\eps \in (0,1]$ and integer $k \geq 0$, such that if the initial data \eqref{civ} satisfies \eqref{IC-small-Engy-Loc} and
	\begin{equation}\label{Induction-Assumption}
	  \begin{aligned}
	    \mathcal{E}_{s,i} (t) + \int_0^t \mathcal{D}_{s,i} (\tau) \dd \tau \leq \delta_0
	  \end{aligned}
	\end{equation}
	for all integer $i \leq k$, $t \in [0,T]$ and $x \in \R^n$, then
	\begin{equation}\label{Claim-induction}
	  \begin{aligned}
	     \mathcal{E}_{s,k+1} (t) + \int_0^t \mathcal{D}_{s,k+1} (\tau) \dd \tau \leq \delta_0
	  \end{aligned}
	\end{equation}
	and
	\begin{equation}\label{Claim-induction-rho}
	  \begin{aligned}
	    r_1 \leq \rho^{\eps, k+1} (t,x) \leq r_2
	  \end{aligned}
	\end{equation}
    for all $(t,x) \in [0,T] \times \R^n$.
}

    Once the claim \eqref{Claim-induction} holds, the induction principle tells us that we obtain a uniform (in $k \geq 0$) energy bounds
    \begin{equation*}
      \begin{aligned}
       \mathcal{E}_{s,k} (t) + \int_0^t \mathcal{D}_{s,k} (t) \dd t \leq \delta_0 \,, \ \ \ r_1 \leq \rho^{\eps, k} (t,x) \leq r_2 \,
      \end{aligned}
    \end{equation*}
    for all $ t \in [0,T] $ and $ x \in \mathbb{R}^n $. Then, the standard compactness arguments can finish the proof of Proposition \ref{Prop-Local}.

    We next prove the claims \eqref{Claim-induction} and \eqref{Claim-induction-rho}. Indeed, we easily know from the initial condition \eqref{IC-small-Engy-Loc} that for all $k \geq 0$,
    \begin{equation}\label{IC-Engy-2}
      \begin{aligned}
        \mathcal{E}_{s,k} (0) = \| \phi^\eps_0 \|^2_{H^s_{p' (\rho^\eps_0)}} + \| {\rm u}^\eps_0 \|^2_{H^s_{\rho^\eps_0}} + \| \widetilde{\rm d}^\eps_0 \|^2_{H^s_{\rho^\eps_0}} + \ka \| \na {\rm d}^\eps_0 \|^2_{H^s} \leq \tfrac{\delta_0}{2} \,.
      \end{aligned}
    \end{equation}

Under the induction assumptions \eqref{Induction-Assumption}, combining the definitions of $ F(t) $ and $ \mathcal{K}_{s,k} (t) $ in \eqref{F(t)} and \eqref{K(t)}, it deduces from \eqref{rho-k} and \eqref{rho-i} that
\begin{align*}
    & F( |\rho^{\eps, k}|, |\rho^{\eps, k}|^{-1}, |\rho^{\eps, k-1}|^{-1}, |\rho^{\eps, k-2}|^{-1} ) \\
    \leq\, & C \Big \{ 1 + \exp \big( n_1 \widehat{C} \delta_0^{\frac{1}{2}} t^{\frac{1}{2}} \big) \Big( \tfrac32 + 2 \widehat{C} \delta_0^{\frac{1}{2}} \exp \big( \widehat{C} \delta_0^{\frac{1}{2}} t^{\frac{1}{2}} \big) t^{\frac12} \Big )^{n_1} \\
    & + \Big [ \exp \big( \widehat{C} \delta_0^{\frac{1}{2}} t^{\frac{1}{2}} \big) \Big ( \tfrac12 - 2 \widehat{C} \delta_0^{\frac{1}{2}} \exp \big( \widehat{C} \delta_0^{\frac{1}{2}} t^{\frac{1}{2}} \big) t^{\frac12} \Big )^{-1} \Big ]^{ \max\{n_2, n_3, n_4\} } \Big \} =: \mathcal{F} (t)
\end{align*}
for some constant $C$, and
\begin{align*}
    \int_0^t \mathcal{K}_{s,k} (\tau) \dd \tau \leq ( \mathcal{F} (t) + | \mathcal{F} (t) |^2 ) ( 1 + \delta_0^{\frac12} t ) ( 1 + \delta_0^{s+4} t ) ( t + \delta_0^{\frac12} t^{\frac12} ) = : \mathcal{G} (t)\,,
\end{align*}
where $ \mathcal{F} (t)$ and $ \mathcal{G} (t) $ are continuous and strictly increasing in $ [0, t_0] $ for some finite and positive constant $ t_0 $ independent of $ \eps $, and $ \mathcal{G} (0) = 0 $.

By the initial bound \eqref{IC-Engy-2} and the induction energy inequality \eqref{Indc-energy}, we arrive at
\begin{equation*}
	   \begin{aligned}
     & \mathcal{E}_{s, k+1} (t) + \int_0^t \mathcal{D}_{s, k+1} (\tau) d \tau \\
    &\qquad  \leq \tfrac{\delta_0}{2} + C_\star \Big [ \big ( \tfrac{\delta_0}{2} + 1 \big )^{-2} - 2 C_\star \mathcal{G} (t) \Big ]^{-\frac{3}{2}} \mathcal{G} (t) = : \mathcal{H} (t) \,.
\end{aligned}
\end{equation*}
Here, $ \mathcal{H} (t) $ is continuous and increasing in $ [0, t_1] $ for some positive constant $ t_1 ( < t_0 ) $ independent of $ \eps $, and $ \mathcal{H} (0) = \frac{\delta_0}{2} $. Thus, there exists a $ t_2 $ with $ 0 < t_2 \leq t_1 $, independent of $ \eps $, such that
    \begin{equation}\label{energy-k+1}
      \begin{aligned}
       \mathcal{E}_{s,k+1} (t) + \int_0^t \mathcal{D}_{s,k+1} (\tau) \dd \tau \leq \delta_0 \,
      \end{aligned}
    \end{equation}
for all $ t \in [0, t_2] $. Next, by the bounds \eqref{Induction-Assumption} and \eqref{energy-k+1}, we deduce from \eqref{Indc-rho} that
    \begin{equation*}
      \begin{aligned}
        & \underbrace{ \Big \{ \tfrac{1}{2} - 2 \widehat{C} \delta_0^{\frac{1}{2}} t^{\frac{1}{2}} \exp \big ( \widehat{C} \delta_0^{\frac{1}{2}} t^{\frac{1}{2}} \big ) \Big \} \exp \big( - \widehat{C} \delta_0^{\frac{1}{2}} t^{\frac{1}{2}} \big) }_{g(t)} \\
	    & \qquad  \leq \rho^{\eps, k+1} (t,x) \leq \underbrace{ \Big \{ \tfrac{3}{2} + 2 \widehat{C} \delta_0^{\frac{1}{2}} t^{\frac{1}{2}} \exp \big ( \widehat{C} \delta_0^{\frac{1}{2}} t^{\frac{1}{2}} \big ) \Big \} \exp \big( \widehat{C} \delta_0^{\frac{1}{2}} t^{\frac{1}{2}} \big) }_{h(t)} \,,
      \end{aligned}
    \end{equation*}
    where $g(t)$ is continuous and strictly decreasing in $\R^+$ with $g(0) = \tfrac{1}{2}$ and $h(t)$ is continuous and strictly increasing in $\R^+$ with $h(0) = \tfrac{3}{2}$. Then, it is easy to know that there exists a $ t_3 $ with $ 0 < t_3 \leq t_2 $, independent of $ \eps $, such that
    $$ r_1 \leq g(t) \leq \rho^{\eps, k+1} (t,x) \leq h(t) \leq r_2 $$
    for all $t \in [0,t_3]$ and $ x \in \mathbb{R}^n $. We thereby take $T$ such that $ 0 < T < \min \{ 1, t_3 \} $. As a result, the claims \eqref{Claim-induction} and \eqref{Claim-induction-rho} hold. Thus, the proof 
     is completed.
\end{proof}

\subsection{Global uniform energy bound \eqref{Unif-Bnd-1}: global well-posedness}

In this subsection, we will globally extend the solution constructed in Proposition \ref{Prop-Local} by seeking some additional dissipative structures on $\phi^\eps$ and $\d$. Consequently, we can derive a uniform global energy bound. We emphasize that, based on the solution to \eqref{csys}-\eqref{civ} constructed in Proposition \ref{Prop-Local}, the values of density function $\rho^\eps (t,x) = 1 + \eps \phi^\eps (t,x)$ on $(t,x) \in [0,T] \times \R^n$ are ranged in $[r_1, r_2]$ for some positive constants $0 < r_1 < \tfrac{1}{2} < \tfrac{3}{2} < r_2$, where $T > 0$ is given in Proposition \ref{Prop-Local}. Consequently, for any $\alpha \in \R^+$, there are two constants $r_1$ and $r_2$ such that
\begin{equation}\label{rho-bnd}
  \begin{aligned}
    r_1^\alpha \leq \big( \rho^\eps (t,x) \big)^\alpha \leq r_2^\alpha
  \end{aligned}
\end{equation}
for all $(t,x) \in [0,T] \times \R^n$ and $0 < \eps \leq 1$.

For any $\eta \in (0,1)$, we define the following so-called {\em instant energy functional} $\mathscr{E}_{s, \eta} (\phi^\eps, \u, \d)$:
\begin{equation}\label{Es-Inst}
\begin{aligned}
\mathscr{E}_{s, \eta} (\phi^\eps, \u, \d) =\, & \mathcal{E}_s (\phi^\eps, \u, \d) + \eps \eta \| \u + \na \phi^\eps \|^2_{H^{s-1}} + \eta \| \cd + \d \|^2_{\dot{H}^s} \\
& - \eps \eta \| \u \|^2_{H^{s-1}} - \eps \eta \| \na \phi^\eps \|^2_{H^{s-1}} - \eta \| \cd \|^2_{\dot{H}^s} - \eta \| \d \|^2_{\dot{H}^s} \,,
\end{aligned}
\end{equation}
and the so-called {\em instant energy dissipative rate functional} $\mathscr{D}_{s, \eta} (\phi^\eps, \u, \d)$:
\begin{equation}\label{Ds-Inst}
\begin{aligned}
\mathscr{D}_{s, \eta} (\phi^\eps, \u, \d) =\, &\mathcal{D}_s (\u, \d) + \tfrac{1}{2} \eta \| \na \phi^\eps \|^2_{H^{s-1}_{w(\rho^\eps)}} + \tfrac{3}{4} \ka \eta \| \na \d \|^2_{\dot{H}^s_{1/\rho^\eps}} \\
& - \eta ( C + C_0 ) \Big( \| \na \u \|^2_{H^s} + \sum_{|m| \leq s} \big\| \p^m \cd + (\p^m \B^\eps) \d + \tfrac{\la_2}{\la_1} ( \p^m \A^\eps ) \d \big\|^2_{L^2} \Big) \,,
\end{aligned}
\end{equation}
where $w(\rho^\eps) = \tfrac{1}{\rho^\eps} p' (\rho^\eps)$ with $\rho^\eps = 1 + \eps \phi^\eps$ and $C, C_0 > 0$ are some fixed constants, independent of $\eps \in (0,1]$. Notice that the above instant functionals may be not positive for all $ \eta \in (0,1] $. However, one can derive the following lemma.

\begin{lem}\label{Lmm-Inst-Eneg}
	There is a small constant $\eta_0 \in (0,1)$, independent of $\eps \in (0,1]$, such that $\mathscr{E}_{s, \eta_0} (\phi^\eps, \u, \d) \geq 0$ and $\mathscr{D}_{s, \eta_0} (\phi^\eps, \u, \d) \geq 0$. Moreover,
	\begin{equation*}
	\begin{aligned}
	\mathscr{E}_{s, \eta_0} (\phi^\eps, \u, \d) \approx \mathcal{E}_{s} (\phi^\eps, \u, \d) \,, \quad \mathscr{D}_{s, \eta_0} (\phi^\eps, \u, \d) \approx \mathbb{D}_{s} (\phi^\eps, \u, \d) \,,
	\end{aligned}
	\end{equation*}
	where the global energy dissipative rate functional $\mathbb{D}_{s} (\phi^\eps, \u, \d)$ is defined as follows:
	\begin{equation*}
	  \begin{aligned}
	    \mathbb{D}_s (\phi^\eps, \u , \d ) = \| \na \phi^\eps \|^2_{H^{s-1}_{w(\rho^\eps)}} + \| \na \d \|^2_{\dot{H}^s_{1/\rho^\eps}} + \mathcal{D}_s (\u, \d) \,.
	  \end{aligned}
	\end{equation*}
	Here the energy dissipative rate $\mathcal{D}_s (\u, \d)$ is given in \eqref{Ds-local}.
\end{lem}

\begin{proof}
	We first notice that
	\begin{equation*}
	  \begin{aligned}
	    & \eta \Big|\eps \| \u \! + \na \phi^\eps \|^2_{H^{s-1}} \! + \| \cd \! + \d \|^2_{\dot{H}^s} \! - \eps \| \u \|^2_{H^{s-1}} \! - \eps \| \na \phi^\eps \|^2_{H^{s-1}} \! - \eta \| \cd \|^2_{\dot{H}^s} \! - \eta \| \d \|^2_{\dot{H}^s} \Big| \\
	    \leq\, & 3 \eta \big( \| \u \|^2_{H^{s-1}} + \| \na \phi^\eps \|^2_{H^{s-1}} + \| \cd \|^2_{\dot{H}^s} + \| \d \|^2_{\dot{H}^s} \big) \\
	    \leq\, & C_3 \eta \big( \| \u \|^2_{H^s_{\rho^\eps}} + \| \phi^\eps \|^2_{H^s_{p'(\rho^\eps)}} + \| \cd \|^2_{H^s_{\rho^\eps}} + \ka \| \na \d \|^2_{H^s} \big) \\
	    =\, & C_3 \eta \mathcal{E}_s (\phi^\eps, \u , \d)
	  \end{aligned}
	\end{equation*}
	for some constant $C_3 > 0$, independent of $\eps$, and for all $\eta \in (0,1)$, $0 < \eps \leq 1$, where the last inequality is implied by the bound \eqref{rho-bnd}. We thereby have
	\begin{equation*}
	  \begin{aligned}
	    ( 1 - C_3 \eta ) \mathcal{E}_s (\phi^\eps, \u , \d) \leq \mathscr{E}_{s, \eta} (\phi^\eps, \u, \d) \leq ( 1 + C_3 \eta ) \mathcal{E}_s (\phi^\eps, \u , \d) \,.
	  \end{aligned}
	\end{equation*}
	If $1 - C_3 \eta > 0$, i.e. $0 < \eta < \min\{ 1 , \tfrac{1}{C_3} \}$, then $ \mathscr{E}_{s, \eta} (\phi^\eps, \u, \d) \geq  ( 1 - C_3 \eta ) \mathcal{E}_s (\phi^\eps, \u , \d) \geq 0 $ and $ \mathscr{E}_{s, \eta} (\phi^\eps, \u, \d) \approx \mathcal{E}_s (\phi^\eps, \u , \d) $.
	
	We now consider the instant energy dissipative rate functional $\mathscr{D}_{s,\eta} (\phi^\eps, \u, \d)$. Observing that
(recalling that  $\mu_4>0$ and $\lambda_1<0$)
	  \begin{align*}
	    & \| \na \u \|^2_{H^s} + \sum_{|m| \leq s} \big\| \p^m \cd + (\p^m \B^\eps) \d + \tfrac{\la_2}{\la_1} ( \p^m \A^\eps ) \d \big\|^2_{L^2} \\
	    \leq\, & ( \tfrac{2}{\mu_4} - \tfrac{1}{\la_1} ) \Big( \tfrac{1}{2} \mu_4 \| \na \u \|^2_{H^s} - \la_1 \sum_{|m| \leq s} \big\| \p^m \cd + (\p^m \B^\eps) \d + \tfrac{\la_2}{\la_1} ( \p^m \A^\eps ) \d \big\|^2_{L^2} \Big) \\
	    \leq\, & ( \tfrac{2}{\mu_4} - \tfrac{1}{\la_1} ) \mathcal{D}_s (\u, \d) \,,
	  \end{align*}
	one has
	\begin{equation*}
	  \begin{aligned}
	   & \Big( 1 - (C+C_0) ( \tfrac{2}{\mu_4} - \tfrac{1}{\la_1} ) \eta \Big) \mathcal{D}_s (\u, \d) + \tfrac{1}{2} \eta \| \na \phi^\eps \|^2_{H^{s-1}_{w(\rho^\eps)}} + \tfrac{3}{4} \ka \eta \| \na \d \|^2_{\dot{H}^s_{1/\rho^\eps}} \\
	   &\qquad \qquad  \leq \mathscr{D}_{s,\eta} (\phi^\eps, \u, \d) \leq \mathcal{D}_s (\u, \d) + \tfrac{1}{2} \eta \| \na \phi^\eps \|^2_{H^{s-1}_{w(\rho^\eps)}} + \tfrac{3}{4} \ka \eta \| \na \d \|^2_{\dot{H}^s_{1/\rho^\eps}} \,.
	  \end{aligned}
	\end{equation*}
	Once $ 1 - (C+C_0) ( \tfrac{2}{\mu_4} - \tfrac{1}{\la_1} ) \eta > 0 $, i.e. $0 < \eta < \min \bigg\{ 1, \tfrac{1}{ (C+C_0) ( \tfrac{2}{\mu_4} - \tfrac{1}{\la_1} ) } \bigg\}$, we have $ \mathscr{D}_{s,\eta} (\phi^\eps, \u, \d) \geq 0 $. Moreover, it is easy to see that $ \mathscr{D}_{s, \eta} (\phi^\eps, \u, \d) \approx \mathbb{D}_{s} (\phi^\eps, \u, \d) $. Consequently, we can take
	\begin{equation*}
	  \eta_0 = \min \bigg\{ 1, \tfrac{1}{C_3}, \tfrac{1}{ (C+C_0) ( \tfrac{2}{\mu_4} - \tfrac{1}{\la_1} ) } \bigg\} \in (0, 1) \,,
	\end{equation*}
	and the proof of Lemma \ref{Lmm-Inst-Eneg} is completed.
\end{proof}

Now, we derive the following proposition.

\begin{prop}\label{Prop-Global}
	Let $0 < \eps \leq 1$. Assume that $(\phi^\eps, \u, \d)$ is the solution on $[0,T]$ to the Cauchy problem \eqref{csys}-\eqref{civ} constructed in Proposition \ref{Prop-Local}. Then there is a constant $C_4 > 0$, independent of $\eps$, such that
	\begin{equation*}
	  \begin{aligned}
	    \tfrac{1}{2} \tfrac{\dd}{\dd t} \mathscr{E}_{s, \eta_0} (\phi^\eps, \u, \d) + \mathscr{D}_{s, \eta_0} (\phi^\eps, \u, \d) \leq C_4 \mathscr{E}_{s, \eta_0}^\frac{1}{2} (\phi^\eps, \u , \d) \mathscr{D}_{s, \eta_0} (\phi^\eps, \u, \d)
	  \end{aligned}
	\end{equation*}
	for all $t \in [0,T]$ and $\eps \in (0,1]$, where $\eta_0 > 0$ is given in Lemma \ref{Lmm-Inst-Eneg}.
\end{prop}

\begin{proof}

We now prove this proposition by three steps.

\vspace*{1mm}

 \emph{Step 1. Estimate the dissipation of $\pe$.} For all $|m| \leq s-1$, making use of the derivative operator $\p^m$ to the second equation of \eqref{csys}, and taking the inner product with $ \eps \pa\na\pe$, we have
\begin{equation*}
  \begin{aligned}
    & \eps \lt \p_t \p^m \u , \na \p^m \phi^\eps \rt + \| \na \p^m \phi^\eps \|^2_{L^2_{w (\rho^\eps)}} \\
    = & - \eps \lt \p^m (\u \cdot \na \u) , \na \p^m \phi^\eps \rt - \big\lt \big[ \p^m , w(\rho^\eps) \na \big] \phi^\eps , \na \p^m \phi^\eps \big\rt \\
    & + \eps \big\lt \p^m \big[ \tfrac{1}{\rho^\eps} \dv ( \Sigma_1^\eps + \Sigma_2^\eps + \Sigma_3^\eps ) \big] , \na \p^m \phi^\eps \big\rt \,,
  \end{aligned}
\end{equation*}
where $w(\rho^\eps) = \tfrac{1}{\rho^\eps} p' (\rho^\eps)$. From the first $\phi^\eps$-equation of \eqref{csys}, we derive that
\begin{equation*}
  \begin{aligned}
    \eps \lt \p_t \p^m \u , \na \p^m \phi^\eps \rt =\, & \tfrac{\dd}{\dd t} \lt \p^m \u , \eps \na \p^m \phi^\eps \rt - \lt \p^m \u , \eps \na \p^m \p_t \phi^\eps \rt \\
    =\, & \tfrac{\dd}{\dd t} \lt \p^m \u , \eps \na \p^m \phi^\eps \rt - \lt \dv \p^m \u , \eps \p^m (\u \cdot \na \phi^\eps) \rt \\
    & - \lt \dv \p^m \u , \eps \p^m ( \phi^\eps \dv \u ) \rt - \| \dv \p^m \u \|^2_{L^2} \,.
  \end{aligned}
\end{equation*}
Consequently, we have that for $0 < \eps \leq 1$ and $|m| \leq s - 1$,
\begin{equation}\label{Dissp-phi-1}
  \begin{aligned}
    & \eps \tfrac{\dd}{\dd t} \lt \p^m \u , \na \p^m \phi^\eps \rt + \| \na \p^m \phi^\eps \|^2_{L^2_{w(\rho^\eps)}} - \| \dv \p^m \u \|^2_{L^2} \\
    =\, & \underbrace{ \lt \dv \p^m \u , \eps \p^m (\u \cdot \na \phi^\eps) \rt }_{O_1} +  \underbrace{ \lt \dv \p^m \u , \eps \p^m ( \phi^\eps \dv \u ) \rt }_{O_2} \\
    &  \underbrace{ - \eps \lt \p^m (\u \cdot \na \u) , \na \p^m \phi^\eps \rt }_{O_3} \  \underbrace{ - \big\lt \big[ \p^m , w(\rho^\eps) \na \big] \phi^\eps , \na \p^m \phi^\eps \big\rt }_{O_4} \\
    & +  \underbrace{ \eps \big\lt \p^m \big( \tfrac{1}{\rho^\eps} \dv \Sigma_1^\eps \big) , \na \p^m \phi^\eps \big\rt }_{O_5} +  \underbrace{ \eps \big\lt \p^m \big( \tfrac{1}{\rho^\eps} \dv \Sigma_2^\eps \big) , \na \p^m \phi^\eps \big\rt }_{O_6} \\
    & +  \underbrace{ \eps \big\lt \p^m \big( \tfrac{1}{\rho^\eps} \dv \Sigma_3^\eps \big) , \na \p^m \phi^\eps \big\rt }_{O_7} \,.
  \end{aligned}
\end{equation}
Now we estimate terms $O_i$ for $1 \leq i \leq 7$ in \eqref{Dissp-phi-1}. First, from the Moser-type calculus inequalities in Lemma \ref{lem1} and the bound \eqref{rho-bnd}, one can easily derive that for all $0 < \eps \leq 1$,
\begin{equation}\label{O1O2O3}
  \begin{aligned}
    O_1 \lesssim\, & \| \na \u \|_{H^s} \| \na \phi^\eps \|_{H^{s-1}} \| \u \|_{H^s} \lesssim \| \na \u \|_{H^s} \| \na \phi^\eps \|_{H^{s-1}_{w(\rho^\eps)}} \| \u \|_{H^s_{\rho^\eps}} \,, \\
    O_2 \lesssim\, & \| \na \u \|_{H^s}^2 \| \phi^\eps \|_{H^s} \lesssim \| \na \u \|^2_{H^s} \| \phi^\eps \|_{H^s_{p' (\rho^\eps)}} \,, \\
    O_3 \lesssim\, & \| \na \u \|_{H^s} \| \na \phi^\eps \|_{H^{s-1}} \| \u \|_{H^s} \lesssim \| \na \u \|_{H^s} \| \na \phi^\eps \|_{H^{s-1}_{w(\rho^\eps)}} \| \u \|_{H^s_{\rho^\eps}} \,.
  \end{aligned}
\end{equation}
The term $O_4$ is equal to zero for $m=0$. We estimate the term $O_4$ for the case $1\leq |m| \leq s-1$. The direct calculations imply that
\begin{equation*}
  \begin{aligned}
    O_4 =\, & - \sum_{0 \neq m' \leq m} C_m^{m'} \big\lt \p^{m'} w(\rho^\eps) \na \p^{m-m'} \phi^\eps , \na \p^m \phi^\eps \big\rt \\
    \lesssim \,& \sum_{0 \neq m' \leq m} \| \p^{m'} w(\rho^\eps) \|_{L^4} \| \na \p^{m-m'} \phi^\eps \|_{L^4} \| \na \p^m \phi^\eps \|_{L^2} \\
    \lesssim \,& \sum_{0 \neq m' \leq m} \| \p^{m'} w(\rho^\eps) \|_{H^1} \| \na \p^{m-m'} \phi^\eps \|_{H^1} \| \na \p^m \phi^\eps \|_{L^2} \\
    \lesssim\, & \sum_{1 \leq |m| \leq s} \| \p^m w (\rho^\eps) \|_{L^2} \| \na \phi^\eps \|^2_{H^{s-1}_{w(\rho^\eps)}} \,,
  \end{aligned}
\end{equation*}
where we have made use of the fact that  $\tfrac{1}{w (\rho^\eps)} \lesssim 1$, derived from the bound \eqref{rho-bnd}. From Lemma \ref{lem3}, Lemma \ref{lem1} and the bound \eqref{rho-bnd}, we deduce that for all $1 \leq |m| \leq s$,
\begin{equation*}
  \begin{aligned}
    \| \p^m w (\rho^\eps) \|_{L^2} =\, & \Bigg\| \sum_{i=1}^{|m|} w^{(i)} (\rho^\eps) \sum_{\substack{ m_1 + \cdots + m_i = m \\ |m_1|, \cdots , |m_i| \geq 1 }} \prod_{1 \leq \ell \leq i} \p^{m_\ell} \rho^\eps \Bigg\|_{L^2} \\
    \lesssim\, & \sum_{i=1}^{|m|} \| w^{(i)} (\rho^\eps) \|_{L^\infty} \sum_{1 \leq k \leq s} \| \rho^\eps \|^k_{\dot{H}^s} \\
    \lesssim\, & \| \eps \phi^\eps \|_{\dot{H}^s} \big( 1 + \| \eps \phi^\eps \|^{s-1}_{\dot{H}^s} \big) \lesssim \| \phi^\eps \|_{H^s_{p'(\rho^\eps)}} \big( 1 + \| \phi^\eps \|^{s-1}_{H^s_{p'(\rho^\eps)}} \big) \,.
  \end{aligned}
\end{equation*}
Here $\eps \in (0,1]$ is required. Consequently, we have
\begin{equation}\label{O4}
  \begin{aligned}
    O_4 \lesssim \big( 1 + \| \phi^\eps \|^{s-1}_{H^s_{p'(\rho^\eps)}} \big) \| \phi^\eps \|_{H^s_{p'(\rho^\eps)}} \| \na \phi^\eps \|^2_{H^{s-1}_{w(\rho^\eps)}}
  \end{aligned}
\end{equation}
for all $0 < \eps \leq 1$.

For the terms $O_5$, $O_6$ and $O_7$, it follows from the similar estimates of Section 6 in \cite{JLT} (see (6.11), (6.12) and (6.14), respectively) that for $0 < \eps \leq 1$,
\begin{align}
&\left. 
  \begin{array}{rl}\label{O5}
   O_5 \lesssim\!\!\! & \big( 1 + P_{s-1} ( \| \rho^\eps \|_{\dot{H}^s} ) \big) \| \na \u \|_{H^s} \| \rho^\eps \|_{\dot{H}^s} \\
    \lesssim\!\! \!& \big( 1 + \| \phi^\eps \|^s_{H^s_{p' (\rho^\eps)}} \big) \| \na \phi^\eps \|_{H^{s-1}_{w(\rho^\eps)}} \| \na \u \|_{H^s} \,,
    \end{array}
\right. \\[-0.2ex]
&  \left. 
  \begin{array}{rl}\label{O6}
    O_6 \lesssim \!\!\!& \big( 1 + P_{s-1} ( \| \rho^\eps \|_{\dot{H}^s} ) \big) \| \na \d \|_{H^s} \| \rho^\eps \|_{\dot{H}^s} \| \na \d \|_{\dot{H}^s} \\
    \lesssim\!\!\! & \big( 1 + \| \phi^\eps \|^s_{H^s_{p' (\rho^\eps)}} \big) \| \na \d \|_{H^s} \| \na \d \|_{\dot{H}^s} \| \na \phi^\eps \|_{H^{s-1}_{w (\rho^\eps)}} \,,
    \end{array}
\right.
\end{align}
and
\begin{equation}\label{O7}
  \begin{aligned}
    O_7 \lesssim \,& \big( 1 + P_{s-1} ( \| \rho^\eps \|_{\dot{H}^s} ) \big) \| \rho^\eps \|_{\dot{H}^s} \Big( \| \na \u \|_{H^s} \sum_{j=0}^4 \| \na \d \|_{H^s}^j + \| \cd \|_{H^s} (1 + \| \na \d \|_{H^s}) \Big) \\
    \lesssim\, & \big( 1 + \| \phi^\eps \|^{s+4}_{H^s_{p' (\rho^\eps)}} + \| \na \d \|^{s+4}_{H^s} \big) \| \na \phi^\eps \|_{H^{s-1}_{w (\rho^\eps)}} \big( \| \na \u \|_{H^s} + \| \cd \|_{H^s} \big) \,.
  \end{aligned}
\end{equation}
By plugging the inequalities \eqref{O1O2O3}, \eqref{O4}, \eqref{O5}, \eqref{O6} and \eqref{O7} into \eqref{Dissp-phi-1} and summing up for all $|m| \leq s - 1$, one directly deduces that
\begin{equation}\label{Dissp-phi}
  \begin{aligned}
    & \tfrac{1}{2} \tfrac{\dd}{\dd t} \big( \eps \| \u + \na \phi^\eps \|^2_{H^{s-1}} - \eps \| \u \|^2_{H^{s-1}} - \eps \| \na \phi^\eps \|^2_{H^{s-1}} \big) + \tfrac{1}{2} \| \na \phi^\eps \|^2_{H^{s-1}_{w(\rho^\eps)}}  \\
    & - C \| \na \u \|^2_{H^s} - C \sum_{|m| \leq s} \big\| \p^m \cd + (\p^m \B^\eps) \d + \tfrac{\la_2}{\la_1} ( \p^m \A^\eps ) \d \big\|^2_{L^2} \\
    \lesssim\, & \big( 1 + \| \phi^\eps \|^{s+3}_{H^s_{p'(\rho^\eps)}} + \| \na \d \|^{s+3}_{H^s} \big) \big( \| \u \|_{H^s_{\rho^\eps}} + \| \phi^\eps \|_{H^s_{p' (\rho^\eps)}} + \| \na \d \|_{H^s} \big) \\
    & \times \Big( \| \na \u \|^2_{H^s} + \| \na \d \|^2_{\dot{H}^s_{1/\rho^\eps}} + \| \na \phi^\eps \|^2_{H^{s-1}_{w(\rho^\eps)}} \\
    & \qquad + \sum_{|m| \leq s} \| \p^m \cd + (\p^m \B^\eps) \d + \tfrac{\la_2}{\la_1} (\p^m \A^\eps) \d \|^2_{L^2} \Big) \\
    \lesssim\, & \Big( 1 + \mathcal{E}_s^\frac{s+3}{2} (\phi^\eps, \u, \d) \Big) \mathcal{E}_s^\frac{1}{2} (\phi^\eps, \u, \d) \mathbb{D}_s (\phi^\eps, \u, \d)
  \end{aligned}
\end{equation}
for all $t \in [0,T]$, $0 < \eps \leq 1$ and for some $\eps$-independent constant $C>0$. Here we also have utilized the Young's inequality, the bound \eqref{rho-bnd}, the relation
\begin{equation*}
  \begin{aligned}
    \big\lt \p^m \u , \na \p^m \phi^\eps \big\rt = \tfrac{1}{2} \| \p^m ( \u + \na \phi^\eps ) \|^2_{L^2} - \tfrac{1}{2} \| \p^m \u \|^2_{L^2} - \tfrac{1}{2} \| \na \p^m \phi^\eps \|^2_{L^2} \,,
  \end{aligned}
\end{equation*}
the bounds $\| \dv \u \|^2_{H^{s-1}} \lesssim \| \na \u \|^2_{H^s}$ and
\begin{equation}\label{dot-d-bnd}
  \begin{aligned}
    \| \cd \|^2_{H^s} \lesssim \| \na \u \|^2_{H^s} + \sum_{|m| \leq s} \big\| \p^m \cd + (\p^m \B^\eps) \d + \tfrac{\la_2}{\la_1} ( \p^m \A^\eps ) \d \big\|^2_{L^2} \,.
  \end{aligned}
\end{equation}

\vspace{0.2cm}

 \emph{Step 2. Estimate the dissipation of $\d$.} 
Noticing that the $\d$-equation of \eqref{csys} does not involve any singular term, then we know that the estimates of dissipation of $\d$ are almost the same as the arguments (6.31) in Step 4 of Section 6 in \cite{JLT} (from Page 172 to Page 178). We therefore only sketch the process of proof and omit the details of calculations here. More precisely, for all $1 \leq |m| \leq s$, acting the derivative operator $\p^m$ on the third $\d$-equation of \eqref{csys}, taking the inner product with $\p^m\d$, and making use of integration by parts, we have
\begin{equation}\label{pmd}
  \begin{aligned}
    & \tfrac{1}{2} \tfrac{\dd}{\dd t} \big( \| \p^m \cd + \p^m \d \|^2_{L^2} - \| \p^m \cd \|^2_{L^2} - \| \p^m \d \|^2_{L^2} \big) + \ka \| \na \p^m \d \|^2_{L^2_{1/\rho^\eps}} - \| \p^m \cd \|_{L^2}^2 \\
    =\, & - \lt \p^m ( \u \cdot \na \d ) , \p^m \cd \rt - \lt \p^m ( \u \cdot \na \cd ) , \p^m \d \rt - \ka \big \lt \na \tfrac{1}{\rho^\eps} \na \p^m \d , \p^m \d \big\rt \\
    & + \ka \big\lt \big[ \p^m , \tfrac{1}{\ro} \De \big] \d , \p^m \d \big\rt + \la_1 \big\lt \tfrac{1}{\ro} \p^m \cd , \p^m \d \big\rt
    + \la_1 \big\lt \big[ \p^m , \tfrac{1}{\ro} \big] \cd , \p^m \d \big\rt \\
    & + \big\lt \p^m \big[ \tfrac{1}{\ro} ( \la_1 \B^\e + \la_2 \A^\e ) \d \big] , \p^m \d \big\rt + \big\lt \p^m \big( \tfrac{1}{\ro} \Gamma^\e \d \big) , \p^m \d \big\rt \,.
  \end{aligned}
\end{equation}
Following the same estimates in Step 4 of Section 6 in \cite{JLT} and summing up for all $1 \leq |m| \leq s$, we thereby have
\begin{equation}\label{d-Dissp-bnd-1}
  \begin{aligned}
    & \tfrac{1}{2} \tfrac{\dd}{\dd t} \big( \| \cd + \d \|^2_{\dot{H}^s} - \| \cd \|^2_{\dot{H}^s} - \| \d \|^2_{\dot{H}^s} \big) + \tfrac{3}{4} \ka \| \na \d \|^2_{\dot{H}^s_{1/\rho^\eps}} \\
    & - \big( 1 + 2 |\la_1|^2 \| \tfrac{1}{\rho^\eps} \|_{L^\infty} \big) \| \cd \|^2_{H^s} - 2 C^2 ( |\la_1| + |\la_2| )^2 \| \rho^\eps \|_{L^\infty} \| \na \u \|^2_{H^s} \\
    \lesssim\, & \big( \| \u \|_{H^s} \| \na \d \|_{\dot{H}^s} + \| \na \d \|_{H^s} \| \na \u \|_{H^s} \big) \| \cd \|_{H^s} \\
    & + \| \na \d \|_{H^s} P_s (\| \rho^\eps \|_{\dot{H}^s}) \big( \| \na \d \|_{\dot{H}^s} + \| \cd \|_{H^s} \big) \\
    & + \big( 1 + \| \na \d \|_{H^s} \big) \| \na \d \|_{H^s} \| \na \d \|_{\dot{H}^s} \big( P_s (\| \rho^\eps \|_{\dot{H}^s}) + \| \na \d \|_{\dot{H}^s} \big) \\
    & + \big( 1 + \| \na \d \|_{H^s} \big) \| \na \d \|_{H^s} \| \cd \|^2_{H^s} \\
    & + \| \na \u \|_{H^s} \big( \| \na \d \|_{\dot{H}^s} + P_s (\| \rho^\eps \|_{\dot{H}^s}) \big) \sum_{j=1}^4 \| \na \d \|^j_{H^s} \,,
  \end{aligned}
\end{equation}
where the polynomial $P_s (\cdot)$ is given in Lemma \ref{lem3}. From the bounds \eqref{rho-bnd} and \eqref{dot-d-bnd}, one easily derives that for all $0 < \eps \leq 1$ and for some $\eps$-independent constant $C_0 > 0$,
\begin{equation}\label{d-Dissp-bnd-2}
  \begin{aligned}
   \big( 1 + 2 |\la_1|^2 \| \tfrac{1}{\rho^\eps} \|_{L^\infty} \big)  &\| \cd \|^2_{H^s} + 2 C^2 ( |\la_1| + |\la_2| )^2 \| \rho^\eps \|_{L^\infty} \| \na \u \|^2_{H^s} \\
    \leq\,& C_0 \Big( \| \na \u \|^2_{H^s} + \sum_{ |m| \leq s} \| \p^m \cd + (\p^m \B^\eps) \d + \tfrac{\la_2}{\la_1} (\p^m \A^\eps) \d \|^2_{L^2} \Big) \,, \\
    \| \u \|_{H^s} \lesssim \,&\| \u \|_{H^s_{\rho^\eps}} \,, \quad \ \ \ \| \na \d \|_{\dot{H}^s} \lesssim \| \na \d \|_{\dot{H}^s_{1/\rho^\eps}} \,, \\
    P_s (\| \rho^\eps \|_{\dot{H}^s}) \lesssim\, &\big( 1 + \| \phi^\eps \|^{s-1}_{H^s_{p'(\rho^\eps)}} \big) \| \na \phi^\eps \|_{H^{s-1}_{w(\rho^\eps)}} \,.
  \end{aligned}
\end{equation}
Then, following from the inequalities \eqref{d-Dissp-bnd-1} and \eqref{d-Dissp-bnd-2}, we have
\begin{equation}\label{Dissp-d}
  \begin{aligned}
    & \tfrac{1}{2} \tfrac{\dd}{\dd t} \big( \| \cd + \d \|^2_{\dot{H}^s} - \| \cd \|^2_{\dot{H}^s} - \| \d \|^2_{\dot{H}^s} \big) + \tfrac{3}{4} \ka \| \na \d \|^2_{\dot{H}^s_{1/\rho^\eps}} \\
    & - C_0 \| \na \u \|^2_{H^s} - C_0 \sum_{ |m| \leq s} \| \p^m \cd + (\p^m \B^\eps) \d + \tfrac{\la_2}{\la_1} (\p^m \A^\eps) \d \|^2_{L^2} \\
    \lesssim\, & \big( 1 + \| \na \d \|^{s+2}_{H^s} + \| \phi^\eps \|^{s+2}_{H^s_{p'(\rho^\eps)}} \big) \big( \| \u \|_{H^s_{\rho^\eps}} + \| \na \d \|_{H^s} \big) \\
    & \times \Big( \| \na \u \|^2_{H^s} + \| \na \d \|^2_{\dot{H}^s_{1/\rho^\eps}} + \| \na \phi^\eps \|^2_{H^{s-1}_{w(\rho^\eps)}} \\
    & \qquad + \sum_{|m| \leq s} \| \p^m \cd + (\p^m \B^\eps) \d + \tfrac{\la_2}{\la_1} (\p^m \A^\eps) \d \|^2_{L^2} \Big) \\
    \lesssim\, & \Big( 1 + \mathcal{E}_s^\frac{s+2}{2} (\phi^\eps, \u, \d) \Big) \mathcal{E}_s^\frac{1}{2} (\phi^\eps, \u, \d) \mathbb{D}_s (\phi^\eps, \u, \d)
  \end{aligned}
\end{equation}
for all $t \in [0,T]$ and $0 < \eps \leq 1$.

\vspace*{2mm}

 {\em Step 3. Close the global energy estimates.}
 First, from the inequalities \eqref{d-Dissp-bnd-2}, \eqref{rho-bnd} and \eqref{dot-d-bnd}, the functional $\mathcal{A}_s (\phi^\eps, \u , \d)$ defined in \eqref{As-local} can be bounded by
\begin{equation*}
  \begin{aligned}
    \mathcal{A}_s (\phi^\eps, \u , \d) \lesssim\, & \| \na \phi^\eps \|^2_{H^{s-1}_{w(\rho^\eps)}} + \| \na \d \|^2_{\dot{H}^s_{1/\rho^\eps}} + \| \na \u \|^2_{H^s} \\
    & + \sum_{|m| \leq s} \| \p^m \cd + (\p^m \B^\eps) \d + \tfrac{\la_2}{\la_1} (\p^m \A^\eps) \d \|^2_{L^2} \\
    \lesssim\, & \mathbb{D}_s (\phi^\eps, \u, \d) \,.
  \end{aligned}
\end{equation*}
Furthermore, Proposition \ref{Prop-Local} tells us that $ \mathcal{Q}_c (\u) \lesssim 1 $. Consequently, the inequality \eqref{Local-Apriori-Est} in Lemma \ref{Lmm-Apriori-Est} implies that
\begin{equation}\label{Dissp-local}
  \begin{aligned}
    \tfrac{1}{2} \tfrac{\dd}{\dd t} \mathcal{E}_s (\phi^\eps, \u, \d) + \mathcal{D}_s (\u, \d) \lesssim \big( 1 + \mathcal{E}_s^{s+1} (\phi^\eps, \u , \d) \big) \mathcal{E}_s^\frac{1}{2} (\phi^\eps, \u , \d) \mathbb{D}_s (\phi^\eps, \u, \d) \,.
  \end{aligned}
\end{equation}
By adding the $\eta_0$ times of the inequalities \eqref{Dissp-phi} and \eqref{Dissp-d} into the inequality \eqref{Dissp-local}, recalling the definitions of the instant energy functional $\mathscr{E}_{s, \eta_0} (\phi^\eps, \u, \d )$ and the instant energy dissipative rate functional $\mathscr{D}_{s, \eta_0} (\phi^\eps, \u, \d )$ in \eqref{Es-Inst} and \eqref{Ds-Inst}, respectively, we deduce that
\begin{equation*}
  \begin{aligned}
    & \tfrac{1}{2} \tfrac{\dd}{\dd t} \mathscr{E}_{s, \eta_0} (\phi^\eps, \u, \d) + \mathscr{D}_{s, \eta_0} (\phi^\eps, \u, \d) \\
    \lesssim \,& \big( 1 + \mathcal{E}_s^{s+3} (\phi^\eps, \u , \d) \big) \mathcal{E}_s^\frac{1}{2} (\phi^\eps, \u , \d) \mathbb{D}_s (\phi^\eps, \u, \d) \\
    \lesssim\, & \mathscr{E}_{s, \eta_0}^\frac{1}{2} (\phi^\eps, \u , \d) \mathscr{D}_{s, \eta_0} (\phi^\eps, \u, \d)
  \end{aligned}
\end{equation*}
for all $t \in [0,T]$ and $0 < \eps \leq 1$. Here $\eta_0 > 0$ is given in Lemma \ref{Lmm-Inst-Eneg} and the last inequality is derived from Lemma \ref{Lmm-Inst-Eneg} and Proposition \ref{Prop-Local}, namely $ 1 + \mathcal{E}_s^{s+3} (\phi^\eps, \u , \d) \leq 1 + \delta_0^{s+3} \leq 2 $. Then the proof 
is finished.
\end{proof}

Now, based on Proposition \ref{Prop-Global}, we justify the global well-posedness and the uniform bound \eqref{Unif-Bnd-1} in Theorem \ref{Thm-global}.

\begin{proof}[Proof of Theorem \ref{Thm-global}: global well-posedness and uniform bounds \eqref{Unif-Bnd-1}
and \eqref{Unif-Bnd-rho}]
	From \linebreak Proposition \ref{Prop-Global}, we have that for all $t \in [0,T]$ and $0 < \eps \leq 1$
	\begin{equation}\label{Cont-BInq}
	  \begin{aligned}
	    \tfrac{1}{2} \tfrac{\dd}{\dd t} \mathscr{E}_{s, \eta_0} (\phi^\eps, \u, \d) + \mathscr{D}_{s, \eta_0} (\phi^\eps, \u, \d) \leq C_4 \mathscr{E}_{s, \eta_0}^\frac{1}{2} (\phi^\eps, \u , \d) \mathscr{D}_{s, \eta_0} (\phi^\eps, \u, \d) \,,
	  \end{aligned}
	\end{equation}
	which, combining Lemma \ref{Lmm-Inst-Eneg} and Proposition \ref{Prop-Local}, easily implies that $\mathscr{E}_{s, \eta_0} (\phi^\eps, \u, \d)$ is continuous on $t \in [0,T]$, where $T>0$ is given in Proposition \ref{Prop-Local}. By the initial conditions in Theorem \ref{Thm-global}, we know that $\mathcal{E}_s (\phi^\eps (0), \u (0), \d (0)) \leq \delta$, where $\delta > 0$ is to be determined. We first take $\delta \leq \tfrac{\delta_0}{2}$ such that Proposition \ref{Prop-Local} holds. From Lemma \ref{Lmm-Inst-Eneg}, one deduces that there are constants $C_5, C_6 > 0$, independent of $\eps \in (0,1]$, such that
	\begin{equation}\label{Cont-Eneg-Equiv}
	  \begin{aligned}
	    C_5 \mathcal{E}_s (\phi^\eps, \u, \d) \leq \mathscr{E}_{s, \eta_0} (\phi^\eps, \u, \d) \leq C_6 \mathcal{E}_s (\phi^\eps, \u, \d) \,.
	  \end{aligned}
	\end{equation}
	We now take $\delta$ such that $C_4 \sqrt{C_6 \delta} \leq \tfrac{1}{4}$, hence $0 < \delta \leq \min\big\{ \tfrac{\delta_0}{2} , \tfrac{1}{16 C_4^2 C_6} \big\}$. We therefore have
	\begin{equation}\label{Cont-0}
	  \begin{aligned}
	    C_4 \mathscr{E}_{s, \eta_0}^\frac{1}{2} (\phi^\eps (0), \u (0), \d (0)) \leq C_4 C_6^\frac{1}{2} \mathcal{E}_s^\frac{1}{2} (\phi^\eps (0), \u (0), \d (0)) \leq C_4 \sqrt{C_6 \delta} \leq \tfrac{1}{4} \,.
	  \end{aligned}
	\end{equation}
	We define
	\begin{equation*}
	  \begin{aligned}
	    T^\star := \sup \Big\{ \tau \in (0,T] ; \sup_{t \in [0,\tau]} C_4 \mathscr{E}_{s, \eta_0}^\frac{1}{2} (\phi^\eps, \u, \d) \leq \tfrac{1}{2} \Big\} \in [0,T] \,.
	  \end{aligned}
	\end{equation*}
	Then, the continuity of $\mathscr{E}_{s, \eta_0} (\phi^\eps, \u, \d)$ and the initial bound \eqref{Cont-0} imply that $T^\star > 0$.
	
	We then claim that $T^\star = T$. Indeed, if $0 < T^\star < T$, we derive from \eqref{Cont-BInq} that
	\begin{equation*}
	  \begin{aligned}
	    \tfrac{1}{2} \tfrac{\dd}{\dd t} \mathscr{E}_{s, \eta_0} (\phi^\eps, \u, \d) + \mathscr{D}_{s, \eta_0} (\phi^\eps, \u, \d) \leq \tfrac{1}{2} \mathscr{D}_{s, \eta_0} (\phi^\eps, \u, \d) \,,
	  \end{aligned}
	\end{equation*}
	which means that
	\begin{equation*}
	  \begin{aligned}
	    \tfrac{\dd}{\dd t} \mathscr{E}_{s, \eta_0} (\phi^\eps, \u, \d) + \mathscr{D}_{s, \eta_0} (\phi^\eps, \u, \d) \leq 0
	  \end{aligned}
	\end{equation*}
	for all $t \in [0,T^\star]$ and $0 < \eps \leq 1$. Integrating the previous inequality over $[0,t] \subseteq [0, T^\star]$, we have
	\begin{equation}\label{Cont-Eneg-Bnd}
	  \begin{aligned}
	    \mathscr{E}_{s, \eta_0} (\phi^\eps, \u, \d) (t) + \int_0^t \mathscr{D}_{s,\eta_0} (\phi^\eps, \u, \d) (\tau) \dd \tau \leq  \mathscr{E}_{s, \eta_0} (\phi^\eps (0), \u (0), \d (0)) \leq C_6 \delta
	  \end{aligned}
	\end{equation}
	for all $t \in [0, T^\star]$ and $0 < \eps \leq 1$, which reduces to
	\begin{equation*}
	  \begin{aligned}
	    \sup_{t \in [0,T^\star]} C_4 \mathscr{E}_{s, \eta_0}^\frac{1}{2} (\phi^\eps, \u, \d) \leq \tfrac{1}{4} < \tfrac{1}{2} \,.
	  \end{aligned}
	\end{equation*}
	The continuity of $ \mathscr{E}_{s, \eta_0} (\phi^\eps, \u, \d) $ tells us that there is a small $t^\star > 0$ such that
	$$ \sup_{t \in [0,T^\star + t^\star]} C_4 \mathscr{E}_{s, \eta_0}^\frac{1}{2} (\phi^\eps, \u, \d) \leq \tfrac{1}{2} \,, $$
	which contradicts to the definition of $T^\star$. Thus, we have $T^\star = T$.
	
	We thereby know that the energy bound \eqref{Cont-Eneg-Bnd} holds for all $t \in [0, T]$ and $0 < \eps \leq 1$. We now further take $ \delta $ with $0 < \delta \leq \min\Big\{ \tfrac{\delta_0}{2} , \tfrac{1}{16 C_4^2 C_6} , \tfrac{C_5}{2 C_6} \delta_0, \tfrac{C_5}{16 C_4^2 C_6^2} \Big\}$. Then the inequality \eqref{Cont-Eneg-Equiv} yields
	\begin{equation*}
	  \begin{aligned}
	    \mathcal{E}_s (\phi^\eps, \u, \d) \leq \tfrac{1}{C_5} \mathscr{E}_{s, \eta_0} (\phi^\eps, \u, \d) \leq \tfrac{C_6}{C_5} \delta \leq \min\Big\{ \tfrac{\delta_0}{2}, \tfrac{1}{16 C_4^2 C_6} \Big\}
	  \end{aligned}
	\end{equation*}
	for all $t \in [0,T]$ and $0 < \eps \leq 1$. Thus, the solution constructed in Proposition \ref{Prop-Local} can be globally extended and the bound \eqref{Cont-Eneg-Bnd}, held then in $t \in \R^+$, implies the uniform bound \eqref{Unif-Bnd-1} in Theorem \ref{Thm-global}. Furthermore, the bound \eqref{Unif-Bnd-1} and Lemma \ref{Lm-rho} yield that
\begin{align*}
    & C_{\rho}'(T_1, \delta) \leq \tfrac{1}{2} \exp \Big ( - \widehat{C} T_1^{\frac12} \big ( \int_0^\infty \| \na {\rm u}^\eps \|_{H^s}^2 \dd \tau \big )^{\frac12} \Big ) \leq \tfrac{1}{2} \exp \Big( - \int_0^t \| \dv {\rm u}^\eps \|_{L^\infty} \dd \tau \Big) \\
    & \quad \leq \rho^\eps (t, x) \leq \tfrac{3}{2} \exp \Big( \int_0^t \| \dv {\rm u}^\eps \|_{L^\infty} \dd \tau \Big) \leq \tfrac{3}{2} \exp \Big ( \widehat{C} T_1^{\frac12} \big ( \int_0^\infty \| \na {\rm u}^\eps \|_{H^s}^2 \dd \tau \big )^{\frac12} \Big ) \leq C_{\rho}(T_1, \delta)
\end{align*}
for all $ t \in [0, T_1] $ with any finite $ 0 < T_1 < + \infty $ and all $ x \in \mathbb{R}^n $, where the constants $ C_{\rho}'(T_1, \delta) $ and $ C_{\rho}(T_1, \delta) $ is independent of $\eps$. That is, the bound \eqref{Unif-Bnd-rho} holds.
\end{proof}

\subsection{Uniform bounds \eqref{Unif-Bnd-2}, \eqref{Unif-Bnd-3} and \eqref{Unif-Bnd-4} in Theorem \ref{Thm-global}}

In this subsection, based on the uniform bounds \eqref{Unif-Bnd-1} and \eqref{Unif-Bnd-rho}, we will show the uniform estimate for the derivative of the system \eqref{csys} with respect to time. Own to the third equation of the system $\eqref{csys}$ without singularity, it inspires us to deduce the uniform estimates of $\p_t\cd$ and $\p_t\d$ by the structure of the equation itself. As for the first two equation of the system $\eqref{csys}$ with singularity, we also use the ideas of cancellation the singularity to give the uniform estimates of $\p_t\pe$ and $\p_t\u$. Next, we divide it into two parts to discuss.

\subsubsection{Uniform estimate of $\p_t \cd$ and $\p_t\d$}
By the equation $\eqref{csys}_3$, and noticing that $\ddot{\rm d}^\e=\p_t\cd+\u\ct\na\cd$, we have
\begin{align}\label{tcd}
    \p_t \cd = - \u \ct \na \cd + \ka \tfrac{1}{\ro} \De \d + \tfrac{1}{\ro} \Gamma^\e \d + \la_1 \tfrac{1}{\ro} \cd + \la_1 \tfrac{1}{\ro} \B^\e \d + \la_2 \tfrac{1}{\ro} \A^\e \d \,.
\end{align}
Acting the derivative operator $\pa \, (0\leq |m| \leq s-1)$ on the above equation \eqref{tcd}, we get
\begin{align}\label{mtcd}
    \p_t \p^m \cd = & - \pa ( \u \ct \na \cd ) + \ka \pa \big( \tfrac{1}{\ro} \De \d \big) + \pa \big( \tfrac{1}{\ro} \Gamma^\e \d \big) \nonumber \\
    & + \la_1 \pa \big( \tfrac{1}{\ro} \cd \big) + \la_1 \pa \big( \tfrac{1}{\ro} \B^\e \d \big) + \la_2 \pa \big( \tfrac{1}{\ro} \A^\e \d \big) \,.
\end{align}
For all $|m| \leq s-1$, we deal with the right-hand side of \eqref{mtcd} term by term. By Moser-type calculus inequalities in Lemma \ref{lem1} and the bound \eqref{Unif-Bnd-rho}, one has
\begin{align*}
    \| \pa ( \u \ct \na \cd ) \|_{L^2} \lesssim\,& ( \| \u \|_{L^\infty} \| \pa \na \cd \|_{L^2} + \| \na \cd \|_{L^\infty} \| \pa \u \|_{L^2}) \lesssim \| \u \|_{H^s} \| \cd \|_{H^s} \,, \\
    \big\| \pa \big( \tfrac{1}{\ro} \De \d \big) \big\|_{L^2} \lesssim\, & \big( \big\| \tfrac{1}{\ro} \big\|_{L^\infty} \| \pa \De \d \|_{L^2} + \| \De \d \|_{L^\infty} \big\| \pa \tfrac{1}{\ro} \big\|_{L^2} \big) \lesssim \big( 1 + \| \pe \|^s_{\dot{H}^s} \big) \| \na \d \|_{H^s} \,.
\end{align*}
Recalling the structure of the Lagrangian multiplier $\Gamma^\e$, the third term on right-hand side of the equality \eqref{mtcd} can be divided into three parts as follows:
\begin{align*}
    \pa ( \tfrac{1}{\ro} \Gamma^\e \d ) = - \pa ( | \cd |^2 \d ) + \ka \pa \big( \tfrac{1}{\ro} | \na \d |^2 \d \big) - \la_2 \pa \big( \tfrac{1}{\ro} ( {\d}^{\top} \A^\e \d ) \d \big) \,.
\end{align*}
For the first part, it holds
\begin{align*}
    \| \pa ( | \cd |^2 \d ) \|_{L^2} \leq \,& \| \pa ( | \cd |^2 ) \|_{L^2} + \sum_{\substack{ a+b=m , \\ |b| \geq 1 }} \sum_{a_1 + a_2 = a } \| \p^{a_1} \cd \|_{L^4} \| \p^{a_2} \cd \|_{L^4} \| \p^b \d \|_{L^2} \\
    \lesssim\, & \| \cd \|_{H^s}^2 (1 + \| \na \d \|_{H^s} ) \,.
\end{align*}
For the second part, we have
\begin{align*}
    \big\| \pa \big( \tfrac{1}{\ro} | \na \d |^2 \d \big) \big\|_{L^2} \leq\, & \big\| \tfrac{1}{\ro} \pa ( | \na \d |^2 \d ) \big\|_{L^2} + \sum_{\substack{ a+b+c=m , \\ |a| \geq 1 }} \sum_{b_1+b_2=b} \big\| \p^a \tfrac{1}{\ro} \p^{b_1} \na \d \p^{b_2} \na \d \p^c \d \big\|_{L^2} \\
    \lesssim \,& \big( 1 + \| \pe \|^s_{\dot{H}^s} \big) \| \na \d \|_{H^s}^2 ( 1 + \| \na \d \|_{H^s} ) \,.
\end{align*}
Finally, for the third part, it deduces that
\begin{align*}
    & \big\| \pa \big( \tfrac{1}{\ro} ( {\d}^{\top} \A^\e \d ) \d \big) \big\|_{L^2} \\
    \leq\, & \big\| \tfrac{1}{\ro} \pa \big( ( {\d}^{\top} \A^\e \d ) \d \big) \big\|_{L^2} + \sum_{\substack{ a+b+c=m , \\ |a| \geq 1 }} \sum_{c_1+c_2+c_3=c} \big\| \p^a \tfrac{1}{\ro} \p^b \na \u \p^{c_1} \d \p^{c_2} \d \p^{c_3} \d \big\|_{L^2} \\
    \lesssim\, & \big( 1 + \| \pe \|^s_{\dot{H}^s} \big) \| \u \|_{\dot{H}^s} \sum_{j=0}^3 \| \na \d \|_{H^s}^j \,.
\end{align*}
With the above three estimates at hand, it leads to
\begin{align*}
    \big\| \pa \big( \tfrac{1}{\ro} \Gamma^\e \d \big) \big\|_{L^2} \lesssim\, & \big( 1 + \| \pe \|^s_{\dot{H}^s} \big) ( \| \na \d \|_{H^s} + \| \u \|_{\dot{H}^s} ) \sum_{k=0}^3 \| \na \d \|_{H^s}^k \\
    & + \| \cd \|_{H^s}^2 ( 1 + \| \na \d\|_{H^s} ) \,.
\end{align*}
Similarly, the last three terms of the right-hand side of \eqref{mtcd} can be controlled as
\begin{align*}
    \big\| \pa \big( \tfrac{1}{\ro} \cd \big) \big\|_{L^2} \lesssim\, & \big( \big\| \tfrac{1}{\ro} \big\|_{L^\infty} \| \pa \cd \|_{L^2} + \| \cd \|_{L^\infty} \big\| \pa \tfrac{1}{\ro} \big\|_{L^2} \big) \lesssim \big( 1 + \| \pe \|^s_{H^s} \big) \| \cd \|_{H^s} \,, \\
    \big\| \pa \big( \tfrac{1}{\ro} \B^\e \d \big) \big \|_{L^2} \leq\, & \big\| \tfrac{1}{\ro} \pa ( \B^\e \d ) \big\|_{L^2} + \sum_{\substack{ a+b+c=m , \\ |a| \geq 1 }} \big\| \p^a \tfrac{1}{\ro} \p^b \na \u \p^c \d \big\|_{L^2} \\
    \lesssim\, & \big( 1 + \| \pe \|^s_{H^s} \big) \| \u \|_{\dot{H}^s} ( 1 + \| \na \d \|_{H^s} ) \,, \\
    \big\| \pa \big( \tfrac{1}{\ro} \A^\e \d \big) \big\|_{L^2} \lesssim \, & \big( 1 + \| \pe \|^s_{H^s} \big) \| \u \|_{\dot{H}^s} ( 1 + \| \na \d \|_{H^s} ) \,.
\end{align*}
As a result, we have
\begin{equation}\label{tdd}
  \begin{aligned}
    \| \p_t \cd \|_{H^{s-1}} =\, & \sum_{|m| \leq s -1} \| \pa \p_t \cd \|_{L^2} \lesssim \| \u \|_{H^s} \| \cd \|_{H^s} + C \| \cd \|_{H^s}^2 ( 1 + \| \na \d \|_{H^s} ) \\
    & + \big( 1 + \| \pe \|^{s+2}_{\dot{H}^s} + \| \na \d \|^{s+2}_{H^s} \big) ( \| \na \d \|_{H^s} + \| \u \|_{\dot{H}^s} + \| \cd \|_{H^s} ) \\
    \lesssim \,& \big( 1 + \mathcal{E}_s^\frac{s+2}{2} (\phi^\eps, \u, \d) \big) \mathcal{E}_s^\frac{1}{2} (\phi^\eps, \u, \d) \,.
  \end{aligned}
\end{equation}

We now turn to estimate $\p_t\d$. Notice $ \p_t \d = \cd - \u \ct \na \d $, using Minkowski's inequality and Moser-type calculus inequality, we derive that
\begin{equation}\label{td}
  \begin{aligned}
    \| \p_t \d \|_{H^s} \leq & \| \cd \|_{H^s} + \| ( \u \ct \na ) \d \|_{H^s} \lesssim \| \cd \|_{H^s} + \| \u \|_{H^s} \| \na \d \|_{H^s} \\
    \lesssim & \big( 1 + \mathcal{E}_s^\frac{1}{2} (\phi^\eps, \u, \d) \big) \mathcal{E}_s^\frac{1}{2} (\phi^\eps, \u, \d) \,.
  \end{aligned}
\end{equation}
By \eqref{tdd} and \eqref{td}, and combining the uniform bound \eqref{Unif-Bnd-1}, it holds
\begin{align}\label{tdue}
    \| \p_t \cd \|_{L^\infty ( \R^+; H^{s-1} ) } + \| \p_t \d \|_{L^\infty ( \R^+; H^{s} ) } \leq \widetilde{C}_1 \,,
\end{align}
where the constant $\widetilde{C}_1 > 0$ is independents of $\e$.

\subsubsection{Uniform estimates of $\p_t \pe$ and $\p_t \u$}

By differentiating both the continuity equation $\eqref{csys}_1$ and the momentum equation $\eqref{csys}_2$ with respect to $t$, we get
\begin{align}\label{tsys}
\left\{
 \begin{array}{ll}
    \p_{tt} \pe + \u \ct \na \p_t \pe + \p_t \u \ct \na \pe + \p_t \pe \dv \u + \pe \dv \p_t \u + \frac{1}{\e} \dv \p_t \u = 0 \,, \\
    \p_{tt} \u + \u \ct \na \p_t \u + \p_t \u \ct \na \u + \frac{1}{\e} \p_t ( \frac{p'(\ro)}{\ro} ) \na \pe + \frac{1}{\e} \frac{p'(\ro)}{\ro} \na \p_t \pe \\
    \qquad\quad\qquad = \p_t ( \frac{1}{\ro} ) \dv ( \Sigma_{1}^\e + \Sigma_{2}^\e + \Sigma_{3}^\e ) + \frac{1}{\ro} \dv \p_t ( \Sigma_{1}^\e + \Sigma_{2}^\e + \Sigma_{3}^\e ) \,,
 \end{array}\right.
\end{align}
where $\rho^\eps = 1 + \eps \phi^\eps$.

For all multi-index $m \in \mathbb{N}^n$ with $ |m| \leq s-2 $, applying the derivative operator $\pa$ to the above two equations, and taking the inner product with $p'(\ro)\pa\p_t\pe$ and $\ro\pa\p_t\u$, respectively, gives
\begin{equation}\label{dt-phi-bnd-1}
  \begin{aligned}
    & \tfrac{1}{2} \tfrac{\dd}{\dd t} \| \p^m \p_t \phi^\eps \|^2_{L^2_{p' (\rho^\eps)}} + \underbrace{ \tfrac{1}{\eps} \big\lt \p^m \dv \p_t \u, p' (\rho^\eps) \p^m \p_t \phi^\eps \big\rt }_{\mathscr{S}_{\phi}}\\
    =\, & \underbrace{ \tfrac{1}{2} \big\lt \p_t  p' (\rho^\eps) , |\p^m \p_t \phi^\eps|^2 \big\rt  - \big\lt \u \cdot \na\p^m \p_t \phi^\eps , p'(\rho^\eps) \p^m \p_t \phi^\eps \big\rt }_{\mathscr{X}_1} \\
    &\underbrace{ - \big\lt [\p^m, \u \cdot \na] \p_t \phi^\eps , p'(\rho^\eps) \p^m \p_t \phi^\eps \big\rt }_{\mathscr{X}_2} \ \underbrace{ - \big\lt \p^m ( \p_t \u \cdot \na \phi^\eps ) , p'(\rho^\eps) \p^m \p_t \phi^\eps \big\rt }_{\mathscr{X}_3} \\
    & \underbrace{ - \big\lt \p^m ( \p_t \phi^\eps \dv \u ) , p'(\rho^\eps) \p^m \p_t \phi^\eps \big\rt }_{\mathscr{X}_4} \ \underbrace{ - \big\lt \p^m (\phi^\eps \dv \p_t \u) , p'(\rho^\eps) \p^m \p_t \phi^\eps \big\rt }_{\mathscr{X}_5}
  \end{aligned}
\end{equation}
and
  \begin{align}\label{dt-u-bnd-1}
    \nonumber & \tfrac{1}{2} \tfrac{\dd}{\dd t} \| \p^m \p_t \u \|^2_{L^2_{\rho^\eps}} + \underbrace{ \tfrac{1}{\eps} \big\lt p' (\rho^\eps) \na \p^m \p_t \phi^\eps , \p^m \p_t \u \big\rt }_{\mathscr{S}_{\rm u}} \\
    \nonumber & + \tfrac{1}{2} \mu_4 \| \na \p^m \p_t \u \|^2_{L^2} + (\tfrac{1}{2} \mu_4 + \xi ) \| \dv \p^m \p_t \u \|^2_{L^2} \\
    \nonumber = \,& \underbrace{ -\big\lt [\p^m, \u \cdot \na] \p_t \u , \rho^\eps \p^m \p_t \u \big\rt }_{\mathscr{Y}_1} \ \underbrace{ - \big\lt \p^m (\p_t \u \cdot \na \u) , \rho^\eps \p^m \p_t \u \big\rt }_{\mathscr{Y}_2} \\
    \nonumber & \underbrace{ - \tfrac{1}{\eps} \big\lt [ \p^m, \tfrac{1}{\rho^\eps} p' (\rho^\eps) \na ] \p_t \phi^\eps, \rho^\eps \p^m \p_t \u \big\rt }_{\mathscr{Y}_3} \ \underbrace{ - \tfrac{1}{\eps} \big\lt \p^m ( \p_t(\tfrac{p' (\rho^\eps)}{\rho^\eps}) \na \phi^\eps ), \rho^\eps \p^m \p_t \u \big\rt }_{\mathscr{Y}_4} \\
    \nonumber & + \underbrace{ \big\lt \dv \p^m \p_t \Sigma_2^\eps , \p^m \p_t \u \big\rt + \big\lt \dv \p^m \p_t \Sigma_3^\eps , \p^m \p_t \u \big\rt }_{\mathscr{Y}_5} \\
    \nonumber & + \underbrace{ \big\lt  \p^m \big( \p_t ( \tfrac{1}{\rho^\eps} ) \dv ( \Sigma_1^\eps + \Sigma_2^\eps + \Sigma_3^\eps ) \big) , \rho^\eps \p^m \p_t \u \big\rt }_{\mathscr{Y}_6} \\
    & + \underbrace{ \big\lt  [ \p^m , \tfrac{1}{\rho^\eps} \dv ] \p_t ( \Sigma_1^\eps + \Sigma_2^\eps + \Sigma_3^\eps ) \big) , \rho^\eps \p^m \p_t \u \big\rt }_{\mathscr{Y}_7} \,.
  \end{align}
One notices that the {\em key point} is to deal with the singular term $\mathscr{S}_\phi$ and $\mathscr{S}_{\rm u}$ in \eqref{dt-phi-bnd-1} and \eqref{dt-u-bnd-1}, respectively. The other terms will be controlled by carefully utilizing the Moser-type calculus inequalities in Lemma \ref{lem1}, the embedding inequalities in Lemma \ref{lem2} and integration by parts, similar to the calculations of the a priori estimates in the proof of Lemma \ref{Lmm-Apriori-Est}. We emphasize that the terms $\mathscr{Y}_3$ and $\mathscr{Y}_4$ are not a real singular term, although there is a coefficient $\tfrac{1}{\eps}$ in the front of it. This is because $\rho^\eps = 1 + \eps \phi^\eps$ will separate out  an $\eps$ after operate the derivative, so that the singularity $\tfrac{1}{\eps}$ will be canceled,

Fortunately, for the singular terms $\mathscr{S}_\phi$ and $\mathscr{S}_{\rm u}$, adding the singular term $\mathscr{S}_\phi$ to $\mathscr{S}_{\rm u}$ gives
\begin{equation}\label{dt-Singular-Cancel}
  \begin{aligned}
    \mathscr{S}_\phi + \mathscr{S}_{\rm u} =\, & \tfrac{1}{\eps} \big\lt \p^m \dv \p_t \u, p' (\rho^\eps) \p^m \p_t \phi^\eps \big\rt + \tfrac{1}{\eps} \big\lt p' (\rho^\eps) \na \p^m \p_t \phi^\eps , \p^m \p_t \u \big\rt \\
    =\, & \tfrac{1}{\eps} \big\lt \p^m \dv \p_t \u, p' (\rho^\eps) \p^m \p_t \phi^\eps \big\rt - \tfrac{1}{\eps} \big\lt \na p' (\rho^\eps) \p^m \p_t \phi^\eps , \p^m \p_t \u \big\rt \\
    & - \tfrac{1}{\eps} \big\lt p' (\rho^\eps) \p^m \p_t \phi^\eps , \p^m \dv \p_t \u \big\rt \\
    = \,& - \tfrac{1}{\eps} \big\lt p'' (\rho^\eps) \na \rho^\eps \p^m \p_t \phi^\eps , \p^m \p_t \u \big\rt = - \underbrace{ \big\lt p'' (\rho^\eps) \na \phi^\eps \p^m \p_t \phi^\eps , \p^m \p_t \u \big\rt }_{\mathscr{Z}} \,,
  \end{aligned}
\end{equation}
where we have made use of the integration by parts over $x \in \R^n$ and the relation $\rho^\eps = 1 + \eps \phi^\eps$. Thus, the summation of the singular terms $\mathscr{S}_\phi$ and $\mathscr{S}_{\rm u}$ in \eqref{dt-Singular-Cancel} will successfully cancel the singularity, namely, the quantity  $\mathscr{Z}$ is not of singularity.

Combining with the singularity cancellation \eqref{dt-Singular-Cancel} and summing up for \eqref{dt-phi-bnd-1} and \eqref{dt-u-bnd-1} tell us
\begin{equation}\label{dt-phi-u-bnd-1}
  \begin{aligned}
    & \tfrac{1}{2} \tfrac{\dd}{\dd t} \big( \| \p^m \p_t \phi^\eps \|^2_{L^2_{p' (\rho^\eps)}} + \| \p^m \p_t \u \|^2_{L^2_{\rho^\eps}} \big) + \tfrac{1}{2} \mu_4 \| \na \p^m \p_t \u \|^2_{L^2} \\
    &\qquad \quad + (\tfrac{1}{2} \mu_4 + \xi ) \| \dv \p^m \p_t \u \|^2_{L^2}  = \sum_{i=1}^5 \mathscr{X}_i + \sum_{j=1}^7 \mathscr{Y}_j + \mathscr{Z}
  \end{aligned}
\end{equation}
for all $|m| \leq s - 2$ and $0 < \eps \leq 1$.

Now we estimate the terms $\mathscr{X}_i$ $(1 \leq i \leq 5)$, $\mathscr{Y}_j$ $(1 \leq j \leq 7)$, and $\mathscr{Z}$ in \eqref{dt-phi-u-bnd-1},
respectively. We point out that the terms $\mathscr{X}_2$, $\mathscr{Y}_1$ and $\mathscr{Y}_7$ are equal to zero for $ |m| = 0 $. Since the estimates of these terms are similar to that of the a priori estimates in Lemma \ref{Lmm-Apriori-Est} and the uniform bound \eqref{Unif-Bnd-1} holds, we only list the results and omit the details here for simplicity. More precisely, we have
\begin{equation}\label{X12345}
  \begin{aligned}
    & \mathscr{X}_1 \lesssim \| \u \|_{H^s_{\rho^\eps}} \| \p_t \phi^\eps \|^2_{H^{s-2}_{p' (\rho^\eps)}} \,, \quad
    \mathscr{X}_2 \lesssim \| \u \|_{H^s_{\rho^\eps}} \| \p_t \phi^\eps \|^2_{H^{s-2}_{p' (\rho^\eps)}} \,, \\
    & \mathscr{X}_3 \lesssim \| \phi^\eps \|_{H^s_{p'(\rho^\eps)}} \| \p_t \u \|_{H^{s-2}_{\rho^\eps}} \| \p_t \phi^\eps \|_{H^{s-2}_{p' (\rho^\eps)}} \,, \quad
    \mathscr{X}_4 \lesssim \| \u \|_{H^s_{\rho^\eps}} \| \p_t \phi^\eps \|^2_{H^{s-2}_{p' (\rho^\eps)}} \,, \\
    & \mathscr{X}_5 \lesssim \| \phi^\eps \|_{H^s_{p' (\rho^\eps)}} \| \p_t \phi^\eps \|_{H^{s-2}_{p' (\rho^\eps)}} \| \na \p_t \u \|_{H^{s-2}} \,,
  \end{aligned}
\end{equation}
and
  \begin{equation}\label{Y1234567}
  \begin{aligned}
    \mathscr{Y}_1 & \lesssim \| \u \|_{H^s_{\rho^\eps}} \| \p_t \u \|^2_{H^{s-2}_{\rho^\eps}} \,, \quad  \mathscr{Y}_2 \lesssim \| \u \|_{H^s_{\rho^\eps}} \| \p_t \u \|^2_{H^{s-2}_{\rho^\eps}} \,, \\
      \mathscr{Y}_3 & \lesssim \big( 1 + \|\phi^\eps \|^s_{H^s_{p' (\rho^\eps)}}\big) \| \p_t \phi^\eps \|_{H^{s-2}_{p' (\rho^\eps)}}\| \p_t \u \|_{H^{s-2}_{\rho^\eps}} \,, \\
      \mathscr{Y}_4 & \lesssim \big( 1 + \| \phi^\eps \|^s_{H^s_{p' (\rho^\eps)}}\big) \| \p_t \phi^\eps \|_{H^{s-2}_{p' (\rho^\eps)}}\| \p_t \u \|_{H^{s-2}_{\rho^\eps}} \\
      \mathscr{Y}_5 & + (\mu_5 + \mu_6 + \tfrac{\la_2^2}{\la_1}) \| (\p^m \p_t \A^\eps) \d \|^2_{L^2} + \mu_1 \| {\d}^\top (\p^m \p_t \A^\eps) \d \|^2_{L^2} \\
      & \quad - \la_1 \| (\p^m \p_t \B^\eps) \d + \tfrac{\la_2}{\la_1} (\p^m \p_t \A^\eps) \d \|^2_{L^2} \\
     & \lesssim (1+\|\na \d \|_{H^s}^4)(1+\| \u \|_{H^s_{\rho^\eps}}+\| \cd \|_{H^s_{\rho^\eps}}) \\
      & \quad \times \big( 1+ \| \p_t \cd \|_{H^{s-1}} + \| \p_t \d \|_{H^s} \big)\big(1+ \| \p_t \u \|_{H^{s-2}_{\rho^\eps}}\big) \| \na \p_t \u \|_{H^{s-2}} \,, \\
     \mathscr{Y}_6 & \lesssim \big(1+\|\na \d\|^4_{H^s}\big) \big( \| \u \|_{H^s_{\rho^\eps}} + \| \cd \|_{H^s_{\rho^\eps}} + \| \na \d \|_{H^s} \big)\big( 1+\| \phi^\eps \|_{H^s_{p' (\rho^\eps)}}^s \big) \\
     & \quad \times \| \p_t \phi^\eps \|_{H^{s-2}_{p'(\rho^\eps)}}\big(\| \na \p_t \u \|_{H^{s-2}}+ \| \p_t \u \|_{H^{s-2}_{\rho^\eps}} \big) \,, \\
     \mathscr{Y}_7 & \lesssim \big( 1+\| \phi^\eps \|^s_{H^s_{p' (\rho^\eps)}} \big)\big(1+\|\na \d\|^4_{H^s}\big)\big( 1+ \| \p_t \cd \|_{H^{s-1}} + \| \p_t \d \|_{H^s} \big)\\
    & \quad \times \big( 1+ \| \u \|_{H^s_{\rho^\eps}} + \| \cd \|_{H^s_{\rho^\eps}} \big)\big( \| \na \p_t \u \|_{H^{s-2}} + \| \p_t \u \|_{H^{s-2}_{\rho^\eps}} \big) ( 1 + \| \p_t \u \|_{H^{s-2}_{\rho^\eps}} ) \,,
  \end{aligned}
  \end{equation}
and
\begin{equation}\label{Z-mathscr}
  \begin{aligned}
    \mathscr{Z} \lesssim \| \phi^\eps \|_{H^s_{p'(\rho^\eps)}} \| \p_t \phi^\eps \|_{H^{s-2}_{p'(\rho^\eps)}} \| \p_t \u \|_{H^{s-2}_{\rho^\eps}} \,,
  \end{aligned}
\end{equation}
where the bound \eqref{rho-bnd} is also used frequently.

By plugging the bounds \eqref{X12345}, \eqref{Y1234567} and \eqref{Z-mathscr} into the equality \eqref{dt-phi-u-bnd-1}, summing up for all $|m| \leq s-2$, combining with the coefficients relations \eqref{Coefficients}, the bounds \eqref{tdd}, \eqref{td} and the definition of $\mathcal{E}_s (\phi^\eps, \u, \d)$ in \eqref{Es-local}, we deduce that
\begin{equation*}
  \begin{aligned}
    & \tfrac{1}{2} \tfrac{\dd}{\dd t} E_s (\p_t \phi^\eps , \p_t \u) + D_s (\p_t \u) \lesssim \mathcal{E}_s^\frac{1}{2} (\phi^\eps, \u, \d) E_s (\p_t \phi^\eps, \p_t \u) \\
    &\qquad  + \big( 1 + \mathcal{E}_s^{s+2} (\phi^\eps, \u, \d) \big) \big ( D_s^\frac{1}{2} (\p_t \u) + E_s^\frac{1}{2} (\p_t \phi^\eps, \p_t \u) \big ) \big( 1 + E_s^\frac{1}{2} (\p_t \phi^\eps, \p_t \u) \big)
  \end{aligned}
\end{equation*}
for all $t \in \R^+$ and $0 < \eps \leq 1$, where
\begin{equation}\label{ED-t-derivative}
  \begin{aligned}
    E_s (\p_t \phi^\eps , \p_t \u) &= \| \p_t \phi^\eps \|^2_{H^{s-2}_{p' (\rho^\eps)}} + \| \p_t \u \|^2_{H^{s-2}_{\rho^\eps}} \,, \\
    D_s (\p_t \u) &= \tfrac{1}{2} \mu_4 \| \na \p_t \u \|^2_{H^{s-2}} + (\tfrac{1}{2} \mu_4 + \xi ) \| \dv \p_t \u \|^2_{H^{s-2}} \,.
  \end{aligned}
\end{equation}
Thanks to  the Young's inequality and the uniform bound \eqref{Unif-Bnd-1}, we deduce that there is a constant $C_8 > 0$, independent of $\eps \in (0,1]$, such that
\begin{equation*}
  \begin{aligned}
    \tfrac{\dd}{\dd t} E_s (\p_t \phi^\eps, \p_t \u) + D_s (\p_t \u) \leq C_8 \big( 1 + E_s (\p_t \phi^\eps, \p_t \u) \big)
  \end{aligned}
\end{equation*}
for all $t \in \R^+$ and $0 < \eps \leq 1$. One easily solves the above differential inequality  to obtain that for any fixed $T > 0$,
\begin{equation}\label{dt-Es-1}
  \begin{aligned}
    E_s (\p_t \phi^\eps, \p_t \u) (t) \leq \big(1 + E_s (\p_t \phi^\eps (0) , \p_t \u (0))\big) \exp (C_8 T)
  \end{aligned}
\end{equation}
holds for all $t \in [0,T]$ and $0 < \eps \leq 1$.

Using the first two equations of \eqref{csys}, we deduce that
\begin{equation*}
  \begin{aligned}
    \p_t \phi^\eps (0) & = - \u_0 \cdot \na \phi^\eps_0 - \phi^\eps_0 \dv \u_0 - \tfrac{1}{\eps} \dv \u_0 \,, \\
    \p_t \u (0) & = - \u_0 \cdot \na \u_0 - \tfrac{1}{\eps} \tfrac{1}{\rho^\eps_0} p' (\rho^\eps_0) \na \phi^\eps_0 + \tfrac{1}{\rho^\eps_0} \dv \big( \Sigma_1^\eps (0) + \Sigma_2^\eps (0) + \Sigma_3^\eps (0) \big) \,,
  \end{aligned}
\end{equation*}
where
\begin{equation*}
  \begin{aligned}
    \Sigma_1^\eps (0) =\, & \tfrac{1}{2} \mu_4 ( \na \u_0 + \na \u_0{}^\top ) + \xi \dv \u_0 \I \,, \\
    \Sigma_2^\eps (0) =\, & \tfrac{1}{2} \ka |\na \d_0|^2 \I - \ka \na \d_0 \odot \na \d_0 \,, \\
    \Sigma_3^\eps (0) =\, & \mu_1 ( \d_0{}^\top \A_0^\eps \d_0 ) \d_0 \otimes \d_0 + \mu_2 \d_0 \otimes (\widetilde{\rm d}^\eps_0 + \B_0^\eps \d_0) \\
    & + \mu_3 (\widetilde{\rm d}^\eps_0 + \B_0^\eps \d_0) \otimes \d_0 + \mu_5 \d_0 \otimes (\A_0^\eps \d_0) + \mu_6 (\A_0^\eps \d_0) \otimes \d_0 \,.
  \end{aligned}
\end{equation*}
From the initial data \eqref{IC-1} and \eqref{IC-2}, one easily deduces that
\begin{equation*}
  \begin{aligned}
    \| \p_t \phi^\eps (0) \|_{H^{s-2}} \lesssim\, & \| \u_0 \|_{H^s_{\rho^\eps_0}} \| \phi^\eps_0 \|_{H^s_{p'(\rho^\eps_0)}} + \tfrac{1}{\eps} \|\dv {\rm u}^\eps_0 \|_{H^{s-2}} \lesssim \delta + C_{\rm u} < + \infty \,, \\
    \| \p_t \u (0) \|_{H^{s-2}} \lesssim\, & \tfrac{1}{\eps} \| \na \phi^\eps_0 \|_{H^{s-2}} \big( 1 + \| \phi^\eps_0 \|^s_{H^s_{p'(\rho^\eps_0)}} \big) + \big( 1 + \| \phi^\eps_0 \|^{s+5}_{H^s_{p'(\rho^\eps_0)}} + \| \u_0 \|^{s+5}_{H^s_{\rho^\eps_0}} + \| \na \d_0 \|^{s+5}_{H^s} \big) \\
    & \quad \times \big( \| \u_0 \|_{H^s_{\rho^\eps_0}} + \| \na \d_0 \|_{H^s} + \| \widetilde{\rm d}^\eps_0 \|_{H^s_{\rho^\eps}} \big) \\
    \lesssim\, & C_\phi ( 1 + \delta^\frac{s}{2} ) + ( 1 + \delta^\frac{s+5}{2} ) \delta^\frac{1}{2} < + \infty
  \end{aligned}
\end{equation*}
for all $0 < \eps \leq 1$, which means that $ E_s ( \p_t \phi^\eps (0) , \p_t \u (0) ) \leq C(\delta, C_{\rm u}, C_{\phi}) < + \infty $. Consequently, the inequality \eqref{dt-Es-1} reduces to the bound \eqref{Unif-Bnd-3}, i.e.,
\begin{equation*}
  \begin{aligned}
    \| \p_t \phi^\eps \|^2_{L^\infty(0,T; H^{s-2}_{p' (\rho^\eps)})} + \| \p_t \u \|^2_{L^\infty(0,T; H^{s-2}_{\rho^\eps})} \leq C_{\phi {\rm u}} (T) < + \infty
  \end{aligned}
\end{equation*}
for any fixed $T > 0$ and for all $0 < \eps \leq 1$, where $ C_{\phi {\rm u}} (T) $ is a constant independent of $ \eps $.

Finally, we justify the uniform bound \eqref{Unif-Bnd-4}. From the first two equations of \eqref{csys}, Moser-type calculus inequalities in Lemma \ref{lem1} and the uniform bound \eqref{Unif-Bnd-rho}, we derive that
\begin{equation*}
  \begin{aligned}
    \tfrac{1}{\eps} \| \dv \u \|_{H^{s-2}} \leq & \| \p_t \phi^\eps \|_{H^{s-2}} + \| \u \cdot \na \phi^\eps \|_{H^{s-2}} + \| \phi^\eps \dv \u \|_{H^{s-2}} \\
    \lesssim & \| \p_t \phi^\eps \|_{H^{s-2}_{p'(\rho^\eps)}} + \| \u \|_{H^s_{\rho^\eps}} \| \phi^\eps \|_{H^s_{p'(\rho^\eps)}} \,, \\
    \tfrac{1}{\eps} \| p' (\rho^\eps) \na \phi^\eps \|_{H^{s-2}} \leq & \| \rho^\eps \p_t \u \|_{H^{s-2}} + \| \rho^\eps \u \cdot \na \u \|_{H^{s-2}} + \| \dv (\Sigma_1^\eps + \Sigma_2^\eps + \Sigma_3^\eps) \|_{H^{s-2}} \\
    \lesssim & \big( 1 + \| \phi^\eps \|^{s+4}_{H^s_{p'(\rho^\eps)}} + \| \u \|^{s+4}_{H^s_{\rho^\eps}} + \| \na \d \|^{s+4}_{H^s} \big) \\
    & \times \big( \| \p_t \u \|_{H^{s-2}_{\rho^\eps}} + \| \u \|_{H^s_{\rho^\eps}}  + \| \na \d \|_{H^s} + \| \cd \|_{H^s_{\rho^\eps}} \big) \,,
  \end{aligned}
\end{equation*}
which, combining with the uniform bounds \eqref{Unif-Bnd-1}, \eqref{Unif-Bnd-rho} and \eqref{Unif-Bnd-3}, implies that
\begin{equation*}
  \begin{aligned}
    \tfrac{1}{\eps} \| \dv \u \|_{L^\infty(0;T;H^{s-2})} + \tfrac{1}{\eps} \| p' (\rho^\eps) \na \phi^\eps \|_{L^\infty(0,T; H^{s-2})} \leq C_{\phi {\rm u}}' (T) < + \infty
  \end{aligned}
\end{equation*}
for any fixed $T > 0$ and all $0 < \eps \leq 1$, where $ C_{\phi {\rm u}}' (T) $ is a constant independent of $ \eps $. Then the proof of Theorem \ref{Thm-global} is finished.
\hfill$\square$

\section{Limit to incompressible hyperbolic Ericksen-Leslie's liquid crystal system: Proof of Theorem \ref{Thm-Limit}}\label{Sec:Limit}

In this section, based on the uniform global energy bounds \eqref{Unif-Bnd-1}, \eqref{Unif-Bnd-rho}, \eqref{Unif-Bnd-2}, \eqref{Unif-Bnd-3} and \eqref{Unif-Bnd-4} in Theorem \ref{Thm-global}, we aim at deriving the incompressible hyperbolic Ericksen-Leslie's liquid crystal model \eqref{isys} from the corresponding compressible system \eqref{csys} as $\eps \rightarrow 0$.
We divide it into two parts: limits from the global energy estimates and convergence to the limit equations.

\subsection{Limits from the global energy estimates}

We first introduce the following Aubin-Lions-Simon Theorem, a fundamental result of compactness in the study of nonlinear evolution problems, which can be referred to Theorem II.5.16 of \cite{Boyer-Fabrie-2013-BOOK}, \cite{Simon-1987-AMPA} or \cite{S}, for instance.

\begin{lem}[Aubin-Lions-Simon Theorem]\label{Lmm-ALS-Thm}
	Let $B_0 \subset B_1 \subset B_2$ be three Banach spaces. We assume that the embedding of $B_1$ in $B_2$ is continuous and that the embedding of $B_0$ in $B_1$ is compact. Let 
 $1 \leq p,r \leq + \infty$. For $T > 0$, we define
	\begin{equation*}
	  \begin{aligned}
	    E_{p,r} = \big\{ u \in L^p (0,T;B_0), \p_t u \in L^r (0,T;B_2) \big\} \,.
	  \end{aligned}
	\end{equation*}
	\begin{enumerate}
		\item If $p < + \infty$, the embedding of $E_{p,r}$ in $ L^p (0,T;B_1) $ is compact.
		\item If $p = + \infty$ and $r > 1$, the embedding of $E_{p,r}$ in $ C (0,T;B_1) $ is compact.
	\end{enumerate}
\end{lem}

From Theorem \ref{Thm-global}, we deduce that the Cauchy problem \eqref{csys}-\eqref{civ} admits a global solution $(\phi^\eps, \u, \d) \in \R \times \R^n \times \mathbb{S}^{n-1}$ with
\begin{equation*}
  \begin{aligned}
    \phi^\eps \in L^\infty (\R^+; H^s_{p' (\rho^\eps)}) \,, \ \u, \cd \in L^\infty (\R^+; H^s_{\rho^\eps}) \,, \ \na \u \in L^2 (\R^+; H^s) \,, \ \na \d \in L^\infty (\R^+; H^s) \,,
  \end{aligned}
\end{equation*}
which subjects to the uniform global energy estimates \eqref{Unif-Bnd-1}-\eqref{Unif-Bnd-2}, \eqref{Unif-Bnd-3} and \eqref{Unif-Bnd-4}. Then there is a positive constant $C$, independent of $\eps$, such that
\begin{equation}\label{Ub-1}
  \begin{aligned}
    \| \phi^\eps \|^2_{L^\infty (\R^+; H^s)} & + \| \u \|^2_{L^\infty (\R^+; H^s)} + \| \cd \|^2_{L^\infty (\R^+; H^s)} \\
    & + \| \na \d \|^2_{L^\infty (\R^+; H^s)} + \tfrac{1}{2} \mu_4 \| \na \u \|^2_{L^2(\R^+; H^s)} \leq C \delta \,,
  \end{aligned}
\end{equation}
\begin{equation}\label{Ub-2}
  \begin{aligned}
 \!\!\!\!\!\!\!\!\!\!\!   \| \p_t \cd \|^2_{L^\infty (\R^+; H^{s-1})} + \| \p_t \d \|^2_{L^\infty (\R^+; H^s)} \leq C \,,
  \end{aligned}
\end{equation}
and
\begin{equation}\label{Ub-3}
  \begin{aligned}
    \| \p_t \phi^\eps \|^2_{L^\infty (0,T; H^{s-2})} + \| \p_t \u \|^2_{L^\infty (0,T; H^{s-2})} \leq C(T) \,,
  \end{aligned}
\end{equation}
\begin{equation}\label{Ub-4}
  \begin{aligned}
    \tfrac{1}{\eps} \| \dv \u \|_{L^\infty (0,T; H^{s-2})} + \tfrac{1}{\eps} \| p' (\rho^\eps) \na \phi^\eps \|_{L^\infty (0,T; H^{s-2})} \leq C(T)
  \end{aligned}
\end{equation}
for any fixed $T > 0$ and some $\eps$-independent $C(T) > 0$ and for all $0 < \eps \leq 1$.

From the uniform bounds \eqref{Ub-1}, there exist the functions $ ( \phi, {\rm w}, {\rm d}, {\rm u} ) $ with
\begin{align*}
    & \phi, {\rm w} \in L^\infty (\R^+; H^s) \,, \;\;  {\rm d} \in L^\infty (\R^+; L^\infty) \cap L^\infty (\R^+; \dot{H}^{s+1})\,, \\
    & {\rm u} \in L^\infty (\R^+; H^s) \cap L^2 (\R^+; \dot{H}^{s+1})  \,,
\end{align*}
such that
\begin{equation}\label{Cnvgc-1}
  \begin{aligned}
    & ( \phi^\eps, \u, \cd ) \rightarrow ( \phi, {\rm u}, {\rm w} )\  \ \textrm{ weakly-}\star \textrm{ for } t \geq 0, \textrm{ weakly in } H^s \,, \\
    & \d \rightarrow {\rm d}\  \ \textrm{ weakly-}\star \textrm{ in } L^\infty (\R^+; L^\infty), \textrm{ weakly-} \star \textrm{ for } t \geq 0, \textrm{ weakly in } \dot{H}^{s+1} \,, \\
    & \na \u \rightarrow \na {\rm u} \  \ \textrm{ weakly in } L^2 (\R^+; H^s) \,,
  \end{aligned}
\end{equation}
as $\eps \rightarrow 0$. Using the uniform estimate \eqref{Ub-4}and Poincar\'{e}'s inequality, we obtain
\begin{align*}
    \Big \| \tfrac{1}{\eps^2} p (\rho^\eps) - \tfrac{1}{\eps^2} \tfrac{1}{|\Omega|} \int_{\Omega} p (\rho^\eps) \Big \|_{L^2(\Omega)} \lesssim \| \tfrac{1}{\eps^2} \na p (\rho^\eps) \|_{L^2} \leq C(T)
\end{align*}
for any compact domain $ \Omega\subseteq \R^n $. There exists a function $ \pi_1 \in L^\infty (0,T; L^2_{\rm loc}) $, such that
\begin{align}\label{Cnvgc-p1}
    \tfrac{1}{\eps^2} p (\rho^\eps) - \tfrac{1}{\eps^2} \tfrac{1}{|\Omega|} \int_{\Omega} p (\rho^\eps) \rightarrow \pi_1 \  \ \textrm{ weakly-} \star \textrm{ for } t \in [0, T], \textrm{ weakly in } L^{2}_{\rm loc} \,.
\end{align}
Then, it derives from \eqref{Ub-4} that
\begin{equation}\label{Cnvgc-p2}
\begin{aligned}
    \tfrac{1}{\eps} p'(\rho^\eps) \na \phi^\eps = \tfrac{1}{\eps^2} \na p (\rho^\eps) \rightarrow \na \pi_1 \  \ \textrm{ weakly-} \star \textrm{ for } t \geq 0, \textrm{ weakly in } H^{s-2}_{\rm loc} \,.
\end{aligned}
\end{equation}
The limits \eqref{Cnvgc-1}, \eqref{Cnvgc-p1} and \eqref{Cnvgc-p2} may hold for some subsequences. But, for convenience, we still employ the original notations of the sequences to denote by the subsequences throughout this paper.

One notices that
\begin{equation}\label{Embeddings}
  \begin{aligned}
    H^s \hookrightarrow H^{s-1}_{\loc} \hookrightarrow H^{s-1}_{\loc} ( \textrm{ or } \hookrightarrow H^{s-2}_{\loc} ) \,,
  \end{aligned}
\end{equation}
where the embedding of $H^s$ in $H^{s-1}_{\loc}$ is compact derived from Rellich-Kondrachov Theorem (see \cite{Adams-Fournier-2003-BOOK}, for instance) and the embedding of $H^{s-1}_{\loc}$ in $H^{s-1}_{\loc}$ (or $H^{s-2}_{\loc}$) is naturally continuous. Then, from Aubin-Lions-Simon Theorem in Lemma \ref{Lmm-ALS-Thm}, the bounds \eqref{Ub-1}-\eqref{Ub-3} and the embeddings \eqref{Embeddings}, we deduce that
\begin{equation}\label{Cnvgc-2}
  \begin{aligned}
    (\phi^\eps, \u , \cd, \na \d) \rightarrow (\phi, {\rm u}, {\rm w}, \na {\rm d})
  \end{aligned}
\end{equation}
strongly in $C(\R^+; H^{s-1}_{\loc})$ as $\eps \rightarrow 0$. We immediately know that
\begin{equation*}
  \begin{aligned}
    \u \cdot \na \d \rightarrow {\rm u} \cdot \na {\rm d}
  \end{aligned}
\end{equation*}
strongly in $C(\R^+; H^{s-1}_{\loc})$ as $\eps \rightarrow 0$. Moreover, from the convergences \eqref{Cnvgc-1} and uniform bound \eqref{Ub-2}, one easily deduces that
\begin{equation*}
  \begin{aligned}
    & \p_t \cd  \rightarrow \p_t {\rm w} \  \ \textrm{ weakly-} \star \textrm{ for } t \geq 0, \ \ \textrm{ weakly in } H^{s-1} \,, \\
    & \p_t \d  \rightarrow \p_t {\rm d} \  \ \textrm{ weakly-} \star \textrm{ for } t \geq 0, \ \ \textrm{ weakly in } H^s \,,
  \end{aligned}
\end{equation*}
as $\eps \rightarrow 0$. Combining with $\cd = \p_t \d + \u \cdot \na \d$, we have
\begin{equation}\label{w=dot-d}
  \begin{aligned}
    {\rm w} = \p_t {\rm d} + {\rm u} \cdot \na {\rm d} \,\,( = \dot{\rm d} ) \,.
  \end{aligned}
\end{equation}
Noticing that $\p_t \d = \cd - \u \cdot \na \d$ and $\p_t {\rm d} = \dot{\rm d} - {\rm u} \cdot \na {\rm d}$, we have
\begin{equation*}
  \begin{aligned}
    \tfrac{1}{2} \tfrac{\dd}{\dd t} \| \d - {\rm d} \|^2_{L^2} = \lt \cd - \dot{\rm d} - ( \u \cdot \na \d - {\rm u} \cdot  \na {\rm d} ) , \d - {\rm d} \rt \,,
  \end{aligned}
\end{equation*}
which implies that $ \tfrac{\dd}{\dd t} \| \d - {\rm d} \|_{L^2} \leq \| \cd - \dot{\rm d} \|_{L^2} + \| \u \cdot \na \d - {\rm u} \cdot \na {\rm d} \|_{L^2}$, namely, for any fixed $T > 0$ and for all $t \in [0,T]$,
\begin{equation*}
  \begin{aligned}
    \| \d - {\rm d} \|_{L^2} \leq \| \d_0 - {\rm d}_0 \|_{L^2} + T \| \cd - \dot{\rm d} \|_{L^\infty (0,T;L^2)} + T \| \u \cdot \na \d - {\rm u} \cdot \na {\rm d} \|_{L^\infty (0,T; L^2)} \rightarrow 0
  \end{aligned}
\end{equation*}
as $\eps \rightarrow 0$. Here the convergence \eqref{Cnvgc-2}, \eqref{w=dot-d} and the initial conditions in Theorem \ref{Thm-Limit} are utilized. Consequently, we have
\begin{equation}\label{Cnvgc-3}
  \begin{aligned}
    \lim_{\eps \rightarrow 0} \| \d - {\rm d} \|_{L^\infty(0,T; L^2)} = 0
  \end{aligned}
\end{equation}
for any fixed $T > 0$. This does not mean that $\d \rightarrow {\rm d}$ strongly in $L^2$, because neither $\d$ nor ${\rm d}$ are in $L^2$.

\subsection{Convergence to the limit equations} In this subsection, we will derive the incompressible hyperbolic Ericksen-Leslie's liquid crystal flow \eqref{isys} from the corresponding compressible model \eqref{sys2} and the convergences obtained in the previous subsection.

\subsubsection{Equation of ${\rm u}$}

First, from the uniform bound \eqref{Ub-4}, we know that $ \dv \u \rightarrow 0 $ strongly in $L^\infty (\R^+; H^{s-2})$ as $\eps \rightarrow 0$, which, combining with the convergence \eqref{Cnvgc-1} or \eqref{Cnvgc-2}, implies that
\begin{equation}\label{Incompressibility}
  \begin{aligned}
    \dv {\rm u} = 0
  \end{aligned}
\end{equation}
for a.e. $ (t, x) \in \mathbb{R}^{+} \times \mathbb{R}^n $. Next, we derive the ${\rm u}$-equation of \eqref{isys} from the second equation of \eqref{csys}, namely
\begin{equation*}
  \begin{aligned}
    \p_t \u + \u \cdot \na \u + \eps \phi^\eps \p_t \u + \eps \phi^\eps \u \cdot \na \u + \tfrac{1}{\eps} p' (\rho^\eps) \na \phi^\eps = \dv ( \Sigma_1^\eps + \Sigma_2^\eps + \Sigma_3^\eps ) \,.
  \end{aligned}
\end{equation*}
The convergence \eqref{Cnvgc-2} implies that
\begin{equation}\label{Cnvgc-u-1}
  \begin{aligned}
    \u \cdot \na \u \rightarrow {\rm u} \cdot \na {\rm u}
  \end{aligned}
\end{equation}
strongly in $C(\R^+; H^{s-2}_{\loc})$ as $\eps \rightarrow 0$. From the uniform bounds \eqref{Ub-1}, \eqref{Ub-3} and Moser-type calculus inequalities in Lemma \ref{lem1}, we have
\begin{equation*}
  \begin{aligned}
    & \| \phi^\eps \p_t \u \|_{L^\infty (0,T; H^{s-2})} + \| \phi^\eps \u \cdot \na \u \|_{L^\infty (\R^+; H^{s-1})} \\
    & \qquad \lesssim \| \phi^\eps \|_{L^\infty (\R^+; H^s)} \| \p_t \u \|_{L^\infty (0,T; H^{s-2})} + \| \phi^\eps \|_{L^\infty (\R^+; H^s)} \| \u \|^2_{L^\infty (\R^+; H^s)} \leq C(T)
  \end{aligned}
\end{equation*}
for any fixed $T > 0$, which implies that 
\begin{equation}\label{Cnvgc-u-2}
  \begin{aligned}
    & \eps \phi^\eps \p_t \u \rightarrow 0 \textrm{ strongly in } L^\infty (0,T; H^{s-2}) \textrm{ for any fixed } T > 0 \,, \\
    & \eps \phi^\eps \u \cdot \na \u \rightarrow 0 \textrm{ strongly in } L^\infty (\R^+; H^{s-1}) \,,
  \end{aligned}
\end{equation}
as $\eps \rightarrow 0$. Recalling that $\Sigma_1^\eps = \tfrac{1}{2} \mu_4 ( \na \u + \na \u{}^\top ) + \xi \dv \u \I$, we deduce from the convergence \eqref{Cnvgc-2} that
\begin{equation}\label{Cnvgc-u-4}
  \begin{aligned}
    \dv \Sigma_1^\eps \rightarrow \dv \big( \tfrac{1}{2} \mu_4 ( \na {\rm u} + \na {\rm u}^\top ) + \xi \dv {\rm u} \,\I \big) = \tfrac{1}{2} \mu_4 \Delta {\rm u} + (\tfrac{1}{2} \mu_4 + \xi) \na \dv {\rm u} = \tfrac{1}{2} \mu_4 \Delta {\rm u}
  \end{aligned}
\end{equation}
strongly in $C(\R^+; H^{s-3}_{\loc})$ as $\eps \rightarrow 0$, where we have made use of the incompressibility \eqref{Incompressibility}. Since $\Sigma_2^\eps = \tfrac{1}{2} \ka |\na \d|^2 \I - \ka \na \d \odot \na \d$, the convergence \eqref{Cnvgc-2} yields that
\begin{equation}\label{Cnvgc-u-5}
  \begin{aligned}
    \dv \Sigma_2^\eps \rightarrow \dv \big( \tfrac{1}{2} \ka |\na {\rm d}|^2 \I - \ka \na {\rm d} \odot \na {\rm d} \big) = - \ka \dv ( \na {\rm d} \odot \na {\rm d} ) + \tfrac{1}{2} \ka \na |\na {\rm d}|^2
  \end{aligned}
\end{equation}
strongly in $C(\R^+; H^{s-2}_{\loc})$ as $\eps \rightarrow 0$. Recall that $\Sigma_3^\eps = \tilde{\sigma}_{\bm{{\mu}}} (\u, \d, \cd)$, where
\begin{equation}\label{sigma-tilde}
  \begin{aligned}
    \tilde{\sigma}_{\bm{{\mu}}} (\u, \d, \cd) =\, & \mu_1 (\d{}^\top \A^\eps \d) \d \otimes \d + \mu_2 \d \otimes ( \cd + \B^\eps \d ) \\
    & + \mu_3 ( \cd + \B^\eps \d ) \otimes \d + \mu_5 \d \otimes (\A^\eps \d) + \mu_6 (\A^\eps \d) \otimes \d \,.
  \end{aligned}
\end{equation}
For the first term in \eqref{sigma-tilde}, we have
\begin{equation*}
  \begin{aligned}
    \mu_1 (\d{}^\top \A^\eps \d) \d \otimes \d = \mathscr{V}_\eps + \mu_1 ( {\rm d}^\top \A^\eps {\rm d} ) {\rm d} \otimes {\rm d} \,,
  \end{aligned}
\end{equation*}
where
\begin{equation*}
  \begin{aligned}
    \mathscr{V}_\eps =\, & \mu_1 \big( (\d - {\rm d}){}^\top \A^\eps \d \big) \d \otimes \d + \mu_1 \big( {\rm d}^\top \A^\eps (\d - {\rm d}) \big) \d \otimes \d \\
    & + \mu_1 ( {\rm d}^\top \A^\eps {\rm d} ) (\d - {\rm d}) \otimes \d + \mu_1 ({\rm d}^\top \A^\eps {\rm d}) {\rm d} \otimes (\d - {\rm d}) \,.
  \end{aligned}
\end{equation*}
It follows from the Moser-type calculus inequalities in Lemma \ref{lem1}, the uniform bound \eqref{Ub-1}, the fact $|\d| = 1$, ${\rm d} \in L^\infty (\R^+; L^\infty)$ and the convergences    \eqref{Cnvgc-2} and  \eqref{Cnvgc-3} that
\begin{equation*}
  \begin{aligned}
    \| \mathscr{V}_\eps \|_{L^\infty (\R^+; H^{s-1}_{\loc})} \lesssim\, & \big( \| \d - {\rm d} \|_{L^\infty (\R^+; L^2)} + \| \na \d - \na {\rm d} \|_{L^\infty (\R^+; H^{s-2}_{\loc})} \big) \| \u \|_{L^\infty (\R^+; H^s)} \\
    \lesssim\, & \| \d - {\rm d} \|_{L^\infty (\R^+; L^2)} + \| \na \d - \na {\rm d} \|_{L^\infty (\R^+; H^{s-2}_{\loc})} \rightarrow 0
  \end{aligned}
\end{equation*}
as $\eps \rightarrow 0$, which means that $\mathscr{V}_\eps \rightarrow 0$ strongly in $L^\infty (\R^+; H^{s-1}_{\loc})$ as $\eps \rightarrow 0$. We denote by $\A = \tfrac{1}{2} (\na {\rm u} + \na {\rm u}^\top)$ and $\B = \tfrac{1}{2} (\na {\rm u} - \na {\rm u}^\top)$. Then, from the convergence \eqref{Cnvgc-2}, we derive that
\begin{equation*}
  \begin{aligned}
    \mu_1 ( {\rm d}^\top \A^\eps {\rm d}) {\rm d} \otimes {\rm d} \rightarrow \mu_1 ( {\rm d}^\top \A {\rm d} ) {\rm d} \otimes {\rm d}
  \end{aligned}
\end{equation*}
strongly in $C (\R^+; H^{s-2}_{\loc})$ as $\eps \rightarrow 0$. In summary, we have
\begin{equation}\label{Cnvgc-mu1}
  \begin{aligned}
    \mu_1 (\d{}^\top \A^\eps \d) \d \otimes \d \rightarrow \mu_1 ( {\rm d}^\top \A {\rm d} ) {\rm d} \otimes {\rm d}
  \end{aligned}
\end{equation}
strongly in $C (\R^+; H^{s-2}_{\loc})$ as $\eps \rightarrow 0$. By taking the similar arguments to that in obtaining  \eqref{Cnvgc-mu1}, one can easily derive from the convergences \eqref{Cnvgc-2} and \eqref{Cnvgc-3} that
\begin{equation}\label{Cnvgc-mu2356}
  \begin{aligned}
    & \mu_2 \d \otimes ( \cd + \B^\eps \d ) + \mu_3 ( \cd + \B^\eps \d ) \otimes \d \rightarrow \mu_2 {\rm d} \otimes ( \dot{{\rm d}} + \B {\rm d} ) + \mu_3 ( \dot{{\rm d}} + \B {\rm d} ) \otimes {\rm d} \,, \\
    & \mu_5 \d \otimes (\A^\eps \d) + \mu_6 (\A^\eps \d) \otimes \d \rightarrow \mu_5 {\rm d} \otimes (\A {\rm d}) + \mu_6 (\A {\rm d}) \otimes {\rm d}  \,,
  \end{aligned}
\end{equation}
strongly in $C (\R^+; H^{s-2}_{\loc})$ as $\eps \rightarrow 0$. Consequently, the limits \eqref{Cnvgc-mu1} and \eqref{Cnvgc-mu2356} give
\begin{equation*}
  \begin{aligned}
    \tilde{\sigma}_{\bm{\mu}} (\u, \d, \cd) \rightarrow\, & \mu_1 ( {\rm d}^\top \A {\rm d} ) {\rm d} \otimes {\rm d} + \mu_2 {\rm d} \otimes ( \dot{{\rm d}} + \B {\rm d} ) \\
    & + \mu_3 ( \dot{{\rm d}} + \B {\rm d} ) \otimes {\rm d} + \mu_5 {\rm d} \otimes (\A {\rm d}) + \mu_6 (\A {\rm d}) \otimes {\rm d}   = \tilde{\sigma}_{\bm{\mu}} ( {\rm u}, {\rm d}, \dot{\rm d})
  \end{aligned}
\end{equation*}
strongly in $C (\R^+; H^{s-2}_{\loc})$ as $\eps \rightarrow 0$, which means that
\begin{equation}\label{Cnvgc-u-6}
  \begin{aligned}
    \dv \Sigma_3^\eps \rightarrow \dv \tilde{\sigma}_{\bm{\mu}} ( {\rm u}, {\rm d}, \dot{\rm d} )
  \end{aligned}
\end{equation}
strongly in $C (\R^+; H^{s-3}_{\loc})$ as $\eps \rightarrow 0$. For any $T > 0$, let a vector-valued test function $\psi (t,x) \in C^1 (0,T; C_c^\infty (\R^n))$ with $\psi (0,x) = \psi_0 (x) \in C_c^\infty (\R^n)$ and $\psi (t,x) = 0$ for $t \geq T'$, where $T' < T$. Then we deduce from the initial conditions in Theorem \ref{Thm-Limit} and the convergence \eqref{Cnvgc-1} that
\begin{equation}\label{Cnvgc-u-7}
  \begin{aligned}
    \int_0^T \int_{\R^n} \p_t \u \cdot \psi (t,x) \dd x \dd t = & - \int_{\R^n} \u_0 (x) \cdot \psi_0 (x) \dd x - \int_0^T \int_{\R^n} \u \cdot \p_t \psi (t,x) \dd x \dd t \\
    \rightarrow & - \int_{\R^n} {\rm u}_0 (x) \cdot \psi_0 (x) \dd x - \int_0^T \int_{\R^n} {\rm u} \cdot \p_t \psi (t,x) \dd x \dd t
  \end{aligned}
\end{equation}
as $\eps \rightarrow 0$. As a consequence, the limits \eqref{Cnvgc-p2}, \eqref{Cnvgc-u-1}-\eqref{Cnvgc-u-5}, \eqref{Cnvgc-u-6} and \eqref{Cnvgc-u-7} imply that ${\rm u} (t,x) \in L^\infty (\R^+; H^s) \cap L^2 (\R^+; \dot{H}^{s+1})$ subjects to the evolution
\begin{equation*}
  \begin{aligned}
    \p_t {\rm u} + {\rm u} \cdot \na {\rm u} + \na \pi = \,& \tfrac{1}{2} \mu_4 \Delta {\rm u} - \ka \dv ( \na {\rm d} \odot \na {\rm d} ) + \dv \tilde{\sigma}_{\bm{\mu}} ({\rm u}, {\rm d}, \dot{\rm d}) \,, \\
    \dv {\rm u} =\, & 0 \,,
  \end{aligned}
\end{equation*}
with the initial data
\begin{equation*}
  \begin{aligned}
    {\rm u} |_{t=0} = {\rm u}_0 (x) \,,
  \end{aligned}
\end{equation*}
where $\pi = \pi_1 - \tfrac{1}{2} \ka |\na {\rm d}|^2 \in L^\infty (\R^+; H^{s-1}_{\rm loc})$.

\subsubsection{Equation of \,${\rm d}$} Based on the convergences obtained in the previous subsection, we now derive the ${\rm d}$-equation in \eqref{isys} from the last equation of \eqref{csys}, i.e.,
\begin{equation}\label{Cnvgc-d-1}
  \begin{aligned}
    \p_t \cd + \u \cdot \na \cd + \eps \phi^\eps \p_t \cd + \eps \phi^\eps \u \cdot \na \cd = \ka \Delta \d + \Gamma^\eps \d + \la_1 ( \cd + \B^\eps \d ) + \la_2 \A^\eps \d \,,
  \end{aligned}
\end{equation}
where $\Gamma^\eps = - |\cd|^2 - \eps \phi^\eps |\cd|^2 + \ka |\na \d|^2 - \la_2 \d{}^\top \A^\eps \d$.

From the convergence \eqref{Cnvgc-2} and the relation \eqref{w=dot-d}, we have
\begin{equation*}
  \begin{aligned}
    \u \cdot \na \cd \rightarrow {\rm u} \cdot \na \dot{{\rm d}}
  \end{aligned}
\end{equation*}
strongly in $C(\R^+; H^{s-2}_{\loc})$ as $\eps \rightarrow 0$. Furthermore, from the Moser-type calculus inequalities in Lemma \ref{lem1}, and the uniform bounds \eqref{Ub-1} and \eqref{Ub-2}, we deduce that
\begin{equation*}
  \begin{aligned}
    & \| \phi^\eps \p_t \cd \|_{L^\infty (\R^+; H^{s-1})} + \| \phi^\eps \u \cdot \na \u \|_{L^\infty (\R^+; H^{s-1})} \\
   &\qquad \lesssim \| \phi^\eps \|_{L^\infty (\R^+; H^s)} \big( \| \p_t \cd \|_{L^\infty (\R^+; H^{s-1})} + \| \u \|^2_{L^\infty (\R^+; H^s)} \big) \lesssim 1 \,,
  \end{aligned}
\end{equation*}
which means that
\begin{equation}\label{Cnvgc-d-2}
  \begin{aligned}
    \eps \phi^\eps \p_t \cd \,, \ \eps \phi^\eps \u \cdot \na \u \rightarrow 0
  \end{aligned}
\end{equation}
strongly in $L^\infty (\R^+; H^{s-1})$ as $\eps \rightarrow 0$. The convergence \eqref{Cnvgc-2} also tells us
\begin{equation}\label{Cnvgc-d-3}
  \begin{aligned}
    \ka \Delta \d \rightarrow \ka \Delta {\rm d}
  \end{aligned}
\end{equation}
strongly in $C(\R^+; H^{s-2}_{\loc})$ as $\eps \rightarrow 0$. It follows from the convergence \eqref{Cnvgc-2}, the similar arguments in the limit \eqref{Cnvgc-mu1} and the analogous derivations of \eqref{Cnvgc-d-2} that
\begin{equation*}
  \begin{aligned}
    - |\cd|^2 + \ka |\na \d|^2 \rightarrow \,&- |\dot{\rm d}|^2 + \ka |\na {\rm d}|^2 \ \textrm{  strongly in } C(\R^+; H^{s-1}_{\loc}) \,, \\
    - \la_2 \d{}^\top \A^\eps \d \rightarrow \,& - \la_2 {\rm d}^\top \A {\rm d} \ \textrm{   strongly in } C(\R^+; H^{s-2}_{\loc}) \,, \\
    - \eps \phi^\eps |\cd|^2 \rightarrow \,& \,0 \ \textrm{  strongly in } L^\infty (\R^+; H^s) \,
  \end{aligned}
\end{equation*}
as $\eps \rightarrow 0$, which immediately implies
\begin{equation}\label{Cnvgc-Gamma}
  \begin{aligned}
    \Gamma^\eps \rightarrow - |\dot{\rm d}|^2 + \ka |\na {\rm d}|^2 - \la_2 {\rm d}^\top \A {\rm d} = \Gamma
  \end{aligned}
\end{equation}
strongly in $C(\R^+; H^{s-2}_{\loc})$ as $\eps \rightarrow 0$. Notice that $\Gamma^\eps \d = \Gamma^\eps (\d - {\rm d}) + \Gamma^\eps {\rm d}$. Then the limit \eqref{Cnvgc-Gamma} reduces to $ \Gamma^\eps {\rm d} \rightarrow \Gamma {\rm d} $ strongly in $C(\R^+; H^{s-2}_{\loc})$ as $\eps \rightarrow 0$. Moreover, for the term $ \Gamma^\eps (\d - {\rm d}) $, we derive from the Moser-type calculus inequalities in Lemma \ref{lem1}, the fact $|\d| = 1$, the uniform bound \eqref{Ub-1}, and the convergences \eqref{Cnvgc-2} and \eqref{Cnvgc-3} that
\begin{equation*}
  \begin{aligned}
    \| \Gamma^\eps (\d - {\rm d}) \|_{L^\infty (\R^+; H^{s-1}_{\loc})} \lesssim & \big( \| \cd \|^2_{L^\infty (\R^+; H^s)} + \| \na \d \|^2_{L^\infty (\R^+; H^s)} + \| \u \|^2_{L^\infty (\R^+; H^s)} \big) \\
    & \times \big( \| \d - {\rm d} \|_{L^\infty (\R^+; L^2)} + \| \na \d - \na {\rm d} \|_{L^\infty (\R^+; H^{s-1}_{\loc})} \big) \\
    \lesssim & \| \d - {\rm d} \|_{L^\infty (\R^+; L^2)} + \| \na \d - \na {\rm d} \|_{L^\infty (\R^+; H^{s-1}_{\loc})} \rightarrow 0
  \end{aligned}
\end{equation*}
as $\eps \rightarrow 0$, which means that $ \Gamma^\eps (\d - {\rm d}) \rightarrow 0 $ strongly in $ L^\infty (\R^+; H^{s-1}_{\loc}) $ as $\eps \rightarrow 0$. Consequently, we have
\begin{equation}\label{Cnvgc-d-4}
  \begin{aligned}
    \Gamma^\eps \d \rightarrow \Gamma {\rm d}
  \end{aligned}
\end{equation}
strongly in $C(\R^+; H^{s-2}_{\loc})$ as $\eps \rightarrow 0$. Finally, from the similar arguments in \eqref{Cnvgc-mu1} or \eqref{Cnvgc-mu2356},
it holds
\begin{equation}\label{Cnvgc-d-5}
  \begin{aligned}
    \la_1 ( \cd + \B^\eps \d ) + \la_2 \A^\eps \d \rightarrow \la_1 ( \dot{{\rm d}} + \B {\rm d} ) + \la_2 \A {\rm d}
  \end{aligned}
\end{equation}
strongly in $ C(\R^+; H^{s-2}_{\loc}) $ as $\eps \rightarrow 0$. For any $T > 0$, let a vector-valued test function $\zeta (t,x) \in C^1 (0,T; C_c^\infty (\R^n))$ with $\zeta (0,x) = \zeta_0 (x) \in C_c^\infty (\R^n)$ and $\zeta (t,x) = 0$ for $t \geq T'$, where $T' < T$. Then we deduce from the initial conditions in Theorem \ref{Thm-Limit} and the convergence \eqref{Cnvgc-1} that
\begin{equation}\label{Cnvgc-d-6}
  \begin{aligned}
    \int_0^T \int_{\R^n} \p_t \cd \cdot \zeta (t,x) \dd x \dd t = & - \int_{\R^n} \widetilde{\rm d}^\eps_0 (x) \cdot \zeta_0 (x) \dd x - \int_0^T \int_{\R^n} \cd \cdot \p_t \zeta (t,x) \dd x \dd t \\
    \rightarrow & - \int_{\R^n} \widetilde{\rm d}_0 (x) \cdot \zeta_0 (x) \dd x - \int_0^T \int_{\R^n} \dot{\rm d} \cdot \p_t \zeta (t,x) \dd x \dd t
  \end{aligned}
\end{equation}
as $\eps \rightarrow 0$. We summarize the limits \eqref{Cnvgc-d-1}-\eqref{Cnvgc-d-3}, and \eqref{Cnvgc-d-4}-\eqref{Cnvgc-d-6} to obtain that $\dot{{\rm d}}$ obey the evolution
\begin{equation*}
  \begin{aligned}
    \p_t \dot{{\rm d}} + {\rm u} \cdot \nabla \dot{{\rm d}} = \ka \Delta {\rm d} + \Gamma {\rm d} + \la_1 ( \dot{{\rm d}} + \B {\rm d} ) + \la_2 \A {\rm d}
  \end{aligned}
\end{equation*}
with the initial data
\begin{equation*}
  \begin{aligned}
    \dot{{\rm d}} |_{t=0} = \widetilde{\rm d}_0 (x) \,.
  \end{aligned}
\end{equation*}
Furthermore, we have proved that $\d \rightarrow {\rm d}$ strongly in $C(\R^+ \dot{H}^s_{\loc})$ as $\eps \rightarrow 0$ in \eqref{Cnvgc-2}. Using the relations $\p_t \d = \cd - \u \cdot \nabla \d$ and $\p_t {\rm d} = \dot{\rm d} - {\rm u} \cdot \nabla {\rm d}$, combining with the initial data conditions given in Theorem \ref{Thm-Limit}, it follows from the similar process as \eqref{Cnvgc-d-6} that
\begin{equation*}
  \begin{aligned}
    {\rm d} |_{t=0} = {\rm d}_0 (x) \,.
  \end{aligned}
\end{equation*}
Since $\widetilde{\rm d}_0 (x) \cdot {\rm d}_0 (x) = 0$, $|{\rm d}_0 (x)| = 1$, ${\rm d} (t, x) \in L^\infty (\R^+; L^\infty)$ and ${\rm u} (t,x), \na {\rm d} (t,x), \dot{\rm d} (t,x) \in L^\infty (\R^+; H^s)$, it follows from Lemma 4.1 in \cite{JL} that
\begin{align*}
    |{\rm d}|=1\quad \hbox{for all}\;\; (t,x) \in \R^+ \times \mathbb{R}^n \,.
\end{align*}
Consequently, the limit functions $({\rm u}, {\rm d}) (t,x)$ satisfy the system
\begin{equation*}
  \left\{
    \begin{array}{l}
       \p_t {\rm u} + {\rm u} \cdot \na {\rm u} + \na \pi = \tfrac{1}{2} \mu_4 \Delta {\rm u} - \ka \dv ( \na {\rm d} \odot \na {\rm d} ) + \dv \tilde{\sigma}_{\bm{\mu}} ({\rm u}, {\rm d}, \dot{\rm d}) \,, \\
       \dv {\rm u} = 0 \,, \\
       \ddot{\rm d} = \ka \Delta {\rm d}  + \Gamma {\rm d} + \la_1 ( \dot{{\rm d}} + \B {\rm d} ) + \la_2 \A {\rm d} \,, \quad
       |{\rm d}| = 1
    \end{array}
  \right.
\end{equation*}
with the initial data $ ({\rm u}, {\rm d}, \dot{{\rm d}} ) |_{t=0} = ( {\rm u}_0 (x), {\rm d}_0 (x) , \widetilde{\rm d}_0 (x) ) $. Moreover, the uniform bound \eqref{Ub-1} gives
\begin{equation*}
  \begin{aligned}
    \| {\rm u} \|^2_{L^\infty (\R^+; H^s)} + \| \dot{\rm d} \|^2_{L^\infty (\R^+; H^s)} + \| \na {\rm d} \|^2_{L^\infty (\R^+; H^s)} + \tfrac{1}{2} \mu_4 \| \na {\rm u} \|^2_{L^2 (\R^+; H^s)} \lesssim \delta \,.
  \end{aligned}
\end{equation*}
Then the proof of Theorem \ref{Thm-Limit} is completed. \hfill$\square$

\section{Convergence rate: Proof of Theorem \ref{Thm-Convergence-Rate}}\label{Sec:Convergence}
 
In this section, we aim at proving the convergence rate (in $L^2$-norms) of the limit process in Theorem \ref{Thm-Limit} by employing the modulated energy method, see for example \cite{JJL}.

\begin{proof}[Proof of Theorem \ref{Thm-Convergence-Rate}]
Multiplying the two equations $\eqref{csys}_2$ and $\eqref{csys}_3$ by $ ( \ro \u, \ro \cd ) $, respectively, integrating
the results over $\mathbb{R}^n$ with respect to $x$, using the continuity equation $\eqref{csys}_1$ and  integration by parts, and noticing that $|\d|=1$, we infer the basic energy (see also Section 2 of \cite{JLT})
\begin{equation}\label{Rate-1}
  \begin{aligned}
    & \tfrac{1}{2} \tfrac{\dd}{\dd t} \big( \| \sqrt{\ro} \u \|^2_{L^2} + \| \sqrt{\ro} \cd \|^2_{L^2} + \ka \| \na \d \|^2_{L^2} + 2 \lt \Pi^\e , 1 \rt \big) \\
    & + \tfrac{1}{2} \mu_4 \| \na \u \|_{L^2}^2 + \big( \tfrac{1}{2} \mu_4 + \xi \big) \| \dv \u \|_{L^2}^2  \\
    = \,& \lt \dv \Sigma_3^\eps , \u \rt + \la_1 \lt \cd + \B^\eps \d , \cd \rt + \la_2 \lt \A^\eps \d , \cd \rt \,,
  \end{aligned}
\end{equation}
where
\begin{align*}
    \Pi^\e = \tfrac{1}{\e^2} \tfrac{ \tilde{a} }{\gamma-1} [ (\ro)^\gamma - \gamma ( \ro - 1 ) -1 ]
\end{align*}
is nonnegative. For any fixed $T>0$, integrating \eqref{Rate-1} over $(0,t) \subseteq (0,T]$ with respect to $t$, we get
\begin{equation}\label{Rate-2}
  \begin{aligned}
    & \tfrac{1}{2} \| \sqrt{\ro} \u \|^2_{L^2} + \tfrac{1}{2} \| \sqrt{\ro} \cd \|^2_{L^2} + \tfrac{1}{2} \ka \| \na \d \|^2_{L^2} + \lt \Pi^\e , 1 \rt \\
    & + \tfrac{1}{2} \mu_4 \int_{0}^{t} \| \na \u \|_{L^2}^2 \dd \tau + \big( \tfrac{1}{2} \mu_4 + \xi \big) \int_{0}^{t} \| \dv \u \|_{L^2}^2 \dd \tau \\
    =\, & \int_0^t \lt \dv \Sigma_3^\eps , \u \rt + \la_1 \lt \cd + \B^\eps \d , \cd \rt + \la_2 \lt \A^\eps \d , \cd \rt \dd \tau  \\
    & + \tfrac{1}{2} \| \sqrt{\ro_0} \u_0 \|^2_{L^2} + \tfrac{1}{2} \| \sqrt{\ro_0} \widetilde{\rm d}^\eps_0 \|^2_{L^2} + \tfrac{1}{2} \ka \| \na \d_0 \|^2_{L^2} + \lt \Pi^\e_0 , 1 \rt \,,
  \end{aligned}
\end{equation}
where $\Pi^\eps_0$ is obtained by replacing the $\rho^\eps$ with $\rho^\eps_0$ in the quantity $\Pi^\eps$.

Applying the calculations in Section 2 of \cite{JL}, we can derive the following basic energy law of the incompressible system \eqref{isys}:
\begin{equation}\label{Rate-3}
  \begin{aligned}
    & \tfrac{1}{2} \| {\rm u} \|^2_{L^2} + \tfrac{1}{2} \| \dot{\rm d} \|^2_{L^2} + \tfrac{1}{2} \ka \| \na {\rm d} \|^2_{L^2} + \tfrac{1}{2} \mu_4 \int_0^t \| \na {\rm u} \|^2_{L^2} \dd \tau \\
    =\, & \int_0^t \lt {\rm u} , \dv \tilde{\sigma}_{\bm{\mu}} ({\rm u}, {\rm d}, \dot{\rm d}) \rt + \la_1 \lt \dot{\rm d} , \dot{\rm d} + \B {\rm d} \rt + \la_2 \lt \dot{\rm d} , \A {\rm d} \rt \dd \tau \\
    & + \tfrac{1}{2} \| {\rm u}_0 \|^2_{L^2} + \tfrac{1}{2} \| \widetilde{\rm d}_0 \|^2_{L^2} + \tfrac{1}{2} \ka \| \na {\rm d}_0 \|^2_{L^2}
  \end{aligned}
\end{equation}
holds for all $t \in [0,T]$, where $T > 0$ is an arbitrary fixed number.

Taking the inner product of $\eqref{csys}_2$ with ${\rm u}$, with the help of integration by parts, and combining the equation $\eqref{isys}_2$, it can be inferred that
  \begin{align}\label{Rate-4}
    \nonumber & \lt \rho^\eps \u , {\rm u} \rt - \lt \rho^\eps_0 \u_0 , {\rm u}_0 \rt \\ \nonumber
=\,& \lt \p_t (\rho^\eps \u) , {\rm u} \rt + \lt \rho^\eps \u, \p_t {\rm u} \rt \\
    \nonumber = \,& - \int_0^t \lt \rho^\eps \u , {\rm u} \cdot \na {\rm u} + \na \pi - \tfrac{1}{2} \mu_4 \Delta {\rm u} + \ka \dv (\na {\rm d} \odot \na {\rm d}) - \dv \tilde{\sigma}_{\bm{\mu}} ({\rm u}, {\rm d}, \dot{\rm d}) \rt \dd \tau \\
    & + \int_0^t \lt \rho^\eps \u \otimes \u : \na {\rm u} \rt \dd \tau + \int_0^t \lt \dv (\Sigma_1^\eps + \Sigma_2^\eps + \Sigma_3^\eps) , {\rm u} \rt \dd \tau \,,
  \end{align}
where we have also made use of the fact that  $\dv {\rm u} = 0$. Noticing that 
$\Sigma_2^\eps = \tfrac{1}{2} \ka |\na \d|^2 \I - \ka \na \d \odot \na \d$, and 
the divergence free property of ${\rm u}$ implies that
\begin{equation*}
  \begin{aligned}
    \lt \dv \Sigma_1^\eps , {\rm u} \rt = - \tfrac{1}{2} \mu_4 \lt \na \u, \na {\rm u} \rt \,,
  \end{aligned}
\end{equation*}
  thus the equality \eqref{Rate-4} can be further rewritten as
\begin{equation}\label{Rate-5}
  \begin{aligned}
    & - \lt \sqrt{\rho^\eps} \u , {\rm u} \rt - \mu_4 \int_0^t \lt \na \u, \na {\rm u} \rt \dd \tau \\
    = \,& \mathscr{R}_{\rm u}^\eps - \lt \rho^\eps_0 \u_0 , {\rm u}_0 \rt - \int_0^t \lt {\rm u} , \dv \Sigma_3^\eps \rt + \lt \u, \dv \tilde{\sigma}_{\bm{\mu}} ({\rm u}, {\rm d}, \dot{\rm d}) \rt \dd \tau \\
    & + \ka \int_0^t \lt \dv \big( \na \d \odot \na \d \big) , {\rm u} \rt \dd \tau + \ka \int_0^t \lt \dv (\na {\rm d} \odot \na {\rm d}) , \u \rt \dd \tau\,,
  \end{aligned}
\end{equation}
where
\begin{equation}\label{Ru-eps}
  \begin{aligned}
    \mathscr{R}_{\rm u}^\eps =\, & \lt (\rho^\eps - \sqrt{\rho^\eps}) \u , {\rm u} \rt + \int_0^t \lt \rho^\eps \u \otimes ( \u - {\rm u} ) : \na {\rm u} \rt \dd \tau - \int_0^t \lt \rho^\eps \u - \u , \tfrac{1}{2} \mu_4 \Delta {\rm u} \rt \dd \tau \\
    & + \int_0^t \lt \rho^\eps \u , \na \pi \rt \dd \tau + \int_0^t \lt (\rho^\eps - 1 ) \u , \ka \dv (\na {\rm d} \odot \na {\rm d}) - \dv \tilde{\sigma}_{\bm{\mu}} ({\rm u}, {\rm d}, \dot{\rm d}) \rt \dd \tau \,.
  \end{aligned}
\end{equation}
The term $\mathscr{R}_{\rm u}^\eps$ contains various difference forms, which will be easily estimated later.

By taking the inner product of $\eqref{csys}_3$ with $\dot{\rm d}$, by means of integration by parts, and along with the equation $\eqref{isys}_3$, it can be deduced that
\begin{equation*}
  \begin{aligned}
    &- \lt \sqrt{\rho^\eps} \cd , \dot{\rm d} \rt  - \ka \lt \na \d , \na {\rm d} \rt\\
 = \,& \lt (\rho^\eps - \sqrt{\rho^\eps}) \cd , \dot{\rm d} \rt - \lt \rho^\eps_0 \widetilde{\rm d}^\eps_0 , \widetilde{\rm d}_0 \rt - \ka \lt \na \d_0 , \na {\rm d}_0 \rt \\
    & + \int_0^t \lt \rho^\eps \cd , {\rm u} \cdot \na \dot{\rm d} - \ka \Delta {\rm d} - \Gamma {\rm d} - \la_1 ( \dot{\rm d} + \B {\rm d} ) - \la_2 \A {\rm d} \rt \dd \tau \\
    & - \int_0^t \lt \rho^\eps \cd \otimes \u : \na \dot{\rm d} \rt \dd \tau + \ka \int_0^t \lt \cd , \Delta {\rm d} \rt \dd \tau - \ka \int_0^t \lt \u \cdot \na \d , \Delta {\rm d} \rt \dd \tau \\
    & - \ka \int_0^t \lt {\rm u} \cdot \na {\rm d} , \Delta \d \rt \dd \tau - \int_0^t \lt \Gamma^\eps \d + \la_1 ( \cd + \B^\eps \d ) + \la_2 \A^\eps \d , \dot{\rm d} \rt \dd \tau  \,,
  \end{aligned}
\end{equation*}
which can be equivalently rewritten as
\begin{equation}\label{Rate-6}
  \begin{aligned}
    - \lt \sqrt{\rho^\eps} \cd , \dot{\rm d} \rt - \ka \lt \na \d , \na {\rm d} \rt =\,& \mathscr{R}^\eps_{\rm d} - \lt \rho^\eps_0 \widetilde{\rm d}^\eps_0 , \widetilde{\rm d}_0 \rt - \ka \lt \na \d_0 , \na {\rm d}_0 \rt - 2 \la_1 \int_0^t \lt \cd , \dot{\rm d} \rt \dd \tau \\
    & - \int_0^t \lt \cd , \la_1 \B {\rm d} + \la_2 \A {\rm d} \rt \dd \tau - \int_0^t \lt \dot{\rm d} , \la_1 \B^\eps \d + \la_2 \A^\eps \d \rt \dd \tau \\
    & - \ka \int_0^t \lt \u \cdot \na \d , \Delta {\rm d} \rt \dd \tau - \ka \int_0^t \lt {\rm u} \cdot \na {\rm d} , \Delta \d \rt \dd \tau \,,
  \end{aligned}
\end{equation}
where
\begin{equation}\label{Rd-eps}
  \begin{aligned}
    \mathscr{R}_{\rm d}^\eps =\, & \lt \sqrt{\rho^\eps} ( \sqrt{\rho^\eps} - 1 ) \u, {\rm u} \rt - \int_0^t \lt \rho^\eps \cd \otimes (  \u - {\rm u} ) : \na \dot{\rm d} \rt \dd \tau \\
    & - \int_0^t \lt (\rho^\eps - 1) \cd , \ka \Delta {\rm d} \rt \dd \tau - \int_0^t \lt \rho^\eps \cd , \Gamma {\rm d} \rt + \lt \dot{\rm d} , \Gamma^\eps \d \rt \dd \tau \\
    & - \int_0^t \lt (\rho^\eps - 1) \cd , \la_1 ( \dot{\rm d} + \B {\rm d} ) + \la_2 \A {\rm d} \rt \dd \tau \,.
  \end{aligned}
\end{equation}
We emphasize that every term in $\mathscr{R}_{\rm d}^\eps$ contains a difference, which will be easily controlled, excluding the term $ - \int_0^t \lt \rho^\eps \cd , \Gamma {\rm d} \rt + \lt \dot{\rm d} , \Gamma^\eps \d \rt \dd \tau $. However, it can be also transformed into a difference form by using the geometric constraints $|\d| = 1$ and $|{\rm d}| = 1$. The detailed derivations will be given later.

Noticing that
  \begin{align*}
    \lt \u , \dv (\na {\rm d} \odot \na {\rm d}) \rt =\, & \lt \u_i \p_i {\rm d} , \p_j \p_j {\rm d} \rt + \lt \u_i , \p_i \p_j {\rm d} \cdot \p_j {\rm d} \rt \\
    =\, & \lt \u \cdot \na {\rm d} , \Delta {\rm d} \rt - \lt (\rho^\eps - 1) \u, \na ( \tfrac{1}{2} |\na {\rm d}|^2 ) \rt + \lt \rho^\eps \u, \na ( \tfrac{1}{2} |\na {\rm d}|^2 ) \rt
  \end{align*}
and
\begin{equation*}
  \begin{aligned}
    \lt {\rm u}, \dv (\na \d \odot \na \d) \rt = \lt {\rm u} \cdot \na \d , \Delta \d \rt + \lt {\rm u}, \na (\tfrac{1}{2} |\na \d|^2) \rt = \lt {\rm u} \cdot \na \d , \Delta \d \rt \,,
  \end{aligned}
\end{equation*}
where the constraint condition  $\dv {\rm u} = 0$ is utilized, we can derive
\begin{equation}\label{Rate-Cancel}
  \begin{aligned}
    & \ka \int_0^t \lt \dv \big( \na \d \odot \na \d \big) , {\rm u} \rt \dd \tau + \ka \int_0^t \lt \dv (\na {\rm d} \odot \na {\rm d}) , \u \rt \dd \tau \\
    & - \ka \int_0^t \lt \u \cdot \na \d , \Delta {\rm d} \rt \dd \tau - \ka \int_0^t \lt {\rm u} \cdot \na {\rm d} , \Delta \d \rt \dd \tau \\
    =\, & \ka \int_0^t \lt {\rm u} \cdot \na (\d - {\rm d}) , \Delta (\d - {\rm d}) \rt - \lt (\u - {\rm u}) \cdot \na (\d - {\rm d}) , \Delta {\rm d} \rt \dd \tau \\
    & - \ka \int_0^t \lt (\rho^\eps - 1) \u , \na (\tfrac{1}{2} |\na {\rm d}|^2) \rt - \lt \rho^\eps \u , \na (\tfrac{1}{2} |\na {\rm d}|^2) \rt \dd \tau \,.
  \end{aligned}
\end{equation}
Summing up for \eqref{Rate-2}, \eqref{Rate-3}, \eqref{Rate-5}, \eqref{Rate-6} and utilizing the cancellation \eqref{Rate-Cancel} imply that
\begin{equation}\label{Rate-7}
  \begin{aligned}
    & \tfrac{1}{2} \| \sqrt{\rho^\eps} \u - {\rm u} \|^2_{L^2} + \tfrac{1}{2} \| \sqrt{\rho^\eps} \cd - \dot{\rm d} \|^2_{L^2} + \tfrac{1}{2} \ka \| \na \d - \na {\rm d} \|^2_{L^2} + \lt \Pi^\eps , 1 \rt \\
    & + \tfrac{1}{2} \mu_4 \int_0^t \| \na \u - \na {\rm u} \|^2_{L^2} \dd \tau + (\tfrac{1}{2} \mu_4 + \xi ) \int_0^t \| \dv \u \|^2_{L^2} \dd \tau \\
    =\, & \mathscr{C}_{\rm disp} + \mathscr{R}_{\rm u}^\eps + \mathscr{R}_{\rm d}^\eps + \mathscr{R}_{\rm sum}^\eps + \tfrac{1}{2} \| \sqrt{\rho^\eps_0} \u_0 - {\rm u}_0 \|^2_{L^2} + \tfrac{1}{2} \| \sqrt{\rho^\eps_0} \widetilde{\rm d}^\eps_0 - \widetilde{\rm d}_0 \|^2_{L^2} \\
    & + \tfrac{1}{2} \| \na \d_0 - \na {\rm d}_0 \|^2_{L^2} + \lt \Pi^\eps_0 , 1 \rt - \lt (\rho^\eps_0 - \sqrt{\rho^\eps_0}) \u_0 , {\rm u}_0 \rt - \lt (\rho^\eps_0 - \sqrt{\rho^\eps_0}) \widetilde{\rm d}^\eps_0 , \widetilde{\rm d}_0 \rt \,,
  \end{aligned}
\end{equation}
where $\mathscr{R}_{\rm u}^\eps$ and  $\mathscr{R}_{\rm d}^\eps$ are defined in \eqref{Ru-eps} and \eqref{Rd-eps}, respectively, and $\mathscr{R}_{\rm sum}^\eps$ is given as
\begin{equation}\label{Rsum-eps}
  \begin{aligned}
    \mathscr{R}_{\rm sum}^\eps =\, & \ka \int_0^t \lt {\rm u} \cdot \na (\d - {\rm d}) , \Delta (\d - {\rm d}) \rt - \lt (\u - {\rm u}) \cdot \na (\d - {\rm d}) , \Delta {\rm d} \rt \dd \tau \\
    & - \ka \int_0^t \lt (\rho^\eps - 1) \u , \na (\tfrac{1}{2} |\na {\rm d}|^2) \rt - \lt \rho^\eps \u , \na (\tfrac{1}{2} |\na {\rm d}|^2) \rt \dd \tau \\
    & + \int_0^t \lt \cd - \dot{\rm d} , \la_1 \B (\d - {\rm d}) + \la_2 \A (\d - {\rm d}) \rt \dd \tau \,,
  \end{aligned}
\end{equation}
which also contains some difference factors what we need in every term, and $\mathscr{C}_{\rm disp}$ reads
\begin{equation}\label{C-disp}
  \begin{aligned}
    \mathscr{C}_{\rm disp} =\, & \int_0^t \la_1 \| \cd - \dot{\rm d} \|^2_{L^2} \dd \tau + \int_0^t \lt \cd - \dot{\rm d} , \la_1 (\B^\eps - \B) \d + \la_2 (\A^\eps - \A) \d \rt \dd \tau \\
    & + \int_0^t \lt \u - {\rm u} , \dv \big( \Sigma_3^\eps - \tilde{\sigma}_{\bm{\mu}} ({\rm u}, {\rm d}, \dot{\rm d}) \big) \rt \dd \tau \,.
  \end{aligned}
\end{equation}

We notice that there is a difference form $\d - {\rm d}$ in $\mathscr{R}_{\rm sum}^\eps$. Thus, a norm $\| \d - {\rm d} \|^2_{L^2}$ is required in the left-hand side of the equality \eqref{Rate-7} to control the difference form $\d - {\rm d}$. Since $\p_t \d = \cd - \u \cdot \na \d$ and $\p_t {\rm d} = \dot{\rm d} - {\rm u} \cdot \na {\rm d}$, we have
\begin{equation*}
  \begin{aligned}
    \p_t (\d - {\rm d}) = (\cd - \dot{\rm d}) - ( \u \cdot \na \d - {\rm u} \cdot \na {\rm d} ) \,.
  \end{aligned}
\end{equation*}
Multiplying by $\d - {\rm d}$ in the previous equality and integrating over $[0,t] \times \R^n$, we obtain that
\begin{equation}\label{Rate-8-RL2-eps}
  \begin{aligned}
    \tfrac{1}{2} \| \d - {\rm d} \|^2_{L^2} = \tfrac{1}{2} \| \d_0 - {\rm d}_0 \|^2_{L^2} + \underbrace{ \int_0^t \lt (\cd - \dot{\rm d}) - ( \u \cdot \na \d - {\rm u} \cdot \na {\rm d} ) , \d - {\rm d} \rt \dd \tau }_{\mathscr{R}_{\rm L2}^\eps} \,.
  \end{aligned}
\end{equation}
By adding \eqref{Rate-8-RL2-eps} into the equality \eqref{Rate-7}, we deduce that
\begin{equation}\label{Rate-9}
  \begin{aligned}
    & \tfrac{1}{2} \| \sqrt{\rho^\eps} \u - {\rm u} \|^2_{L^2} + \tfrac{1}{2} \| \sqrt{\rho^\eps} \cd - \dot{\rm d} \|^2_{L^2} + \tfrac{1}{2} \ka \| \na \d - \na {\rm d} \|^2_{L^2} + \tfrac{1}{2} \| \d - {\rm d} \|^2_{L^2} \\
    & + \lt \Pi^\eps , 1 \rt + \tfrac{1}{2} \mu_4 \int_0^t \| \na \u - \na {\rm u} \|^2_{L^2} \dd \tau + (\tfrac{1}{2} \mu_4 + \xi ) \int_0^t \| \dv \u \|^2_{L^2} \dd \tau \\
    & = \mathscr{C}_{\rm disp} + \mathscr{R}_{\rm 0}^\eps + \mathscr{R}_{\rm u}^\eps + \mathscr{R}_{\rm d}^\eps + \mathscr{R}_{\rm sum}^\eps + \mathscr{R}_{\rm L2}^\eps \,,
  \end{aligned}
\end{equation}
where the term $\mathscr{R}_{\rm 0}^\eps$ is given as
\begin{equation}\label{R0-eps}
  \begin{aligned}
    \mathscr{R}_{\rm 0}^\eps = \,& \tfrac{1}{2} \| \sqrt{\rho^\eps_0} \u_0 - {\rm u}_0 \|^2_{L^2} + \tfrac{1}{2} \| \sqrt{\rho^\eps_0} \widetilde{\rm d}^\eps_0 - \widetilde{\rm d}_0 \|^2_{L^2} + \tfrac{1}{2} \| \na \d_0 - \na {\rm d}_0 \|^2_{L^2} + \lt \Pi^\eps_0 , 1 \rt \\
    & + \tfrac{1}{2} \| \d_0 - {\rm d}_0 \|^2_{L^2} - \lt (\rho^\eps_0 - \sqrt{\rho^\eps_0}) \u_0 , {\rm u}_0 \rt - \lt (\rho^\eps_0 - \sqrt{\rho^\eps_0}) \widetilde{\rm d}^\eps_0 , \widetilde{\rm d}_0 \rt \,.
  \end{aligned}
\end{equation}

Next, we will estimate the terms $\mathscr{C}_{\rm disp}$, $\mathscr{R}_{\rm 0}^\eps$, $\mathscr{R}_{\rm u}^\eps$, $\mathscr{R}_{\rm d}^\eps$, $\mathscr{R}_{\rm sum}^\eps$ and $\mathscr{R}_{\rm L2}^\eps$ in \eqref{Rate-9}. We first give the following three lemmas.

\begin{lem}\label{Lmm-rho-Rt}
	Under the same assumptions in Theorem \ref{Thm-Convergence-Rate}, we have
	\begin{equation*}
	  \begin{aligned}
	    \| \sqrt{\rho^\eps} - 1 \|_{L^2} \lesssim \eps \lt \Pi^\eps , 1 \rt^\frac{1}{2} \,, \quad  \| \sqrt{\rho^\eps_0} - 1 \|_{L^2} \lesssim \eps \lt \Pi^\eps_0 , 1 \rt^\frac{1}{2} \lesssim \eps^{1 + \tfrac{\alpha_0}{2}} \,.
	  \end{aligned}
	\end{equation*}
\end{lem}
\allowdisplaybreaks[4]
\begin{lem}\label{Lmm-Rmnd-Cntrl}
	Under the same assumptions in Theorem \ref{Thm-Convergence-Rate}, the quantities $\mathscr{R}_{\rm 0}^\eps$, $\mathscr{R}_{\rm u}^\eps$, $\mathscr{R}_{\rm d}^\eps$, $\mathscr{R}_{\rm sum}^\eps$, $ \mathscr{R}_{\rm L2}^\eps $ defined in \eqref{R0-eps}, \eqref{Ru-eps}, \eqref{Rd-eps}, \eqref{Rsum-eps}, \eqref{Rate-8-RL2-eps}, respectively, can be bounded as
	\begin{equation*}
	  \begin{aligned}
	    \mathscr{R}_{\rm 0}^\eps \lesssim \,& \eps^{\alpha_0} + \eps^{1 + \tfrac{\alpha_0}{2}} \,, \\
	    \mathscr{R}_{\rm u}^\eps \lesssim \,& (1 + T) \eps^2 + \eps^{1 + \tfrac{\alpha_0}{2}} + \eta_1 \lt \Pi^\eps , 1 \rt + \int_0^t \| \sqrt{\rho^\eps} \u - {\rm u} \|^2_{L^2} + \lt \Pi^\eps , 1 \rt \dd \tau \,, \\
	    \mathscr{R}_{\rm d}^\eps \lesssim\, & (1 + T) \eps^2 + \eps^{1+\tfrac{\alpha_0}{2}} + \eta_1 \lt \Pi^\eps , 1 \rt + \eta_1 \int_0^t \| \na \u - \na {\rm u} \|^2_{L^2} \dd \tau \\
	    & + \int_0^t \lt \Pi^\eps , 1 \rt + \| \sqrt{\rho^\eps} \cd - \dot{\rm d} \|^2_{L^2} + \| \na \d - \na {\rm d} \|^2_{L^2} + \| \d - {\rm d} \|^2_{L^2} \dd \tau \,, \\
	    \mathscr{R}_{\rm sum}^\eps \lesssim \,& (1+T) \eps^2 + \eps^{1+\tfrac{\alpha_0}{2}} + \eta_1 \lt \Pi^\eps , 1 \rt + \int_0^t \lt \Pi^\eps , 1 \rt + \| \sqrt{\rho^\eps} \u - {\rm u} \|^2_{L^2} \dd \tau \\
	    & + \int_0^t \| \sqrt{\rho^\eps} \cd - \dot{\rm d} \|^2_{L^2} + \| \na \d - \na {\rm d} \|^2_{L^2} + \| \d - {\rm d} \|^2_{L^2} \dd \tau \,, \\
	    \mathscr{R}_{\rm L2}^\eps \lesssim \,& \int_0^t \lt \Pi^\eps , 1 \rt + \| \sqrt{\rho^\eps} \u - {\rm u} \|^2_{L^2} + \| \sqrt{\rho^\eps} \cd - \dot{\rm d} \|^2_{L^2} \dd \tau \\
	    & + \int_0^t \| \na \d - \na {\rm d} \|^2_{L^2} + \| \d - {\rm d} \|^2_{L^2} \dd \tau \,,
	  \end{aligned}
	\end{equation*}
	for all $t \in [0,T]$ and $0 < \eps \leq 1$, where $T > 0$ is an arbitrary number and $\eta_1 > 0$ is small to be determined.
\end{lem}

\begin{lem}\label{Lmm-Contrl-C-disp}
	Under the same assumptions in Theorem \ref{Thm-Convergence-Rate}, the term $\mathscr{C}_{\rm disp}$ defined in \eqref{C-disp} can be rewritten as
	\begin{equation}\label{Cdisp-equ}
	  \begin{aligned}
	    \mathscr{C}_{\rm disp} = \,& -\mu_1 \int_0^t \| {\d}^\top({\A}^\eps-\A)\d \|^2_{L^2} {\dd} \tau +  \la_1 \int_0^t \big\| \cd - \dot{\rm d} + (\B^\eps - \B) \d + \tfrac{\la_2}{\la_1} (\A^\eps - \A) \d \big\|^2_{L^2} \dd \tau \\
	    & - ( \mu_5 + \mu_6 + \tfrac{\la_2^2}{\la_1} ) \int_0^t \| (\A^\eps - \A) \d \|^2_{L^2} \mathrm{d} \tau + \mathscr{R}_{\rm{\Sigma}}^\eps \,,
	  \end{aligned}
	\end{equation}
	where the term $\mathscr{R}_{\rm{\Sigma}}^\eps  = \int_0^t \lt \u - {\rm u} , \dv ( \widehat{\Sigma}_3^\eps - \tilde{\sigma}_{\bm{\mu}} ({\rm u}, {\rm d}, \dot{\rm d}) ) \rt \dd \tau $ and $ \widehat{\Sigma}_3^\eps $ is given  below in \eqref{Sigma-3-hat}. Moreover, the term $  \mathscr{R}_{\rm{\Sigma}}^\eps $ can be bounded by
	\begin{equation*}
	  \begin{aligned}
	    \mathscr{R}_{\rm{\Sigma}}^\eps \lesssim \eta_1 \int_0^t \| \na \u - \na {\rm u} \|^2_{L^2} \dd \tau + \int_0^t \| \d - {\rm d} \|^2_{L^2} \dd \tau
	  \end{aligned}
	\end{equation*}
	for all $t \in [0,T]$, in which $\eta_1 > 0$ is small to be determined and $T > 0$ is an any fixed number.
\end{lem}

Now, by employing Lemmas \ref{Lmm-rho-Rt}, \ref{Lmm-Rmnd-Cntrl} and \ref{Lmm-Contrl-C-disp} to the equality \eqref{Rate-9}, we deduce that for any fixed $T > 0$,

\begin{equation}\label{Rate-10}
  \begin{aligned}
    & \tfrac{1}{2} \| \sqrt{\rho^\eps} \u - {\rm u} \|^2_{L^2} + \tfrac{1}{2} \| \sqrt{\rho^\eps} \cd - \dot{\rm d} \|^2_{L^2} + \tfrac{1}{2} \ka \| \na \d - \na {\rm d} \|^2_{L^2} + \tfrac{1}{2} \| \d - {\rm d} \|^2_{L^2} \\
    & + \lt \Pi^\eps , 1 \rt + \tfrac{1}{2} \mu_4 \int_0^t \| \na \u - \na {\rm u} \|^2_{L^2} \dd \tau + (\tfrac{1}{2} \mu_4 + \xi ) \int_0^t \| \dv \u \|^2_{L^2} \dd \tau \\
    & + \mu_1 \int_0^t \| {\d}^\top({\A}^\eps-\A)\d \|^2_{L^2} {\dd} \tau \\
    & - \la_1 \int_0^t \big\| \cd - \dot{\rm d} + (\B^\eps - \B) \d + \tfrac{\la_2}{\la_1} (\A^\eps - \A) \d \big\|^2_{L^2} \dd \tau \\
    & + ( \mu_5 + \mu_6 + \tfrac{\la_2^2}{\la_1} ) \int_0^t \| (\A^\eps - \A) \d \|^2_{L^2} \dd \tau \\
    \lesssim\, & \eps^{\alpha_0} + \eps^{1+\tfrac{\alpha_0}{2}} + (1+T) \eps^2 + \eta_1 \lt \Pi^\eps , 1 \rt + \eta_1 \int_0^t \| \na \u - \na {\rm u} \|^2_{L^2} \dd \tau \\
    & + \int_0^t \lt \Pi^\eps ,  1 \rt + \| \sqrt{\rho^\eps} \u - {\rm u} \|^2_{L^2} + \| \sqrt{\rho^\eps} \cd - \dot{\rm d} \|^2_{L^2} + \| \na \d - \na {\rm d} \|^2_{L^2} + \| \d - {\rm d} \|^2_{L^2} \dd \tau
  \end{aligned}
\end{equation}
for all $t \in [0,T]$ and $0 < \eps \leq 1$, where $\eta_1 > 0$ is small to be determined. We then take $\eta_1 > 0$ is sufficiently small such that \eqref{Rate-10} reduces to
\begin{equation*}
  \begin{aligned}
    f (t) \leq f(t) + \int_0^t g (\tau) \dd \tau \leq C (1 + T) \eps^{\beta_0} + C \int_0^t \mathit{f} (\tau) \dd \tau
  \end{aligned}
\end{equation*}
for all $t \in [0,T]$, where $\beta_0 = \min\{ 2, \alpha_0, 1 + \tfrac{\alpha_0}{2} \} > 0$, $C > 0$ is an $\eps$-independent constant, the functional $f(t)$ is given as
\begin{equation*}
  \begin{aligned}
    f(t) = \lt \Pi^\eps ,  1 \rt + \| \sqrt{\rho^\eps} \u - {\rm u} \|^2_{L^2} + \| \sqrt{\rho^\eps} \cd - \dot{\rm d} \|^2_{L^2} + \| \na \d - \na {\rm d} \|^2_{L^2} + \| \d - {\rm d} \|^2_{L^2} \,,
  \end{aligned}
\end{equation*}
and $g(t) \geq 0$ reads
\begin{equation*}
  \begin{aligned}
    g(t)  =\, & \tfrac{1}{4} \mu_4 \int_0^t \| \na \u - \na {\rm u} \|^2_{L^2} \dd \tau + (\tfrac{1}{2} \mu_4 + \xi ) \int_0^t \| \dv \u \|^2_{L^2} \dd \tau \\
    & + \mu_1 \int_0^t \| {\d}^\top({\A}^\eps-\A)\d \|^2_{L^2} {\dd} \tau  
 - \la_1 \int_0^t \big\| \cd - \dot{\rm d} + (\B^\eps - \B) \d + \tfrac{\la_2}{\la_1} (\A^\eps - \A) \d \big\|^2_{L^2} \dd \tau \\
    & + ( \mu_5 + \mu_6 + \tfrac{\la_2^2}{\la_1} ) \int_0^t \| (\A^\eps - \A) \d \|^2_{L^2} \dd \tau \,.
  \end{aligned}
\end{equation*}
Then, the Gr\"onwall's inequality tells us that
\begin{equation*}
  \begin{aligned}
    f(t) \leq C (1 + T) \eps^{\beta_0} \exp (C t) \leq C (1 + T) \exp (CT) \eps^{\beta_0}
  \end{aligned}
\end{equation*}
for all $t \in [0,T]$, which concludes  Theorem \ref{Thm-Convergence-Rate}. \end{proof}

It remains to prove Lemmas \ref{Lmm-rho-Rt}, \ref{Lmm-Rmnd-Cntrl} and \ref{Lmm-Contrl-C-disp}.

\begin{proof}[Proof of Lemma \ref{Lmm-rho-Rt}]
	It is easy to justify that the following elementary inequalities
	\begin{align*}
	  & | \sqrt{x} - 1 |^2 \leq K_1 |x-1|^\gamma \,, \quad | x - 1 | \geq 1/2 \,, \;\ \  \gamma > 1 \,,\\
	  & | \sqrt{x} - 1 |^2 \leq K_1 | x - 1 |^2 \,, \quad x \geq 0 \,, \\
  	  & x^\gamma - \gamma ( x - 1 ) - 1 \geq
	  \left\{
	    \begin{array}{ll}
	      K_2 | x - 1 |^2 \,,\  \ & |x-1| \leq 1/2 \,,\ \  \; \gamma > 1 \,, \\
	      K_2 | x - 1 |^\gamma \,, \ \ & | x - 1 | \geq 1/2 \,,\  \ \; \gamma > 1 \,,
	    \end{array}
	  \right.
	\end{align*}
	hold for some positive constants $K_1, K_2$, depending only on $\gamma$. We thereby have
	\begin{equation*}
	  \begin{aligned}
	    \| \sqrt{\rho^\eps} - 1 \|^2_{L^2} =\, & \big\| ( \sqrt{\rho^\eps} - 1) \mathbf{1}_{\{|\rho^\eps - 1| \leq \tfrac{1}{2}\}} \big\|^2_{L^2} + \big\| ( \sqrt{\rho^\eps} - 1) \mathbf{1}_{\{|\rho^\eps - 1| \geq \tfrac{1}{2}\}} \big\|^2_{L^2} \\
	    \lesssim\, & \big\| (\rho^\eps - 1) \mathbf{1}_{\{|\rho^\eps - 1| \leq \tfrac{1}{2}\}} \big\|^2_{L^2} + \big\| (\rho^\eps - 1) \mathbf{1}_{\{|\rho^\eps - 1| \geq \tfrac{1}{2}\}} \big\|^\gamma_{L^\gamma} \\
	    \lesssim\, & \eps^2 \big\lt \tfrac{1}{\eps^2} \tfrac{ \tilde{a} }{\gamma - 1} \big( (\rho^\eps)^\gamma - \gamma (\rho^\eps - 1) - 1 \big) , 1 \big\rt = \eps^2 \lt \Pi^\eps, 1 \rt \,,
	  \end{aligned}
	\end{equation*}
	which means that
	\begin{align}\label{sro}
	  \| \sqrt{\ro} - 1 \|_{L^2} \lesssim \e \lt \Pi^\eps, 1 \rt^{\frac{1}{2}} \,.
	\end{align}
	Similarly, combining the initial conditions in Theorem \ref{Thm-Convergence-Rate}, we can easily prove that $\| \sqrt{\rho^\eps_0} - 1 \|_{L^2} \lesssim \eps \lt \Pi^\eps_0 , 1 \rt^\frac{1}{2} \lesssim \eps^{1 + \tfrac{\alpha_0}{2}}$. Thus the proof of Lemma \ref{Lmm-rho-Rt} is finished.
\end{proof}

\begin{proof}[Proof of Lemma \ref{Lmm-Rmnd-Cntrl}]
	We will justify this lemma by three steps: 1) to prove the bounds of $\mathscr{R}_{\rm 0}^\eps$, $\mathscr{R}_{\rm L2}^\eps$ and $\mathscr{R}_{\rm sum}^\eps$; 2) to prove the bound of $\mathscr{R}_{\rm u}^\eps$; 3) to prove the bound of $\mathscr{R}_{\rm d}^\eps$. We emphasize that, in what follows, the difference forms
	\begin{equation*}
	  \begin{aligned}
	    \sqrt{\rho^\eps} - 1 \,, \ \sqrt{\rho^\eps} \u - {\rm u} \,, \ \sqrt{\rho^\eps} \cd - \dot{\rm d} \,, \ \na \d - \na {\rm d} \,, \ \d - {\rm d}
	  \end{aligned}
	\end{equation*}
	are the most important terms, and the other terms without difference forms, like $\u$, ${\rm u}$, $\rho^\eps$, $\cd$, $\na \d$, $\dot{\rm d}$, $\na {\rm d}$, $\d$, ${\rm d}$, etc., can be bounded by the norms $\| \u \|_{H^s}$, $\| \u \|_{H^s}$, $\| \rho^\eps \|_{L^\infty}$, $\| \rho^\eps \|_{\dot{H}^s} = \eps \| \phi^\eps \|_{\dot{H}^s} \leq \| \phi^\eps \|_{H^s}$, $\| \cd \|_{H^s}$, $\| \na \d \|_{H^s}$, $\| \dot{\rm d} \|_{H^s}$, $\| \na {\rm d} \|_{H^s}$, and $|\d| = |{\rm d}| = 1$ via utilizing the Moser-type calculus inequalities in Lemma \ref{lem1}. From the uniform bounds \eqref{Unif-Bnd-1} and \eqref{Unif-Bnd-rho} in Theorem \ref{Thm-global} and the energy bound \eqref{Bnd-Limit} in Theorem \ref{Thm-Limit},  these norms will be bounded by some constants. Consequently, for simplicity, we will focus on the difference forms and control the other terms by some harmless constants in the estimates later. 

	\vspace*{1.2mm}
	
	 {\em Step 1. Estimates for $\mathscr{R}_{\rm 0}^\eps$, $\mathscr{R}_{\rm L2}^\eps$ and $\mathscr{R}_{\rm sum}^\eps$.}
		First, it is derived from Lemma \ref{Lmm-rho-Rt} and the Moser-type calculus inequalities in Lemma \ref{lem1} that
	\begin{equation*}
	  \begin{aligned}
	    & - \lt (\rho^\eps_0 - \sqrt{\rho^\eps_0}) \u_0 , {\rm u}_0 \rt - \lt (\rho^\eps_0 - \sqrt{\rho^\eps_0}) \widetilde{\rm d}^\eps_0 , \widetilde{\rm d}_0 \rt \\
	    &\qquad \lesssim \| \sqrt{\rho^\eps_0} - 1 \|_{L^2} \big( \| \u_0 \|_{H^s} \| {\rm u}_0 \|_{H^s} + \| \widetilde{\rm d}^\eps_0 \|_{H^s} \| \widetilde{\rm d}_0 \|_{H^s} \big) \lesssim \eps^{1+\tfrac{\alpha_0}{2}} \,,
	  \end{aligned}
	\end{equation*}
	which, combining with the initial conditions given in Theorem \ref{Thm-Convergence-Rate}, implies that
	\begin{equation*}
	  \begin{aligned}
	    \mathscr{R}_{\rm 0}^\eps \lesssim \eps^{\alpha_0} + \eps^{1+\tfrac{\alpha_0}{2}} \,.
	  \end{aligned}
	\end{equation*}
	
	Then, we estimate the term $\mathscr{R}_{\rm L2}^\eps$ defined in \eqref{Rate-8-RL2-eps}. Notice that
	\begin{equation*}
	  \begin{aligned}
	    \mathscr{R}_{\rm L2}^\eps =\, & \int_0^t \lt \sqrt{\rho^\eps} \cd - \dot{\rm d}, \d - {\rm d} \rt - \lt (\sqrt{\rho^\eps} - 1) \cd , \d - {\rm d} \rt + \lt \u \cdot \na (\d - {\rm d}) , \d - {\rm d} \rt \dd \tau \\
	    & + \int_0^t \lt (\sqrt{\rho^\eps} \u - {\rm u}) \cdot \na {\rm d} , \d - {\rm d} \rt - \lt (\sqrt{\rho^\eps} - 1) \u \cdot \na {\rm d} , \d - {\rm d} \rt \dd \tau \\
	    \lesssim\, & \int_0^t \big( \| \sqrt{\rho^\eps} - 1 \|_{L^2} + \| \sqrt{\rho^\eps} \u - {\rm u} \|_{L^2} + \| \sqrt{\rho^\eps} \cd - \dot{\rm d} \|_{L^2}   + \| \na \d - \na {\rm d} \|_{L^2} \big) \| \d - {\rm d} \|_{L^2} \dd \tau \\
	    \lesssim\, & \int_0^t \Big( \eps^2 \lt \Pi^\eps , 1 \rt + \| \sqrt{\rho^\eps} \u - {\rm u} \|_{L^2}^2 + \| \sqrt{\rho^\eps} \cd - \dot{\rm d} \|_{L^2}^2   + \| \na \d - \na {\rm d} \|_{L^2}^2 + \| \d - {\rm d} \|_{L^2}^2 \Big) \dd \tau \,,
	  \end{aligned}
	\end{equation*}
	where the last inequality is derived from the Young's inequality and Lemma \ref{Lmm-rho-Rt}.
	
	Finally, we estimate the term $\mathscr{R}_{\rm sum}^\eps$ given in \eqref{Rsum-eps}. For the term $\ka \int_0^t \lt {\rm u} \cdot \na (\d - {\rm d}) , \Delta (\d - {\rm d}) \rt \dd \tau$, we have
	\begin{equation}\label{Rsum-1}
	  \begin{aligned}
	    \ka \int_0^t \lt {\rm u} \cdot \na (\d - {\rm d}) , \Delta (\d - {\rm d}) \rt \dd \tau =\, & - \ka \int_0^t \lt \na {\rm u} \na (\d - {\rm d}) , \na (\d - {\rm d}) \rt \dd \tau \\
	    \lesssim \,& \int_0^t \| \na (\d - {\rm d}) \|^2_{L^2} \dd \tau \,.
	  \end{aligned}
	\end{equation}
	where the first equality is derived from the condition $\dv {\rm u} = 0$. For the term $\ka \int_0^t \lt (\u - {\rm u}) \cdot \na (\d - {\rm d}) , \Delta {\rm d} \rt \dd \tau$, we have
	\begin{equation}\label{Rsum-2}
	  \begin{aligned}
	    & \ka \int_0^t \lt (\u - {\rm u}) \cdot \na (\d - {\rm d}) , \Delta {\rm d} \rt \dd \tau \\
	    = \,& \ka \int_0^t \lt \big( (\sqrt{\rho^\eps}\u - {\rm u}) - (\sqrt{\rho^\eps} - 1) \u \big) \cdot \na (\d - {\rm d}) , \Delta {\rm d} \rt \dd \tau \\
	    \lesssim \,& \int_0^t \big( \| \sqrt{\rho^\eps} \u - {\rm u} \|_{L^2} + \| \sqrt{\rho^\eps} - 1 \|_{L^2} \big) \| \na (\d - {\rm d} ) \|_{L^2} \dd \tau \\
	    \lesssim \,& \int_0^t \eps^2 \lt \Pi^\eps , 1 \rt + \| \sqrt{\rho^\eps} \u - {\rm u} \|^2_{L^2} + \| \na \d - \na {\rm d} \|^2_{L^2} \dd \tau \,,
	  \end{aligned}
	\end{equation}
	where the last inequality is implied by Lemma \ref{Lmm-rho-Rt}. For the term $- \ka \int_0^t \lt (\rho^\eps - 1) \u , \na ( \tfrac{1}{2} |\na {\rm d}|^2 ) \rt \dd \tau$, we have
	\begin{equation}\label{Rsum-3}
	  \begin{aligned}
	   & - \ka \int_0^t \lt (\rho^\eps - 1) \u , \na ( \tfrac{1}{2} |\na {\rm d}|^2 ) \rt \dd \tau\\
 =\,&  - \ka \int_0^t \lt (\sqrt{\rho^\eps} - 1) (\sqrt{\rho^\eps} + 1) \u , \na ( \tfrac{1}{2} |\na {\rm d}|^2 ) \rt \dd \tau \\
	    \lesssim \, & \int_0^t \| \sqrt{\rho^\eps} - 1 \|_{L^2} \dd \tau \lesssim \int_0^t \eps \lt \Pi^\eps , 1 \rt^\frac{1}{2} \dd \tau \lesssim \eps^2 T + \int_0^t \lt \Pi^\eps , 1 \rt \dd \tau
	  \end{aligned}
	\end{equation}
	for all $t \in [0,T]$, where the last two inequalities is implied by Lemma \ref{Lmm-rho-Rt}. For the term $ \ka \int_0^t \lt \rho^\eps \u , \na (\tfrac{1}{2} |\na {\rm d}|^2) \rt \dd \tau$, we have
	\begin{equation}\label{Rsum-4}
	  \begin{aligned}
	    & \ka \int_0^t \lt \rho^\eps \u , \na (\tfrac{1}{2} |\na {\rm d}|^2) \rt \dd \tau\\
 =\, & - \ka \int_0^t \lt \dv (\rho^\eps \u ) , \tfrac{1}{2} |\na {\rm d}|^2 \rt \dd \tau \\
	    =\, & \ka \int_0^t \lt \p_t (\rho^\eps - 1) , \tfrac{1}{2} |\na {\rm d}|^2 \rt \dd \tau \\
	    = \,& \ka \lt (\rho^\eps - 1) , \tfrac{1}{2} |\na {\rm d}|^2 \rt - \ka \lt (\rho^\eps_0 - 1) , \tfrac{1}{2} |\na {\rm d}_0|^2 \rt - \ka \int_0^t \lt (\rho^\eps - 1), \na {\rm d} : \na \p_t {\rm d} \rt \dd \tau \\
	    \lesssim \,& \| \sqrt{\rho^\eps} - 1 \|_{L^2} + \| \sqrt{\rho^\eps_0} - 1 \|_{L^2} + \int_0^t \| \sqrt{\rho^\eps} - 1 \|_{L^2} \dd \tau \\
	    \lesssim \,& \eps \lt \Pi^\eps , 1 \rt^\frac{1}{2} + \eps^{1 + \tfrac{\alpha_0}{2}} + \int_0^t \eps \lt \Pi^\eps , 1 \rt^\frac{1}{2} \dd \tau \\
	    \lesssim \,& (1 + T) \eps^2 + \eps^{1 + \tfrac{\alpha_0}{2}} + \eta_1 \lt \Pi^\eps , 1 \rt + \int_0^t \lt \Pi^\eps , 1 \rt \dd \tau
	  \end{aligned}
	\end{equation}
	for all $t \in [0, T]$, where $\eta_1 > 0$ is small to be determined and the last second inequality is derived from Lemma \ref{Lmm-rho-Rt}. For the term $ \int_0^t \lt \cd - \dot{\rm d} , \la_1 \B (\d - {\rm d}) + \la_2 \A (\d - {\rm d}) \rt \dd \tau $, we have
	\begin{equation}\label{Rsum-5}
	  \begin{aligned}
	    & \int_0^t \lt \cd - \dot{\rm d} , \la_1 \B (\d - {\rm d}) + \la_2 \A (\d - {\rm d}) \rt \dd \tau \\
	    =\, & \int_0^t \lt (\sqrt{\rho^\eps} \cd - \dot{\rm d}) - (\sqrt{\rho^\eps} - 1) \cd , \la_1 \B (\d - {\rm d}) + \la_2 \A (\d - {\rm d}) \rt \dd \tau \\
	    \lesssim \,& \int_0^t \big( \| \sqrt{\rho^\eps} \cd -\dot{\rm d} \|_{L^2} + \| \sqrt{\rho^\eps} - 1 \|_{L^2} \big) \| \d - {\rm d} \|_{L^2} \dd \tau \\
	    \lesssim \,& \int_0^t \eps^2 \lt \Pi^\eps , 1 \rt + \| \sqrt{\rho^\eps} \cd -\dot{\rm d} \|_{L^2}^2 + \| \d - {\rm d} \|^2_{L^2} \dd \tau \,,
	  \end{aligned}
	\end{equation}
	where the last inequality is implied by the Young's inequality and Lemma \ref{Lmm-rho-Rt}. Consequently, the bounds \eqref{Rsum-1}-\eqref{Rsum-5} tell us that 
	\begin{equation*}
	  \begin{aligned}
	    \mathscr{R}_{\rm sum}^\eps \lesssim \,& (1+T) \eps^2 + \eps^{1+\tfrac{\alpha_0}{2}} + \eta_1 \lt \Pi^\eps , 1 \rt + \int_0^t \lt \Pi^\eps , 1 \rt + \| \sqrt{\rho^\eps} \u - {\rm u} \|^2_{L^2} \dd \tau \\
	    & + \int_0^t \| \sqrt{\rho^\eps} \cd - \dot{\rm d} \|^2_{L^2} + \| \na \d - \na {\rm d} \|^2_{L^2} + \| \d - {\rm d} \|^2_{L^2} \dd \tau \,.
	  \end{aligned}
	\end{equation*}
	
	\vspace*{1.2mm}
		{\em Step 2. Estimates for the term $\mathscr{R}_{\rm u}^\eps$ in \eqref{Ru-eps}.}
		First, we derive from Lemma \ref{Lmm-rho-Rt} that
	\begin{equation}\label{Ru-1}
	  \begin{aligned}
	    \lt (\rho^\eps - \sqrt{\rho^\eps}) \u, {\rm u} \rt \lesssim \| \sqrt{\rho^\eps} - 1 \|_{L^2} \lesssim \eps \lt \Pi^\eps , 1 \rt^\frac{1}{2} \lesssim \eps^2 + \eta_1 \lt \Pi^\eps , 1 \rt
	  \end{aligned}
	\end{equation}
	for a small $\eta_1 > 0$ to be determined. By the divergence free property of ${\rm u}$, one has
	\begin{equation*}
	  \begin{aligned}
	    & \int_0^t \lt \rho^\eps \u \otimes ( \u - {\rm u} ) : \na {\rm u} \rt \dd \tau \\
	    = \,&  \int_0^t \lt {\rm u} , \rho^\eps \u \cdot \na {\rm u} \rt \dd \tau + \int_0^t \lt \sqrt{\rho^\eps} \u - {\rm u} , ( \sqrt{\rho^\eps} \u - {\rm u} ) \cdot \na {\rm u} \rt \dd \tau \\
	    & - \int_0^t \lt {\rm u} \otimes (\sqrt{\rho^\eps} - 1) \sqrt{\rho^\eps} \u + (\sqrt{\rho^\eps} - 1) \sqrt{\rho^\eps} \u \otimes {\rm u} : \na {\rm u} \rt \dd \tau  \,.
	  \end{aligned}
	\end{equation*}
	Taking  the similar arguments to that in \eqref{Rsum-4}, we have
	\begin{equation*}
	  \begin{aligned}
	    \int_0^t \lt {\rm u} , \rho^\eps \u \cdot \na {\rm u} \rt \dd \tau \lesssim & (1 + T) \eps^2 + \eps^{1 + \tfrac{\alpha_0}{2}} + \eta_1 \lt \Pi^\eps , 1 \rt + \int_0^t \lt \Pi^\eps , 1 \rt \dd \tau
	  \end{aligned}
	\end{equation*}
	for small $\eta_1 > 0$ to be determined. Furthermore, we estimate
	\begin{equation*}
	  \begin{aligned}
	    & \int_0^t \lt \sqrt{\rho^\eps} \u - {\rm u} , ( \sqrt{\rho^\eps} \u - {\rm u} ) \cdot \na {\rm u} \rt \dd \tau \\
	    & \qquad   - \int_0^t \lt {\rm u} \otimes (\sqrt{\rho^\eps} - 1) \sqrt{\rho^\eps} \u + (\sqrt{\rho^\eps} - 1) \sqrt{\rho^\eps} \u \otimes {\rm u} : \na {\rm u} \rt \dd \tau \\
	    \lesssim\, & \int_0^t \| \sqrt{\rho^\eps} \u - {\rm u} \|^2_{L^2} + \| \sqrt{\rho^\eps} - 1 \|_{L^2} \dd \tau \lesssim T \eps^2 + \int_0^t \lt \Pi^\eps , 1 \rt + \| \sqrt{\rho^\eps} \u - {\rm u} \|^2_{L^2} \dd \tau \,.
	  \end{aligned}
	\end{equation*}
	Consequently, we have
	\begin{equation}\label{Ru-2}
	  \begin{aligned}
	    & \int_0^t \lt \rho^\eps \u \otimes ( \u - {\rm u} ) : \na {\rm u} \rt \dd \tau \\
	    & \qquad \lesssim    (1 + T) \eps^2 + \eps^{1 + \tfrac{\alpha_0}{2}} + \eta_1 \lt \Pi^\eps , 1 \rt + \int_0^t \lt \Pi^\eps , 1 \rt + \| \sqrt{\rho^\eps} \u - {\rm u} \|^2_{L^2} \dd \tau
	  \end{aligned}
	\end{equation}
	for all $t \in [0,T]$. Moreover, Lemma \ref{Lmm-rho-Rt} implies
	\begin{equation}\label{Ru-3}
	  \begin{aligned}
	    & - \int_0^t \lt \rho^\eps \u - \u , \tfrac{1}{2} \mu_4 \Delta {\rm u} \rt \dd \tau \\
=\, & - \int_0^t \lt ( \sqrt{\rho^\eps} - 1 ) ( \sqrt{\rho^\eps} + 1 ) \u , \tfrac{1}{2} \mu_4 \Delta {\rm u} \rt \dd \tau \\
	    \lesssim \,& \int_0^t \| \sqrt{\rho^\eps} - 1 \|_{L^2} \dd \tau \lesssim \int_0^t \eps \lt \Pi^\eps , 1 \rt^\frac{1}{2} \dd \tau \lesssim T \eps^2 + \int_0^t \lt \Pi^\eps , 1 \rt \dd \tau \,,
	  \end{aligned}
	\end{equation}
	and similarly
	\begin{equation}\label{Ru-4}
	  \begin{aligned}
	    \int_0^t \lt (\rho^\eps - 1) \u , \ka \dv (\na {\rm d} \odot \na {\rm d}) - \dv \tilde{\sigma}_{\bm{\mu}} ({\rm u}, {\rm d}, \dot{\rm d}) \rt \dd \tau \lesssim T \eps^2 + \int_0^t \lt \Pi^\eps , 1 \rt \dd \tau \,.
	  \end{aligned}
	\end{equation}
	It is also deduced from the analogous arguments  in \eqref{Rsum-4} that
	\begin{equation}\label{Ru-5}
	  \begin{aligned}
	    \int_0^t \lt \rho^\eps \u , \na \pi \rt \dd \tau \lesssim (1 + T) \eps^2 + \eps^{1 + \tfrac{\alpha_0}{2}} + \eta_1 \lt \Pi^\eps , 1 \rt + \int_0^t \lt \Pi^\eps , 1 \rt \dd \tau
	  \end{aligned}
	\end{equation}
	for all $t \in [0,T]$ and $\eta_1 > 0$ is small to be determined. As a result, by collecting the bounds \eqref{Ru-1}-\eqref{Ru-5} together, we deduce that
	\begin{equation*}
	  \begin{aligned}
	    \mathscr{R}_{\rm u}^\eps \lesssim (1 + T) \eps^2 + \eps^{1 + \tfrac{\alpha_0}{2}} + \eta_1 \lt \Pi^\eps , 1 \rt + \int_0^t \| \sqrt{\rho^\eps} \u - {\rm u} \|^2_{L^2} + \lt \Pi^\eps , 1 \rt \dd \tau \,.
	  \end{aligned}
	\end{equation*}
	
	\vspace*{1.2mm}
	
	 {\em Step 3. Estimates for the term $\mathscr{R}_{\rm d}^\eps$ in \eqref{Rd-eps}.}
	 	Taking the analogous estimates to that in \eqref{Ru-1} and \eqref{Ru-2}, we can derive that
	\begin{equation}\label{Rd-1}
	  \begin{aligned}
	    \lt \sqrt{\rho^\eps} (\sqrt{\rho^\eps} - 1) \u , {\rm u} \rt \lesssim \eps^2 + \eta_1 \lt \Pi^\eps , 1 \rt
	  \end{aligned}
	\end{equation}
	and
	\begin{equation}\label{Rd-2}
	  \begin{aligned}
	    & - \int_0^t \lt \rho^\eps \cd \otimes (\u - {\rm u}) : \na \dot{\rm d} \rt \dd \tau \\
	    &\qquad \lesssim (1 + T) \eps^2 + \eps^{1 + \tfrac{\alpha_0}{2}} + \eta_1 \lt \Pi^\eps , 1 \rt + \int_0^t \lt \Pi^\eps , 1 \rt + \| \sqrt{\rho^\eps} \u - {\rm u} \|^2_{L^2} + \| \sqrt{\rho^\eps} \cd - \dot{\rm d} \|^2_{L^2} \dd \tau
	  \end{aligned}
	\end{equation}
	for all $t \in [0,T]$, where $\eta_1 > 0$ is small to be determined. For the term $- \int_0^t \lt (\rho^\eps - 1) \cd , \ka \Delta {\rm d} \rt \dd \tau$, we have
	\begin{equation}\label{Rd-3}
	  \begin{aligned}
	    - \int_0^t \lt (\rho^\eps - 1) \cd , \ka \Delta {\rm d} \rt \dd \tau =\, & - \int_0^t \lt (\sqrt{\rho^\eps} - 1) (\sqrt{\rho^\eps} + 1) \cd , \ka \Delta {\rm d} \rt \dd \tau \\
	    \lesssim \,& \int_0^t \| \sqrt{\rho^\eps} - 1 \|_{L^2} \dd \tau \lesssim \int_0^t \eps \lt \Pi^\eps , 1 \rt^\frac{1}{2} \dd \tau \lesssim T \eps^2 + \int_0^t \lt \Pi^\eps , 1 \rt \dd \tau
	  \end{aligned}
	\end{equation}
	for all $t \in [0,T]$. Similarly, we also have
	\begin{equation}\label{Rd-4}
	  \begin{aligned}
	    - \int_0^t \lt (\rho^\eps - 1) \cd , \la_1 (\dot{\rm d} + \B {\rm d}) + \la_2 \A {\rm d} \rt \dd \tau \lesssim T \eps^2 + \int_0^t \lt \Pi^\eps , 1 \rt \dd \tau  \,.
	  \end{aligned}
	\end{equation}
	
	It remains to control the term $- \int_0^t \lt \rho^\eps \cd , \Gamma {\rm d} \rt + \lt \dot{\rm d} , \Gamma^\eps \d \rt \dd \tau$. It is easy to know that
	\begin{equation*}
	  \begin{aligned}
	    & - \int_0^t \lt \rho^\eps \cd , \Gamma {\rm d} \rt + \lt \dot{\rm d} , \Gamma^\eps \d \rt \dd \tau \\
	    =\, & - \int_0^t \lt (\rho^\eps - 1) \cd , \Gamma {\rm d} \rt \dd \tau - \int_0^t \lt \cd , \Gamma {\rm d} \rt + \lt \dot{\rm d} , \Gamma^\eps \d \rt \dd \tau \,.
	  \end{aligned}
	\end{equation*}
	By the similar arguments in \eqref{Rd-3}, one has
	\begin{equation}\label{Rd-52}
	  \begin{aligned}
	    - \int_0^t \lt (\rho^\eps - 1) \cd , \Gamma {\rm d} \rt \dd \tau \lesssim T \eps^2 + \int_0^t \lt \Pi^\eps , 1 \rt \dd \tau
	  \end{aligned}
	\end{equation}
	for all $t \in [0, T]$. Since $|\d| = |{\rm d}| = 1$, we have $\cd \cdot \d = \dot{\rm d} \cdot {\rm d} = 0$. Thus, we can calculate that
	\begin{equation*}
	  \begin{aligned}
	    - \int_0^t \lt \cd , \Gamma {\rm d} \rt + \lt \dot{\rm d} , \Gamma^\eps \d \rt \dd \tau = \int_0^t \lt \cd - \dot{\rm d} , \Gamma^\eps (\d - {\rm d}) \rt \dd \tau + \int_0^t \lt \cd - \dot{\rm d} , ( \Gamma^\eps - \Gamma ) {\rm d} \rt \dd \tau \,.
	  \end{aligned}
	\end{equation*}
	For the quantity $\int_0^t \lt \cd - \dot{\rm d} , \Gamma^\eps (\d - {\rm d}) \rt \dd \tau$, we have
	\begin{equation}\label{Rd-54}
	  \begin{aligned}
	    \int_0^t \lt \cd - \dot{\rm d} , \Gamma^\eps (\d - {\rm d}) \rt \dd \tau = \,& \int_0^t \lt (\sqrt{\rho^\eps} \cd - \dot{\rm d}) - (\sqrt{\rho^\eps} - 1) \cd , \Gamma^\eps (\d - {\rm d}) \rt \dd \tau \\
	    \lesssim \,& \int_0^t ( \| \sqrt{\rho^\eps} \cd - \dot{\rm d} \|_{L^2} + \| \sqrt{\rho^\eps} - 1 \|_{L^2} ) \| \d - {\rm d} \|_{L^2} \dd \tau \\
	    \lesssim \,& \int_0^t \eps^2 \lt \Pi^\eps, 1 \rt + \| \sqrt{\rho^\eps} \cd - \dot{\rm d} \|^2_{L^2} + \| \d - {\rm d} \|^2_{L^2} \dd \tau \,,
	  \end{aligned}
	\end{equation}
	where the last inequality is derived from the Young's inequality and Lemma \ref{Lmm-rho-Rt}. Furthermore, we similarly have
	\begin{equation}\label{Rd-55}
	  \begin{aligned}
	    \int_0^t \lt \cd - \dot{\rm d} , ( \Gamma^\eps - \Gamma ) {\rm d} \rt \dd \tau \lesssim \int_0^t ( \| \sqrt{\rho^\eps} \cd - \dot{\rm d} \|_{L^2} + \eps \lt \Pi^\eps , 1 \rt^\frac{1}{2} ) \| \Gamma^\eps - \Gamma \|_{L^2} \dd \tau \,.
	  \end{aligned}
	\end{equation}
	Recalling the definitions of $\Gamma^\eps$ and $\Gamma$ in \eqref{Ga-eps} and \eqref{Gal}, respectively, one easily has
	\begin{equation}\label{Rd-56}
	  \begin{aligned}
	    \| \Gamma^\eps - \Gamma \|_{L^2} \lesssim \,& \| (\sqrt{\rho^\eps} \cd - \dot{{\rm d}}) \cdot (\sqrt{\rho^\eps} \cd + \dot{{\rm d}}) \|_{L^2} + \| \na (\d - {\rm d}) : \na (\d + {\rm d}) \|_{L^2} \\
	    & + \| {\d}^\top (\A^\eps - \A) \d \|_{L^2} + \| (\d - {\rm d})^\top \A \d \|_{L^2} + \| {\rm d}^\top \A (\d - {\rm d}) \|_{L^2} \\
	    \lesssim\, & \| \sqrt{\rho^\eps} \cd - \dot{\rm d} \|_{L^2} + \| \na \d - \na {\rm d} \|_{L^2} + \| \d - {\rm d} \|_{L^2} + \| \na \u - \na {\rm u} \|_{L^2} \,.
	  \end{aligned}
	\end{equation}
	By substituting the bound \eqref{Rd-56} into \eqref{Rd-55} and using the Young's inequality, we deduce that
	\begin{equation}\label{Rd-57}
	  \begin{aligned}
	    \int_0^t \lt \cd - \dot{\rm d} , ( \Gamma^\eps - \Gamma ) \d \rt \dd \tau \lesssim\, & \eta_1 \int_0^t \| \na \u - \na {\rm u} \|^2_{L^2} \dd \tau + \int_0^t \eps^2 \lt \Pi^\eps , 1 \rt \dd \tau \\
	    & + \int_0^t \| \sqrt{\rho^\eps} \cd - \dot{\rm d} \|^2_{L^2} + \| \na \d - \na {\rm d} \|^2_{L^2} + \| \d - {\rm d} \|^2_{L^2} \dd \tau \,,
	  \end{aligned}
	\end{equation}
	where $\eta_1 > 0$ is small to be determined. Collecting the bounds \eqref{Rd-52}, \eqref{Rd-54} and \eqref{Rd-57}
together, we get
	\begin{equation}\label{Rd-5}
	  \begin{aligned}
	    - \int_0^t \lt \rho^\eps \cd , \Gamma {\rm d} \rt + \lt \dot{\rm d} , \Gamma^\eps & \d \rt \dd \tau \lesssim T \eps^2 + \eta_1 \int_0^t \| \na \u - \na {\rm u} \|^2_{L^2} \dd \tau + \int_0^t \eps^2 \lt \Pi^\eps , 1 \rt \dd \tau \\
	    & + \int_0^t \| \sqrt{\rho^\eps} \cd - \dot{\rm d} \|^2_{L^2} + \| \na \d - \na {\rm d} \|^2_{L^2} + \| \d - {\rm d} \|^2_{L^2} \dd \tau \,.
	  \end{aligned}
	\end{equation}
	Finally, putting  the inequalities \eqref{Rd-1}-\eqref{Rd-4} and \eqref{Rd-5} together, we have
	\begin{equation*}
	  \begin{aligned}
	    \mathscr{R}_{\rm d}^\eps \lesssim \,& (1 + T) \eps^2 + \eps^{1+\tfrac{\alpha_0}{2}} + \eta_1 \lt \Pi^\eps , 1 \rt + \eta_1 \int_0^t \| \na \u - \na {\rm u} \|^2_{L^2} \dd \tau \\
	    & + \int_0^t \lt \Pi^\eps , 1 \rt + \| \sqrt{\rho^\eps} \cd - \dot{\rm d} \|^2_{L^2} + \| \na \d - \na {\rm d} \|^2_{L^2} + \| \d - {\rm d} \|^2_{L^2} \dd \tau \,.
	  \end{aligned}
	\end{equation*}
	Thus the proof of Lemma \ref{Lmm-Rmnd-Cntrl} is completed.
\end{proof}

	\vspace*{1.2mm}
	
\begin{proof}[Proof of Lemma \ref{Lmm-Contrl-C-disp}]
	We first introduce a tensor
	\begin{equation}\label{Sigma-3-hat}
	  \begin{aligned}
	    \widehat{\Sigma}_3^\eps =\, & \mu_1 ( {\d}^\top \A \d ) \d \otimes \d + \mu_2 \d \otimes ( \dot{\rm d} + \B \d ) + \mu_3 ( \dot{\rm d} + \B \d ) \otimes \d \\
	    & + \mu_5 \d \otimes (\A \d) + \mu_6 (\A \d ) \otimes \d \,.
	  \end{aligned}
	\end{equation}
	By taking  the same derivations in Section 2 of \cite{JL}, we can infer that
	\begin{equation*}
	  \begin{aligned}
	    & \int_0^t \la_1 \| \cd - \dot{\rm d} \|^2_{L^2} + \lt \cd - \dot{\rm d} , \la_1 (\B^\eps - \B) \d + \la_2 (\A^\eps - \A) \d \rt \dd \tau \\
	    & + \int_0^t \lt \u - {\rm u} , \dv (\Sigma_3^\eps - \widehat{\Sigma}_3^\eps) \rt \dd \tau \\
	    = & - \mu_1 \int_0^t \| {\d}^\top({\A}^\eps-\A)\d \|^2_{L^2} {\dd} \tau \\
 &+ \la_1 \int_0^t \big\| \cd - \dot{\rm d} + (\B^\eps - \B) \d + \tfrac{\la_2}{\la_1} (\A^\eps - \A) \d \big\|^2_{L^2} \dd \tau \\
	    & - (\mu_5 + \mu_6 + \tfrac{\la_2^2}{\la_1}) \int_0^t \| (\A^\eps - \A) \d \|^2_{L^2} \dd \tau \,.
	  \end{aligned}
	\end{equation*}
	Denote by $ \mathscr{R}_{\rm{\Sigma}}^\eps = \int_0^t \lt \u - {\rm u} , \dv ( \widehat{\Sigma}_3^\eps - \tilde{\sigma}_{\bm{\mu}} ({\rm u}, {\rm d}, \dot{\rm d}) ) \rt \dd \tau $.
Then, by the definition of $\mathscr{C}_{\rm disp}$ in \eqref{C-disp}, we know that the equality \eqref{Cdisp-equ} holds. We only need to estimate the term $ \mathscr{R}_{\rm{\Sigma}}^\eps $.
	
	First, we have
	\begin{equation}\label{R-Sigma-1}
	  \begin{aligned}
	    \mathscr{R}_{\rm{\Sigma}}^\eps = & - \int_0^t \lt \na \u - \na {\rm u} : \widehat{\Sigma}_3^\eps - \tilde{\sigma}_{\bm{\mu}} ({\rm u}, {\rm d}, \dot{\rm d}) \rt \dd \tau \\
	    & \lesssim \int_0^t \| \na \u - \na {\rm u} \|_{L^2} \| \widehat{\Sigma}_3^\eps - \tilde{\sigma}_{\bm{\mu}} ({\rm u}, {\rm d}, \dot{\rm d}) \|_{L^2} \dd \tau \,.
	  \end{aligned}
	\end{equation}
	Noticing that
	\begin{equation*}
	  \begin{aligned}
	    & \mu_1 ({\d}^\top \A \d) \d \otimes \d - \mu_1 ( {\rm d}^\top \A {\rm d} ) {\rm d} \otimes {\rm d} \\
	    = \,& \mu_1 ( (\d - {\rm d})^\top \A \d ) \d \otimes \d + ({\rm d}^\top \A (\d - {\rm d})) \d \otimes \d \\
	    & + \mu_1 ({\rm d}^\top \A {\rm d}) (\d - {\rm d}) \otimes \d + \mu_1 ({\rm d}^\top \A {\rm d}) {\rm d} \otimes ( \d - {\rm d} ) \,,
	  \end{aligned}
	\end{equation*}
	we bound it as
	\begin{equation*}
	  \begin{aligned}
	    \| \mu_1 ({\d}^\top \A \d) \d \otimes \d - \mu_1 ( {\rm d}^\top \A {\rm d} ) {\rm d} \otimes {\rm d} \|_{L^2} \lesssim \| \d - {\rm d} \|_{L^2} \,.
	  \end{aligned}
	\end{equation*}
	Similarly, we have
	\begin{equation*}
	  \begin{aligned}
	   \| \mu_2 \d \otimes ( \dot{\rm d} + \B \d ) - \mu_2 {\rm d} \otimes ( \dot{\rm d} + \B {\rm d} ) \|_{L^2} \lesssim \,& \| \d - {\rm d} \|_{L^2} \,, \\
	    \| \mu_3 ( \dot{\rm d} + \B \d ) \otimes \d - \mu_3 ( \dot{\rm d} + \B {\rm d} ) \otimes {\rm d} \|_{L^2} \lesssim \,&\| \d - {\rm d} \|_{L^2} \,, \\
	    \| \mu_5 \d \otimes (\A \d) - \mu_5 {\rm d} \otimes (\A {\rm d}) \|_{L^2} \lesssim \,& \| \d - {\rm d} \|_{L^2} \,, \\
	      \| \mu_6 (\A \d) \otimes \d - \mu_6 (\A {\rm d}) \otimes {\rm d} \|_{L^2} \lesssim\, & \| \d - {\rm d} \|_{L^2} \,.
	  \end{aligned}
	\end{equation*}
	Then, we obtain that
	\begin{equation}\label{R-Sigma-2}
	  \begin{aligned}
	    \| \widehat{\Sigma}_3^\eps - \tilde{\sigma}_{\bm{\mu}} ({\rm u}, {\rm d}, \dot{\rm d}) \|_{L^2} \lesssim \| \d - {\rm d} \|_{L^2} \,.
	  \end{aligned}
	\end{equation}
	As a consequence, plugging the bound \eqref{R-Sigma-2} into \eqref{R-Sigma-1} and employing the Young's inequality, we infer that
	\begin{equation*}
	  \begin{aligned}
	    \mathscr{R}_{\rm{\Sigma}}^\eps \lesssim \eta_1 \int_0^t \| \na \u - \na {\rm u} \|^2_{L^2} \dd \tau + \int_0^t \| \d - {\rm d} \|^2_{L^2} \dd \tau
	  \end{aligned}
	\end{equation*}
	for small $\eta_1 > 0$ to be determined. Then the proof of Lemma \ref{Lmm-Contrl-C-disp} is finished.	
\end{proof}

\section*{Acknowledgments}

Liang Guo is supported by Doctoral Research Foundation of  Henan University (Grant No. CX3050A0970029).
Ning Jiang is supported by NSFC (Grant Nos. 11471181 and 11731008).
 Fucai Li is supported in part by NSFC (Grant Nos.  12071212 and 11971234) and a project funded by the priority academic program development of Jiangsu higher education institutions.
Y.-L. Luo is supported by grants from the NSFC (Grant No. 12201220), the Guang Dong Basic and Applied Basic Research Foundation (Grant No. 2021A1515110210), and the Science and Technology Program of Guangzhou (Grant No.  202201010497.

\vskip 5mm
\textbf{Conflict of Interest: The authors declare that they have no conflict of interest.}

\vskip5mm

\textbf{Data Availability Statements}:
\textbf{Data sharing not applicable to this article as no datasets were generated or analysed during the current study.
}
\vskip 5mm

\bibliographystyle{plain}

\end{document}